%% file: sliqn_arxiv.tex
\documentclass{article}
\usepackage[preprint]{neurips_2023}%
\usepackage{enumitem}
\usepackage{hyperref}       % hyperlinks
\usepackage{url}            % simple URL typesetting
\usepackage{booktabs}       % professional-quality tables
\usepackage{amsfonts}       % blackboard math symbols
\usepackage{nicefrac}       % compact symbols for 1/2, etc.
\usepackage{microtype}      % microtypography
\usepackage{mathtools}
\usepackage{graphicx}
\usepackage{textcomp}
\usepackage{xcolor}
\usepackage{hyperref}
\usepackage{amsbsy}
\usepackage{amssymb}
\usepackage{algorithmic}
\usepackage[]{algorithm}
\usepackage{bm}
\usepackage{subcaption}
\usepackage[toc,page,header]{appendix}
\usepackage{minitoc}

\hypersetup{
	colorlinks=true,
	linkcolor=blue,
	citecolor=blue,
	urlcolor=blue
}

\input{command_preamble.tex}
\author{
	Aakash Lahoti$^*$ \\
	Carnegie Mellon University \\
	\texttt{alahoti@andrew.cmu.edu}
	\And
	Spandan Senapati$^*$ \\
	University of Southern California \\
	\texttt{ssenapat@usc.edu}
	\AND
	\hspace{-12.5mm}Ketan Rajawat \\
		\hspace{-12.5mm}Indian Institute of Technology Kanpur \\
		\hspace{-12.5mm}\texttt{ketan@iitk.ac.in} 
	\And
		\hspace{-10mm}Alec Koppel \\
	\hspace{-10mm}J.P. Morgan AI Research \\
		\hspace{-10mm}\texttt{alec.koppel@jpmchase.com}
}	
 \title{Sharpened Lazy Incremental Quasi-Newton Method}

\begin{document}
\maketitle
\doparttoc % Tell to minitoc to generate a toc for the parts
\faketableofcontents % Run a fake tableofcontents command for the partocs

% \part{} % Start the document part
% \parttoc % Insert the document TOC


% \runningauthor{Lahoti, Senapati, Rajawat, Koppel}

% \twocolumn[

% \aistatstitle{Sharpened Lazy Incremental Quasi-Newton Method}

% \aistatsauthor{
%     Aakash Sunil Lahoti$^{*}$ \And 
%     Spandan Senapati$^{*}$
% }

% \aistatsaddress{
%     Carnegie Mellon University  \And  
%     University of Southern California 
% } 

% \aistatsauthor{ 
%     Ketan Rajawat \And 
%     Alec Koppel }

% \aistatsaddress{
%     Indian Institute of Technology Kanpur  \And 
%     J.P. Morgan AI Research
% }
% ]

\begin{abstract}
\input{content/abstract}
\end{abstract}

\section{INTRODUCTION}

\def\thefootnote{*}\footnotetext{Equal Contribution. Work done when AL and SS were at the Indian Institute of Technology Kanpur.}
\def\thefootnote{1}\footnotetext{Accepted at the 27th International Conference on Artificial Intelligence and Statistics (AISTATS) 2024.}
% \footnotetext{Equal Contribution} 
\input{content/introduction}

\section{RELATED WORK}
\input{content/related_work}

\section{NOTATION AND PRELIMINARIES}\label{sec:overview}
\input{content/notations_and_preliminaries}

\section{PROPOSED ALGORITHM}\label{sec:HybridIQN_algorithm}
\input{content/proposed_algorithm}

\section{THEORETICAL ANALYSIS OF ALGORITHM \ref{alg:eff_h_iqn}}\label{convergence}
\input{content/theoretical_analysis}

\section{NUMERICAL EXPERIMENTS}\label{sec:experiments}
\input{content/numerical_experiments}

\section{CONCLUSION AND FUTURE WORK}
\input{content/conclusion}

\bibliography{references}

\newpage

\appendix
\addcontentsline{toc}{section}{Appendix} % Add the appendix text to the document TOC
\part{Appendix} % Start the appendix part
\parttoc % Insert the appendix TOC

% \tableofcontents

\section{ESTABLISHED RESULTS}\label{sec:prelim_lemma}
\input{appendices/restated_lemmas}

\section{SUPPORTING LEMMAS}\label{app:supporting_lemma}
\input{appendices/supporting_lemmas}

\section{EFFICIENT IMPLEMENTATION OF IQN} \label{app:iqn_efficient_implementation}
\input{appendices/efficient_implementation}

% \newpage

\section{LOW MEMORY IMPLEMENTATION OF INCREMENTAL METHODS}
\label{app:low_memory_implementation}
\input{appendices/low_memory_implementation}

% \newpage

\section{CONVERGENCE ANALYSIS OF SIQN}
\label{app:convergence_iqn}
\input{appendices/convergence_analysis_of_SIQN}

% \newpage

\section{EFFICIENT IMPLEMENTATION OF SLIQN}\label{app:sliqn}
\input{appendices/efficient_implementation_sliqn}

\section{CONVERGENCE ANALYSIS OF SLIQN}\label{sec:convergence_sliqn}
\input{appendices/convergence_analysis_sliqn}

\section{GENERALIZED SHARPENED INCREMENTAL QUASI-NEWTON METHOD (G-SLIQN)}\label{sec:restricted_broyden}
\label{app:broyden_extension}
\input{appendices/broyden_class_extension}

% \newpage
\section{NUMERICAL SIMULATIONS}\label{sec:numerical_simulations}
\input{appendices/additional_numerical_simulations}
\end{document}

%% file: command_preamble.tex
\providecommand{\abs}[1]{\lvert#1\rvert}% abs operator
\renewcommand{\b}[1]{\ensuremath{\mathbf{#1}}} % bold letters
\newcommand{\bi}[1]{\ensuremath{\bm{#1}}} % bold symbols
\renewcommand{\c}[1]{\ensuremath{\mathcal{#1}}} % caligraphic font
\newcommand{\norm}[1]{\ensuremath{\left\|#1\right\|}} % norm operator
\newcommand{\norme}[1]{\ensuremath{\|#1\|}} % norm operator
\def\ceil#1{\lceil #1 \rceil}
\def\floor#1{\lfloor #1 \rfloor}
\providecommand{\ip}[2]{\langle #1, #2 \rangle} % inner product operator
\newcommand{\Br}[3]{\mathrm{Broyd}_{#1}^{#3}(#2)}
\newcommand{\bfgs}[1]{\mathrm{BFGS}(#1)}
\newcommand{\dfp}[1]{\mathrm{DFP}(#1)}

\DeclareMathOperator*{\argmax}{arg\,max}
\DeclareMathOperator*{\argmin}{arg\,min}

\def \iqn {{\textit{siqn}}}
\def \x {{\bi{x}}}
\def \y {{\bi{y}}}
\def \b {{\bi{b}}}
\def \a {{\bi{a}}}
\def \z {{\bi{z}}}
\def \s {{\bi{s}}}
\def \e {{\bi{e}}}
\def \w {{\bi{w}}}

\def \g {{\bi{g}}}
\def \u {{\bi{u}}}
\def \v {{\bi{v}}}
\def \bp{{\bi{\psi}}}

\def \A {{\bi{A}}}
\def \B {{\bi{B}}}

\def \H {{\bi{H}}}
\def \I {{\bi{I}}}

\def \G {{\bi{G}}}

\def \P {{\bi{P}}}
\def \Q {{\bi{Q}}}

\def \D {{\bi{D}}}

\def \K {{\bi{K}}}

\def \X {{\bi{X}}}
\def \Y {{\bi{Y}}}
\def \Z {{\bi{Z}}}

\def \T {{\mathsf{T}}}

\def \O {{\c{O}}}

\def \Rn {{\mathbb{R}}}

\def \ub {{\bar{\u}}}

\def\cO{\mathcal{O}}

\def \zero {{\mathbf{0}}}
\def\nn{\nonumber}

\def \bB {{\bar{\B}}}

\newcommand{\new}{\marginpar{NEW}}

\newtheorem{assumption}{}

\newtheorem{rem}{\bf Remark}

\newtheorem{thm}{Theorem}
\newtheorem{theorem}{Theorem}
\counterwithin{theorem}{section}

\newtheorem{lem}{Lemma}
\newtheorem{lemma}{Lemma}
\counterwithin{lemma}{section}

\newenvironment{proof}{\paragraph{Proof:}}{\hfill$\square$}
\newtheorem{proposition}[theorem]{Proposition}
\newtheorem{corollary}[theorem]{Corollary}

\def \Lem {{\textit{Lem}}}
\def \new {{\textit{new}}}

%% file: content/abstract.tex
The problem of minimizing the sum of \(n\) functions in \(d\) dimensions is ubiquitous in machine learning and statistics.
In many applications where the number of observations \(n\) is large, it is necessary to use incremental or stochastic methods, as their per-iteration cost is independent of \(n\).
Of these, Quasi-Newton (QN) methods strike a balance between the per-iteration cost and the convergence rate. 
Specifically, they exhibit a superlinear rate with \(\mathcal{O}(d^2)\) cost in contrast to the linear rate of first-order methods with \(\mathcal{O}(d)\) cost and the quadratic rate of second-order methods with \(\mathcal{O}(d^3)\) cost.~However, existing incremental methods have notable shortcomings: Incremental Quasi-Newton (IQN) only exhibits asymptotic superlinear convergence.
In contrast, Incremental Greedy BFGS (IGS) offers explicit superlinear convergence but suffers from poor empirical performance and has a per-iteration cost of $\mathcal{O}(d^3)$.
To address these issues, we introduce the Sharpened Lazy Incremental Quasi-Newton Method (SLIQN) that achieves the best of both worlds: an explicit superlinear convergence rate, and superior empirical performance at a per-iteration $\mathcal{O}(d^2)$ cost. 
SLIQN features two key changes: first, it incorporates a hybrid strategy of using both classic and greedy BFGS updates, allowing it to empirically outperform both IQN and IGS.
Second, it employs a clever constant multiplicative factor along with a lazy propagation strategy, which enables it to have a cost of $\mathcal{O}(d^2)$.
Additionally, our experiments demonstrate the superiority of SLIQN over other incremental and stochastic Quasi-Newton variants and establish its competitiveness with second-order incremental methods.

%% file: content/introduction.tex
We consider the finite sum minimization problem,
\begin{align}
	\label{eq:main_prob}\tag{$\mathcal{P}$}
	\x^\star = \arg\min_{\x \in \Rn^d} \frac{1}{n} \sum_{i=1}^{n} f_i(\x),
\end{align}	
where each $f_i$ is $\mu$-strongly convex, $L$-smooth and has a Lipschitz continuous Hessian.
The canonical example of \eqref{eq:main_prob} is the empirical risk minimization problem in supervised learning, where $\x$ denotes model parameters, $n$ is the number of training samples, and $f_i$ denotes the loss incurred by the $i^{\text{th}}$ sample. 
Other instances of this problem arise in maximum likelihood estimation (\cite{mle_1,mle_2}), control theory (\cite{wu2018error}), and unsupervised learning (\cite{song}).
In many applications, \eqref{eq:main_prob} is both high-dimensional (large $d$) and data-intensive (large $n$). 

When $n$ is large, it becomes infeasible to process the entire dataset at every iteration, thus making classical algorithms such as gradient descent or Newton's method impractical. 
Consequently, stochastic and incremental variants of these algorithms have been widely adopted for such problems, because their per-iteration complexity is independent of $n$.
While first-order methods like stochastic gradient descent (SGD) enjoy a low per-iteration complexity of $\mathcal{O}(d)$, their convergence rates, even with enhancements like variance reduction or acceleration (\cite{saga, johnson2013accelerating}), remain linear at best. 
In contrast, second-order methods like Newton Incremental Method (NIM) by \cite{rodomanov2016superlinearly}, achieve a superlinear rate but at a $\O(d^3)$ cost, which is prohibitively large for high-dimensional problem settings.

Stochastic and incremental Quasi-Newton (QN) methods strike a balance between the computational efficiency of SGD and the fast convergence rate of NIM.
Specifically, the Incremental Quasi-Newton (IQN) method from  \cite{mokhtari2018iqn}, was the first QN method to achieve a superlinear convergence rate with a per-iteration complexity of $\O(d^2)$. 
However, the analysis presented in \cite{mokhtari2018iqn} was asymptotic and did not include an explicit rate of convergence. 
Typically, explicit rates are preferred over asymptotic ones as they enable a more fine-grained comparison among algorithms. 
For instance, with $\rho \in (0,1)$,
both $\mathcal{O}(\rho^{t^2})$ and $\mathcal{O}(\rho^{t\ln(t)})$ 
qualify as superlinear rates. However, the former, $\mathcal{O}(\rho^{t^2})$, is faster than the latter, $\mathcal{O}(\rho^{t\ln(t)})$. 
Furthermore, the mathematical expression of $\rho$ informs the rate's dependence on parameters like condition number and dimension.

The Incremental Greedy BFGS (IGS) method by \cite{gao2020incremental} aimed to address this issue by incorporating the greedy updates, first introduced in \cite{rodomanov2021greedy}, into the IQN framework.  
While IGS achieves an explicit superlinear convergence rate of $\O(e^{-t^2/n^2})$,
it suffers from several major drawbacks:
First, like NIM, it has a large per-iteration cost of $\mathcal{O}(d^3)$. 
This stems from IGS's Hessian updates not being low ran, which precludes an efficient evaluation of the Hessian inverse.
Since this complexity mirrors that of NIM, it undermines the computational advantages of incorporating QN updates.
Second, the empirical performance of IGS was not studied in \cite{gao2020incremental}, and our experiments (ref.~Figure \ref{all_combined_SN_edited_II.eps}) indicate that on multiple datasets, IGS severely underperforms compared to IQN. 
Finally, IGS lacked theoretical analysis to support its lemmas and theorems.
In this light, we ask the following question: 

\textit{Can we devise an incremental QN method with a per-iteration complexity of $\O(d^2)$, achieving the best-known incremental rate of $\O(e^{-t^2/n^2})$, and demonstrating superior empirical performance?}

We put forth the Sharpened Lazy IQN (SLIQN) method that meets all these objectives. 
SLIQN is inspired by the recent work of \cite{jin2022sharpened}, which showcased the superior performance of sharpened updates over greedy updates in the \textit{non-incremental setting}.
We first show that a direct incorporation of sharpened updates into the IQN framework does not work because the Hessian update matrices corresponding to sharpened updates are not low rank, and therefore the resulting Sharpened IQN (SIQN) method incurs a per-iteration cost of $\O(d^3)$.
We then propose our novel Sharpened Lazy IQN (SLIQN) algorithm that overcomes this limitation by modifying the updates of SIQN using a clever constant multiplicative factor and incorporating a lazy propagation strategy. 
The resulting algorithm incurs a per-iteration complexity of $\O(d^2)$ and achieves a convergence rate of $\O(\rho^{t^2/n^2})$, where $\rho := 1-\mu/dL$, which is the best-known rate in the incremental setting.
We also establish an explicit linear rate of convergence of the Hessian approximation to the true Hessian. 
Moreover, in contrast to IGS, we provide a comprehensive theoretical analysis.
Furthermore, we demonstrate the superior empirical performance of SLIQN as compared to IQN, IGS, and other state-of-the-art incremental and stochastic QN methods.\footnote{The code for the experiments is available on the repository: \href{https://github.com/aakashlahoti/sliqn}{https://github.com/aakashlahoti/sliqn}.} 
Notably, SLIQN demonstrates performance competitive to NIM, 
which is a second-order algorithm, with a $\O(d^3)$ cost, that utilizes the full Hessian information when taking the descent step.

%% file: content/related_work.tex
In recent decades, several works have developed first-order, QN, and second-order methods for stochastic or incremental settings. 
Typically, the goal for first-order methods is to achieve a linear rate at a \(\O(d)\) cost, while for QN and second-order methods, the goals are to achieve superlinear rates at costs of \(\O(d^2)\) and \(\O(d^3)\), respectively.
These methods cater to different objectives: first-order methods are preferred for low-precision solutions, due to their lower computational cost, whereas higher-order methods are more effective for high-precision solutions, due to their faster rate.

Early works like \cite{mokhtari2014res, mokhtari2015global,byrd2016stochastic} were only successful in developing QN methods with sub-linear convergence guarantees. Subsequent works like \cite{moritz2016linearly,  chang2019accelerated, derezinski2022stochastic}
employed various acceleration and variance reduction techniques to recover a linear rate. 
IQN by \cite{mokhtari2018iqn} was the first QN algorithm to achieve an asymptotic superlinear rate of convergence.
IGS by \cite{gao2020incremental} employed greedy updates, introduced in \cite{rodomanov2021greedy}, within the IQN framework to derive a explicit superlinear rate, albeit at a large $\mathcal{O}(d^3)$ per-iteration cost.
Another recent work, \cite{san}, put forth a QN style algorithm with a cost of $\mathcal{O}(d)$ for Generalized Linear Models (GLMs). However, the method only enjoys a linear rate of convergence, and is inefficient for general functions with a cost of $\mathcal{O}(d^3)$.

Other works have focused on developing first-order and second-order methods for stochastic or incremental settings. 
First order methods like \cite{saga, johnson2013accelerating} have employed variance reduction techniques to derive methods with a linear rate of convergence. 
On the second-order front, recent works include the Newton Incremental method (NIM) by \cite{rodomanov2016superlinearly} and Stochastic Newton (SN) by \cite{kovalev2019stochastic}.
While SN and NIM are both Newton-like methods with a per-iteration complexity of \(\cO(d ^ 3)\), NIM has a fast superlinear convergence rate while SN only enjoys a linear rate of convergence. The only setting under which SN has been shown to converge superlinearly is with a full batch of size $n$.
We have consolidated the memory usage, computational cost, and convergence rates of principal-related algorithms in Table \ref{tb:algo_table}. 

In a different line of work, Newton-LESS from \cite{derezinski2021newton} utilized sketching algorithms like Newton Sketch \cite{pilanci2016iterative, pilanci2017newton} to attain a local linear convergence rate. 
Additional works such as \cite{gonen2016solving, liu2019acceleration} have used second-order information to accelerate SVRG, which is a first-order method. However, both methods were only able to attain an improved linear rate of convergence.

\begin{table*}[t]
	\begin{center}
		\begin{tabular}{|c | c | c | c | c |} 
			\hline
			Algorithm & Memory & Computation cost & Convergence Rate & Limit  \\ 
			\hline
			SN & {\(\bm{\cO(d^2)}\)} & \(\cO(d ^ 3)\)& Linear & \textbf{Non-asymptotic}\\ 
			%		  		SAN\cite{chen2022san} & \(\cO(d)\) & & Linear & Non-asymptotic\\
			NIM & \(\cO(nd + d ^ 2)\) & \(\cO(d ^ 3)\) & \textbf{Superlinear} & \textbf{Non-asymptotic} \\
			IQN & \(\cO(nd ^ 2)\) & \(\bm{\cO(d ^ 2)}\) & \textbf{Superlinear} & Asymptotic \\
			IGS & \(\cO(nd ^ 2)\) & \(\cO(d ^ 3)\) & \textbf{Superlinear} & \textbf{Non-asymptotic}\\
			SIQN(This work) & \(\cO(nd ^ 2)\) & \(\cO(d ^ 3)\) & \textbf{Superlinear} & \textbf{Non-asymptotic}\\
			SLIQN(This work) & \(\cO(nd ^ 2)\) & \(\bm{\cO(d ^ 2)}\) & \textbf{Superlinear} & \textbf{Non-asymptotic} \\ 
			\hline
		\end{tabular}
	\end{center}
	\caption{\label{tb:algo_table} Comparison of the maximum memory requirement, computation cost (per iteration), convergence rate, and the limit of attainment of the convergence rate for different algorithms.}
\end{table*}

%% file: content/notations_and_preliminaries.tex
% \section{Notation and Preliminaries}\label{sec:overview}
Vectors (matrices) are denoted by lowercase (uppercase) bold alphabets. 
The $i$-th standard basis vector of $\Rn^d$ is denoted by $\e_i\in\Rn^d$ for $i\in [d] \coloneqq  \{1, \ldots, d\}$. 
We define the index function $i_t  \coloneqq  1 + (t-1) \bmod n$. The symbol $\zero$ denotes the all-zero matrix or vector, whose size can be inferred from the context. 
We use \(\X \succeq \zero \) and \(\X \succ \zero \) to denote that the symmetric matrix \(\X\) is positive semi-definite and positive definite, respectively. 
Likewise, the notation \(\X \succeq \Y\) ($\X \succ \Y$) denotes \(\X - \Y \succeq \zero \) ($\X-\Y\succ \zero$). 
For vectors \(\u, \v \in \Rn ^ d\), we denote the inner product by \(\ip{\u}{\v}  \coloneqq  \u ^ {\T}\v\) and the Euclidean norm by $\norm{\u}  \coloneqq  \sqrt{\ip{\u}{\u}}$. 
For matrices  \(\X, \Y \in \Rn ^ {d\times d}\), we define \(\ip{\X}{\Y}  \coloneqq  \mathrm{Tr}(\X^\T\Y)\), and we let \(\norm{\X}\) denote the spectral norm of the matrix. 
Given a convex function \(f: \Rn ^ d \to \Rn\), 
we define the norm of the vector $\y$ with respect to $\nabla^2f(\x)$ as
\(\norm{\y}_\x  \coloneqq  \sqrt{\ip{\y}{\nabla ^ 2 f(\x)\y}}\).
For a function \(f\), we denote \(\mu\) as the strong convexity parameter, \(L\) as the smoothness parameter, \(\tilde{L}\) as the Hessian Lipschitz continuity parameter, and \(M = \tilde{L}\mu^{-\frac{3}{2}}\) as the strong self-concordance parameter.

%\begin{definition}[Superlinear convergence] \label{super_linear} %\cite{lin2022explicit} 
%	Suppose a scalar sequence \(\{x_{k}\}, k \in \mathbb{N}\) converges to 0 with \(
%	\lim_{k \to \infty} \frac{\abs{x_{k + 1}}}{\abs{x_ {k}}} = q \in [0, 1).
%	\) Now suppose another scalar sequence \(\{y_{k}\}, k \in \mathbb{N}\) converges to \(y ^ {\star}\) and satisfies \(\abs{y_k - y^{\star}} \leq \abs{x_k},\) for all \(k \in \mathbb{N}\). We say \(\{y_{k}\}\) converges superlinearly iff \(q = 0\).
%\end{definition}
\subsection{Quasi-Newton (QN) Methods}
We introduce QN methods as an iterative algorithm to optimize the problem \eqref{eq:main_prob} with $n=1$. 
At iteration $t \in \mathbb{Z}_{+}$, given the current iterate $\x^t$ and the positive definite Hessian approximation $\B^t$ of $\nabla^2f(\x^t)$, the next iterate $\x^{t+1}$ is computed as,
\begin{align}\label{eq:quasi_newton_template}
	\x^{t+1} &=  \x^{t} - (\B^t)^{-1} \nabla f(\x^t).
\end{align} 

The Hessian approximation for the next iteration $\B^{t+1}$, is obtained by applying a constant rank update to $\B^{t}$. 
The precise update distinguishes the exact type of QN algorithm, such as BFGS, DFP, or SR1 (\cite{nocedal}).
The efficiency of QN methods, $\mathcal{O}(d^2)$, over second-order methods, $\mathcal{O}(d^3)$, stems from the fact that the $\B^{t}$ update is low rank, which allows us to use Sherman-Morrison formula (Appendix \ref{sec:prelim_lemma}) to efficiently evaluate $(\B^{t+1})^{-1}$ from $(\B^{t})^{-1}$ in $\mathcal{O}(d^2)$ cost.

Though in the remainder of the paper, we are primarily concerned with BFGS updates, all the follows can also be extended to the entire restricted Broyden class (Appendix \ref{sec:restricted_broyden}). 
Given a matrix $\K$ and its approximation $\B$, the generalized BFGS update refines this approximation along direction $\u\in \Rn^d$ as,
\begin{align}
	\label{BFGS}
	\bfgs{\B, \K, \u} = \B_{+} 
	\coloneqq
	\B - 
	\frac{
		\B\u\u ^ {\T}\B
	}{
		\ip{\u}{\B\u}
	} 
	+ 
	\frac{
		\K\u\u ^ {\T}\K
	}
	{
		\ip{\u}{\K\u}
	}.
\end{align}
Setting \(\K^t = \int_{0} ^ {1} \nabla ^ 2 f(\x ^ {t} + \tau(\x^ {t + 1} - \x ^ {t})) d\tau\),
and $\u^t = \s^t := \x ^ {t + 1} - \x ^ t$ yields the classical BFGS update,
\begin{align}\label{bfgs}
	\B ^ {t + 1} &= 
	\bfgs{\B ^ {t}, \K ^ {t}, \u ^ t} = 
	\B ^ t - 
	\frac{
		\B ^ t\s ^ t(\s ^ t) ^ {\T}\B ^ t
	}
	{
		\ip{\s ^ t}{\B ^ {t}\s^t}
	} + 
	\frac{
		\y ^ t (\y ^ t) ^ {\T}
	}
	{
		\ip{\s ^ t}{\y ^ {t}}
	},
\end{align}
where $\y_t \coloneqq \K ^ {t}\s^t = \nabla f (\x ^ {t + 1}) - \nabla f (\x ^ t)$. 
This update seeks to approximate the Hessian along the Newton direction $\s^t$ and has been shown by  \cite{jin2022non, rodomanov2021rates, rodomanov2020new} to achieve a superlinear convergence rate. Furthermore, since BFGS makes a rank $2$ update to $\B^t$, we can use the Sherman-Morrison formula twice to evaluate \((\B ^ {t+1}) ^ {-1}\) from  \((\B ^ {t}) ^ {-1}\) in \(\cO(d ^ {2})\) cost. 

In contrast to classical BFGS update, the greedy BFGS update by \cite{rodomanov2021greedy} sets $\K^t = \nabla^2f(\x^t)$, and defines the greedy vector,
\begin{align}\label{greedy_vec}
	\bar{\u}(\B^t, \K^t) 
	\coloneqq
	\argmax_{\u \in \{\e_i\}_{i = 1} ^ d} 
	\tfrac{
		\ip{\u}{\B^t\u}
	}
	{\ip{\u}      
		{\K^t\u}},
\end{align}
which results in $\B^{t+1} = \bfgs{\B^t, \K^t, \ub^t(\B ^ t, \K ^ t)}$. 
Greedy BFGS, similar to classic BFGS, exhibits a superlinear rate of convergence. 
However, unlike classic BFGS, it can guarantee convergence in the $\sigma$ sense.
Specifically, the Hessian approximation error, 
\begin{align}\label{sigma}
	\sigma(\B^t, \K^t) \coloneqq \langle (\K^t)^{-1}, \B^t - \K^t \rangle,
\end{align}
decays linearly with $t$. 
In practice, a trade-off exists between the two updates.
Greedy BFGS dedicates initial iterations to construct a reliable Hessian approximation. 
In contrast, classic BFGS gains an initial advantage because it updates along the Newton direction. 
Greedy BFGS only outperforms classic BFGS if it has enough time to 
build an accurate approximation before convergence.
Sharpened BFGS proposed by \cite{jin2022sharpened} incorporated both classic and greedy BFGS updates, to achieve the best of both updates: an explicit superlinear convergence of $\x^t$, a linear convergence of $\B^t$ and no initial ``ramp-up" phase to build the approximation.

\subsection{Incremental Quasi-Newton (IQN)}\label{iqn_section} 
We now introduce IQN by \cite{mokhtari2018iqn}.
For each iteration $t \ge 1$, IQN maintains tuples of the form 
$(\z_i^t, \nabla f_i(\z_i^t), \B_i^t)$ for each index $i \in [n]$. 
Here $\z_i^t \in \Rn^d$ is the iterate corresponding to the function $f_i$,  $\nabla f_i(\z_i^t)$ is the gradient of $f_i$ at $\z_{i}^t$, and $\B_i^t$ is the positive definite Hessian approximation of $\nabla^2 f_i(\z_i^t)$. 

IQN begins by constructing a second-order Taylor approximation $g_i^{t}$ of each $f_i$, centered at $\z_i^{t-1}$ and using the Hessian approximation $\B_i^{t-1}$ as,
\begin{align}
	g_i^{t}(\x) 
	&:= 
	f_i(\z_i^{t-1}) + 
	\ip{\nabla {f_i(\z_i^{t-1})}}{\x-\z_i^{t-1}} + 
	\frac{1}{2} \ip{\x-\z_i^{t-1}}{\B_i^{t-1} (\x-\z_i^{t-1})}.
\end{align}
The iterate $\x ^ {t}$ is then calculated as,
\begin{align}
	\x ^ {t} &= \argmin_\x \frac{1}{n} \sum_{i=1}^{n}g_i^{t}(\x) = (\Bar{\B}^{t-1})^ {-1}\bigg(\sum_{i = 1} ^ {n}\B_{i} ^ {t - 1}\z_{i} ^ {t - 1} - \nabla f_{i}(\z_{i} ^ {t - 1})\bigg),\label{x_update_iqn}
\end{align}
where $\Bar{\B}^{t-1}=\sum_{i = 1} ^ {n} \B_{i} ^ {t - 1}$.
In every iteration $t$, IQN only updates the tuple whose index is given by $i_t = 1 + (t-1)\bmod n$, using the following scheme:
\begin{enumerate}
    \item $\label{b_update_temp_1}\z_{i_t}^t = \x^t$,  $\nabla f_{i_t}(\z_{i_t}^t) = \nabla f_{i_t}(\x ^ t)$.
    \item 
$\label{b_update_temp_3}\z_j^t = \z_j^{t-1}$, $\nabla f_{j}(\z_{j} ^ t)  = \nabla f_{j}(\z_{j} ^ {t - 1})$, and $\B_j^t = \B_j^{t-1}$, for all $j\neq i_t$.
    \item 
\(\label{b_update_i_T}
\B_{i_t} ^ {t} = 
\bfgs{\B_{i_t} ^ {t - 1}, 
	\K ^ t, 
	\z_{i_t} ^ t - \z_{i_t} ^ {t - 1}},
\) where \(\K ^ t = \int_{0} ^ {1} \nabla ^ 2 f_{i_t}(\z_{i_t} ^ {t - 1} + \tau (\z_{i_t} ^ t - \z_{i_t} ^ {t - 1})) d\tau\).
\end{enumerate}

The per-iteration complexity of IQN: 
the cost of gradient evaluation in \eqref{b_update_temp_1} is $\mathcal{O}(d)$,
the no-operation step in \eqref{b_update_temp_3} incurs no cost, 
and the BFGS step \eqref{b_update_i_T} has a $\mathcal{O}(d^2)$ cost as it is a constant rank update to $\bB_{i_t}^t$.
To compute the iterate $\x^{t+1}$, we first evaluate the inverse of \(\bB^t = \bB^{t-1} + \B_{i_t}^t - \B_{i_t}^{t-1}\) from $(\bar{\B}^{t-1})^{-1}$ using the Sherman-Morrison formula at a cost of $\O(d^2)$.
Then, we calculate  $\sum_{i = 1} ^ {n}\B_{i} ^ {t}\z_{i} ^ {t} - \nabla f_{i}(\z_{i} ^ {t})$ from the memoized value of
$\sum_{i = 1} ^ {n}\B_{i} ^ {t - 1}\z_{i} ^ {t - 1} - \nabla f_{i}(\z_{i} ^ {t - 1})$ in $\O(d^2)$ cost. 
Therefore, the overall per-iteration cost of IQN is $\mathcal{O}(d^2)$.
Please refer to Appendix \ref{app:iqn_efficient_implementation} for details.

%% file: content/proposed_algorithm.tex
We first introduce the SIQN method, which incorporates the sharpened updates into the IQN framework.
However, this direct incorporation results in an inefficient $\O(d^3)$ method.
This is because, ensuring the positive semi-definiteness of the Hessian approximation with respect to the true Hessian results in Hessian updates which are not low rank. 
And consequently, the Hessian inversion incurs a large $\mathcal{O}(d^3)$ cost. 
We then propose the SLIQN algorithm that overcomes this problem by modifying the SIQN updates using a constant corrective multiplicative factor based on the theoretical analysis of SIQN. 
It then utilizes a novel lazy propagation strategy to implement this factor correction efficiently with a per-iteration cost of $\mathcal{O}(d^2)$. 

\subsection{Sharpened Incremental Quasi-Newton (SIQN)}\label{section_siqn}
Similar to IQN, SIQN also maintains tuples of the form
$(\z_i^t, \nabla f_i(\z_i^t), \B_i^t)$ for each iteration $t\geq 1$ and each index $i \in [n]$. 
The iterate $\x^t$ is calculated as,
\begin{align}\label{x_update}
	\x ^ {t} = (\Bar{\B}^{t-1})^ {-1}
	\bigg(
	\sum_{i = 1} ^ {n}\B_{i} ^ {t - 1}\z_{i} ^ {t - 1} - \nabla f_{i}(\z_{i} ^ {t - 1})
	\bigg),
\end{align}
where $\Bar{\B}^{t-1}=\sum_{i = 1} ^ {n} \B_{i} ^ {t - 1}$.
The tuples are updated in a deterministic cyclic order as follows:
\begin{enumerate}
    \item $\label{sliqn_update_temp_1}\z_{i_t}^t = \x^t$,  $\nabla f_{i_t}(\z_{i_t}^t) = \nabla f_{i_t}(\x ^ t)$.
    \item 
$\label{sliqn_update_temp_3}\z_j^t = \z_j^{t-1}$, $\nabla f_{j}(\z_{j} ^ t)  = \nabla f_{j}(\z_{j} ^ {t - 1})$, and $\B_j^t = \B_j^{t-1}$, for all $j\neq i_t$.
\end{enumerate}
To update $\B_{i_t}^t$, SIQN first performs the classic BFGS update followed by the greedy update,
\begin{enumerate}
\setcounter{enumi}{2}
	\item ${\Q ^ t} = 
	\bfgs{
		(1 + \beta_t) ^ 2\B_{i_t} ^ {t - 1},
		(1 + \beta_t)\K ^ {t}, 
		\z_{i_t} ^ {t} - \z_{i_t} ^ {t - 1}
	}$, \nonumber \\
        \\
	$\B_{i_t} ^ {t} = 
	\bfgs{
		{\Q} ^ t, 
		\nabla ^ 2 f_{i_t}(\z_{i_t} ^ t), 
		\bar{\u}(\Q^t,
		\nabla ^ 2 f_{i_t}(\z_{i_t} ^ t))}, 
    $
    \label{modified_second_step_preprint}
\end{enumerate} 
where, 
\(\K ^ {t} \coloneqq \int_{0}^{1} \nabla ^ 2 f_{i_t}(\z_{i_t} ^ {t - 1} + \tau (\z_{i_t} ^ t - \z_{i_t} ^ {t - 1})) d\tau\)
and the scaling factor
$\beta_t \coloneqq \tfrac{M}{2}\norm{\z_{i_t}^t-\z_{i_t}^{t-1}}_{\z_{i_t}^{t-1}}$ 
ensures that before the classical BFGS update,
the Hessian approximation is positive semi-definite with respect to the corresponding Hessian,
\((1 + \beta_t) ^ 2 {\B}_{i_t} ^ {t - 1} \succeq (1 + \beta_t) \K ^ t\). For technical details, please refer to Lemma \ref{lem:formalize_psdness}.
The pseudo-code for SIQN is provided in Algorithm \ref{alg:h_iqn}.

\begin{algorithm}[t]
	\caption {Sharpened Incremental Quasi-Newton}
	\begin{algorithmic}[1] \label{alg:h_iqn}
		\STATE\textbf{Initialize: \(\{\z_{i}^{0} = \x^{0}\}_{i = 1} ^ {n}, \{\B_{i} ^ {0}\}_{i = 1} ^ {n}\)} such that for all $i \in [n]$,
  \(\B_{i} ^ 0 \succeq \nabla ^ 2 f_{i}(\z_{i}) ^ 0\)
		%			\STATE Evaluate \(\H^0 = (\sum_{i} \B_{i} ^ 0) ^ {-1}\)
		\STATE\textbf{while} \textit{not converged:}
		\STATE\hspace{3mm} Set current index \(i_{t} = (t - 1) \mod n + 1\);
		\STATE\hspace{3mm} Update \(\x^t\) as per \eqref{x_update}; 
		%			\STATE\hspace{3mm} Set \(\Q^t = \bfgs{\P^{t - 1}, \big(1 + \beta_t\big) ^ 2 \K ^ t, \z_{i_t}^t - \z_{i_{t}}^{t - 1}}\) where \(\P ^ {t - 1} = (1 + \beta_t) ^ 2 \B_{i_t} ^ {t - 1}\);
		%			\(r_{t} = \norm{\z_{i_t} ^ t - \z_{i_t}^{t - 1}}_{\z_{i_t}^{t - 1}}\)
		%			\STATE\hspace{3mm} Update \(\B_{i_t}^t = \bfgs{\Q^t, \nabla ^ {2}f_{i_t}(\z_{i_t}^t), \bar{\u}(\Q^t, \nabla ^ {2} f_{i_t}(\z_{i_t}^t))}\);
		\STATE\hspace{3mm} Update \(\z_{i_t} ^ t = \x ^ t\) and \(\B_{i_t} ^ t\) as per \eqref{modified_second_step_preprint};
		\STATE\hspace{3mm} Update the tuples with index \(j \neq i_t\) as 	\(\z_j^{t} = \z_j^{t-1}, 
		\B_j^{t} = \B_j^{t-1}\);
		\STATE\hspace{3mm} Increment the iteration counter \(t\);
		\STATE\textbf{end while}
	\end{algorithmic}
	\label{alg:hiqn}
\end{algorithm} 
We now consider the per-iteration complexity of SIQN: 
the cost of gradient evaluation is $\mathcal{O}(d)$, 
the no-operation step incurs no cost
and the BFGS step can be computed in $\mathcal{O}(d^2)$ cost.
To compute $\x^{t+1}$,
we need to calculate the inverse of \(\bB^t = \bB^{t-1} + \B_{i_t}^t - \B_{i_t}^{t-1}\),
\begin{align}
	\B_{i_t} ^ t - \B_{i_t} ^ {t - 1} 
	&= \bfgs{\Q ^ {t}, \nabla ^ 2 f_{i_t}(\z_{i_t} ^ t), \bar{\u}(\Q ^ {t}, \nabla ^ 2 f_{i_t}(\z_{i_t} ^ t))} - \B_{i_t} ^ {t - 1}, \nn \\
	&=
	\Q^ t - \frac{\Q ^ t\bar{\u} ^ t (\bar{\u} ^ t) ^ {\T}\Q ^ t}{\ip{\bar{\u} ^ t}{\Q ^ t\bar{\u} ^ t}} + \frac{
		\nabla ^ 2 f_{i_t}(\z_{i_t} ^ t)\bar{\u} ^ {t}(\bar{\u} ^ {t}) ^ {\T}\nabla ^ 2 f_{i_t}(\z_{i_t} ^ t)}
	{
		\ip{\bar{\u} ^ {t}}{\nabla ^ 2 f_{i_t}(\z_{i_t} ^ t)\bar{\u} ^ {t}}
	} 
	- \B_{i_t} ^ {t - 1}, \nn \\
	&=
	((1 + \beta_{t}) ^ 2 -1)
	\B_{i_t} ^ {t - 1} - (1+\beta_t) ^ 2\frac{\B_{i_t} ^ {t - 1}\s_{i_t} ^ t{\s_{i_t} ^ t} ^ {\T}\B_{i_t} ^ {t - 1}}{\ip{\s_{i_t} ^ t}{\B_{i_t} ^ {t - 1}\s_{i_t} ^ t}}\nn \\
	& \hspace{5mm}+ (1 + \beta_t)\frac{\y_{i_t} ^ t{\y_{i_t} ^ t} ^ {\T}}{\ip{{\y_{i_t} ^ t} ^ \T}{\s_{i_t} ^ t}} - \frac{\Q ^ t\bar{\u} ^ t (\bar{\u} ^ t) ^ {\T}\Q ^ t}{\ip{\bar{\u} ^ t}{\Q ^ t\bar{\u} ^ t}} +\frac{\nabla ^ 2 f_{i_t}(\z_{i_t} ^ t)\bar{\u} ^ {t}(\bar{\u} ^ {t}) ^ {\T}\nabla ^ 2 f_{i_t}(\z_{i_t} ^ t)}{\ip{\bar{\u} ^ {t}}{\nabla ^ 2 f_{i_t}(\z_{i_t} ^ t)\bar{\u} ^ {t}}}
	\label{observation},
\end{align}
where \(\bar{\u} ^ {t} = \bar{\u}(\Q ^ {t}, \nabla ^ 2 f_{i_t}(\z_{i_t} ^ t))\), \(\s_{i_t} ^ t = \z_{i_t} ^ t -\z_{i_t} ^ {t - 1}\) and \(\y_{i_t} ^ t = \nabla f_{i_t}(\z_{i_t} ^ t)  - \nabla f_{i_t}(\z_{i_t} ^ {t - 1})\). 
Observe that the expression \eqref{observation} is generally not a matrix of constant rank, since the rank of \(((1 + \beta_t) ^ 2 - 1) \B_{i_t} ^ {t - 1}\) may be as large as $d$. 
Therefore, it is not possible to compute the inverse of $\bB^t$ by using the inverse of $\bB^{t-1}$ in $\mathcal{O}(d^2)$ cost. 

% \begin{rem}
%     We remark that the complication arises only due to the $(1+\beta_t)^2$ correction factor introduced by the classical BFGS update rule and it also plagues the greedy IQN variant proposed in \cite{gao2020incremental}. 
% \end{rem}

\subsection{Sharpened Lazy Incremental Quasi-Newton (SLIQN)} \label{subsec:efficient_hiqn}
We now present the SLIQN method as a solution to the aforementioned issues. 
We observe that the primary reason for the update of $\Bar{\B}^{t}$ in SIQN not being low-rank, is the presence of the scaling factor $\beta_t$.
To resolve this issue, we begin by noting that the convergence analysis of SIQN (ref.~Appendix \ref{app:convergence_iqn}) indicates that there exists a factor $\alpha_{\lceil t/n - 1\rceil}$, such that $\beta_t \leq \alpha_{\lceil t/n - 1\rceil}$ and using $\alpha_{\lceil t/n - 1\rceil}$ instead of $\beta_t$ preserves the convergence properties of SIQN.
Since $\alpha_{\lceil t/n - 1\rceil}$ is constant throughout an epoch,
rather than multiplying $\alpha_{\lceil t/n - 1\rceil}$ to $\B_{i_t}^{t-1}$ at iteration $t$, 
we instead pre-multiply $\alpha_{\lceil t/n - 1\rceil}$ to each Hessian approximation at the start of every epoch. 
This enables us to compute the inverse of $\Bar{\B}^{t}$ trivially by dividing the old inverse by $\alpha_{\lceil t/n - 1\rceil}$.
However, this pre-multiplication step is a $\mathcal{O}(nd^2)$ operation and it undermines the utility of incremental algorithms.
To address this issue, we employ a lazy propagation strategy, wherein we scale the individual Hessian approximations just before they are updated in their respective iterations, but treat all memoized quantities as if the approximations are already scaled. 
These key changes enable SLIQN to achieve an $\mathcal{O}(d^2)$ per-iteration cost along with an explicit superlinear rate.

In what follows, we will denote the Hessian approximations by $\{\D_i^{t}\}_{i=1}^n$ instead of \(\{\B_i ^ t\}_{i = 1} ^ {n}\) to distinguish SLIQN updates from SIQN updates. We now formally present the SLIQN algorithm.

\textbf{Initialization:} At $t=0$, we initialize the iterates \(\{\z_i ^ {0}\}_{i = 1} ^ {n}\) as \(\z_i ^ {0} = \x ^ 0\) 
and their corresponding hessian approximations $\{\D_{i}^0\}_{i=1}^n$ as 
\(\D_{i} ^ 0 =  (1 + \alpha_0) ^ 2\I_i ^ 0\), where \(\x_0, \alpha_0, \{\I_i ^ 0\}_{i = 1} ^ n\) satisfy the the premise of Lemma \ref{lem:convergence_efficient_h_iqn}.

\textbf{Iterative Updates:}
For each iteration \(t \geq 1\), 
we set the iterate $\x^t$ as,
\begin{align}\label{x_new_update}
	\x ^ t = \left(\Bar{\D}^{t-1}\right)^ {-1}
    \bigg(\sum_{i = 1} ^ {n} \D_{i} ^ {t - 1}\z_{i} ^ {t - 1} - \nabla f_{i}(\z_{i} ^ {t - 1})\bigg),
\end{align} 
where $\Bar{\D}^{t-1} =  \sum_{i = 1} ^ {n} \D_{i} ^ {t - 1}$.
The scheme to update each tuple is as follows:
\begin{enumerate}
    \item $\z_{i_t}^t = \x^t$,  $\nabla f_{i_t}(\z_{i_t}^t) = \nabla f_{i_t}(\x ^ t)$.
    \item $\label{gamma_i_neq_i_t}
	\z_i ^ t = \z_i ^ {t - 1}, \D_i ^ t = \omega_t \D_i ^ {t - 1}$, $\forall i \in [n]\backslash \{i_t\}$, where \(\omega_t := (1 + \alpha_{\ceil{t / n}}) ^ 2\) if \(t \hspace{-4pt} \mod n = 0\) and \(1\) otherwise.
    \item 
    $\Q ^ {t} = \bfgs{\D_{i_t} ^ {t - 1}, (1 + \alpha_{\ceil{t / n - 1}}) \K ^ t, \z_{i_t} ^ t - \z_{i_t} ^ {t - 1}} \label{eff_first_modified_update_}$,
    \\
    \\
    $\D_{i_t} ^ t = \omega_t\bfgs{\Q ^ {t}, \nabla ^ 2 f_{i_t}(\z_{i_t} ^ t), \bar{\u} ^ t (\Q^t, \nabla ^ 2 f_{i_t}(\z_{i_t} ^ t))} \label{eff_second_modified_update}$.
\end{enumerate}

Algorithm \ref{alg:eff_h_iqn} provides the pseudo-code for SLIQN. 
We now consider its per-iteration complexity. 
Observe that in the $j+1^{\text{th}}$ epoch, the updates \eqref{gamma_i_neq_i_t} and \eqref{eff_second_modified_update}  are carried out differently for \(t \hspace{-4pt} \mod n = 0\) and \(t \hspace{-4pt} \mod n \neq 0\). 
Specifically, for \(t \in \{ nj+1, \dots, nj + n - 1 \}\), we can carry out the update
\eqref{x_new_update} in \(\cO(d ^ 2)\) cost using Sherman-Morrison formula (ref.~Appendix \ref{update_efficiently})
and can compute other iterative updates in $\mathcal{O}(d^2)$ cost as they consist of a constant number of matrix-vector multiplications.
However, for \(t = nj + n\), each \(\D_i ^ t\) is multiplied with the scaling factor of \((1 + \alpha_{j + 1}) ^ 2\), which incurs a large \(\cO(nd^2)\) overhead. 
Instead, we implement this step by lazily scaling the Hessian approximations at the iteration in which they are updated while treating all memoized quantities as if the approximations are already scaled. The details of the lazy strategy are provided in Appendix \ref{subsec:lazy_algorithm}.

\begin{algorithm}
	\caption{Sharpened Lazy IQN}
	\begin{algorithmic}[1]
		\STATE\textbf{Initialization:} Initialize the iterates \(\{\z_i ^ {0}\}_{i = 1} ^ {n}\) as \(\z_i ^ {0} = \x ^ 0\) 
and their corresponding hessian approximations $\{\D_{i}^0\}_{i=1}^n$ as 
\(\D_{i} ^ 0 =  (1 + \alpha_0) ^ 2\I_i ^ 0\)
		%			\STATE Evaluate \( (\sum_{i} \D_{i} ^ 0) ^ {-1};\)
		\STATE\textbf{while} \textit{not converged:}
		%			\STATE\hspace{3mm}Current index to be updated is \(i_{t} = (t - 1) \mod n + 1;\)
		%			\STATE\hspace{3mm}\textbf{if} \(i_{t} = 1\)
		%			\STATE\hspace{3mm}\hspace{3mm}\textbf{for} \(j = 1, \dots, n\)
		%			\STATE\hspace{3mm}\hspace{3mm}\hspace{3mm} Reset \(\D_{j} ^ {t - 1} = (1 + M \sqrt{L}\rho ^ {\floor{\frac{t}{n}}}\epsilon) ^ {2}\B_{j} ^ {t - 1};\)
		%			\STATE\hspace{3mm}\hspace{3mm}\textbf{end for}
		%			\STATE\hspace{3mm}\textbf{end if}
		\STATE\hspace{3mm}Update \(\x^t\) as per \eqref{x_new_update};
		\STATE\hspace{3mm}Update \(\z_{i_t} ^ t\) as \(\z_{i_t} ^ t = \x ^ t\);
		%			\STATE\hspace{3mm}Set \(\Q^t = \bfgs{\D_{i_t}^{t - 1}, (1 + M\sqrt{L}\rho ^ {\floor{\frac{t}{n}}}\epsilon) ^ 2 \K ^ t, \z_{i_t}^t - \z_{i_{t}}^{t - 1}}\);
		\STATE\hspace{3mm}Update \(\Q ^ t\) and \(\D_{i_t} ^ t\) as per \eqref{eff_second_modified_update};
		\STATE\hspace{3mm}Update the tuples with index \(i \neq i_t\) as per \eqref{gamma_i_neq_i_t};
		\STATE\hspace{3mm}Increment the iteration counter \(t\);
		%			\STATE\hspace{3mm}Update \(\D_{i_{t}}^t = \B_{i_t}^t\) and \(\D_{i}^t = \D_{i} ^ {t - 1},\) for \(i \neq i_t\);
		\STATE\textbf{end while}
	\end{algorithmic}
	\label{alg:eff_h_iqn}
\end{algorithm}

\begin{rem}
    We remark that both IQN and SLIQN exhibit a per-iteration cost of $\mathcal{O}(d^2)$ which is in contrast to the $\mathcal{O}(nd^2)$ cost for QN methods. However, this efficiency comes with an increased memory cost of $\mathcal{O}(nd^2)$.
    To address this issue we propose a pipelining scheme in Appendix \ref{app:low_memory_implementation}, in which we leverage a much larger disk to augment the main memory by prefetching the data and processing the updates in parallel.
\end{rem}

%% file: content/theoretical_analysis.tex
\subsection{Assumptions}
We analyze SLIQN under the assumptions of smoothness, strong convexity, and Lipschitz continuity of the Hessian. These assumptions are commonly used in the analysis of Quasi-Newton methods.
\begin{assumption}\label{ass:smooth_strong}
	(\textbf{Strong convexity and smoothness}) 
    The functions $\{f_i\}_{i=1}^n$ are $\mu$-strongly convex and $L$-smooth, 
    that is
    $
         \frac{\mu}{2} \norm{\y - \x} ^ {2} 
         \leq f_i(\y) - f_i(\x) - \ip{\nabla f_i(\x)}{\y - \x} 
    $, 
    and 
    $
        f_i(\y) - f_i(\x) - \ip{\nabla f_i(\x)}{\y - \x} 
        \leq \frac{L}{2} \norm{\y - \x} ^ {2}
    $ 
    hold for all $\x,\y\in\Rn^d$, and for all $i\in[n]$. 
\end{assumption}
\begin{assumption}\label{ass:lip_hessian}
	(\textbf{Lipschitz continuous Hessian}) The Hessians $\{\nabla^2f_i\}_{i = 1} ^ {n}$ are  \(\tilde{L}\)-Lipschitz continuous, that is, 
 $\norm{\nabla ^ 2 f_{i}(\y) - \nabla ^ 2 f_{i}(\x)} \leq \tilde{L} \norm{\y - \x}$, for all $\x,\y\in\Rn^d$, and for all $i\in[n]$. 
\end{assumption}
The above assumptions also imply that the functions $\{f_i\}_{i=1}^n$ are $M$-strongly self-concordant, which is defined as,  
\(
\nabla ^ 2 f_i(\y) - \nabla ^ 2 f_i(\x) \preceq M \norm{\y - \x}_\z \nabla ^ 2 f_i(\w) \), $\forall \x, \y, \z, \w\in\Rn^d$ with \(M \coloneqq \tilde{L}\mu ^ {-\frac{3}{2}}\)
\cite[Ex 4.1]{rodomanov2021greedy}.

\subsection{Convergence Lemmas}
We establish the convergence guarantees in three steps: 
Lemma \ref{lem:distance_change_single_iter_ver_II} establishes a one-step inequality that  bounds the residual \(\norm{\x^{t} - \x^{\star}}\) in terms of the previous residuals \(\norm{\z_{i} ^ {t - 1} - \x^\star}\) and the norm error in the Hessian approximation \(\norm{\D_{i} ^ {t - 1} - \nabla ^ 2 f_{i}(\z_{i} ^ {t - 1})}\), $\forall i \in [n]$. 
% Note that the residuals \(\norm{\z_{i} ^ {t - 1} - \x^\star}\) for \(i \in [n]\) are related to the previous $n$ residuals \(\{\norm{\x ^ {t - j} - \x^\star}\}_{j = 1} ^ {n}\). 
Lemma \ref{lem:convergence_efficient_h_iqn} 
uses the result of Lemma \ref{lem:distance_change_single_iter_ver_II} to inductively show that both the residual $\norm{\x^t-\x^\star}$ and the Hessian approximation error in the $\sigma$ sense, i.e., \(\sigma(\D_{i_t} ^ {t}, \nabla ^ 2 f_{i_t}(\z_{i_t} ^ t))\),  decrease linearly with $\ceil{t/n}$. 
Using the result of Lemma \ref{lem:distance_change_single_iter_ver_II} and Lemma \ref{lem:convergence_efficient_h_iqn}, Lemma \ref{lem:mean_linear_verII} establishes a mean-superlinear convergence result. 
We finally show in Theorem \ref{lem:final_rate} that the residuals can be upper bounded by a superlinearly convergent sequence.

\begin{lem}\label{lem:distance_change_single_iter_ver_II}
	If Assumptions \ref{ass:smooth_strong} and \ref{ass:lip_hessian} hold, the sequence of iterates generated by Algorithm \ref{alg:eff_h_iqn} satisfy 
	\begin{align}\label{lem1eq_preprint}
		\norm{\x^{t} - \x^\star} &\leq {\frac{\tilde{L}\Gamma ^ {t - 1}}{2}} \sum_{i = 1} ^ {n} \norm{\z_{i} ^ {t - 1} - \x^\star} ^ 2 
		+ {\Gamma ^ {t - 1}} \sum_{i = 1} ^ {n} &\norm{\D_i ^ {t - 1} - \nabla ^ 2f_i(\z_{i} ^ {t - 1})}\norm{\z_{i}^{t - 1} - \x^\star},
	\end{align} 
	for all \(t \geq 1\), where \(\Gamma ^ {t} := \norm{\big(\sum_{i = 1} ^ {n} \D_{i} ^ {t}\big)^{-1}}\).
\end{lem}
The proof of this result can be found in Appendix \ref{app:distance_change_single_iter_ver_II}. It is important to note that our bound in \eqref{lem1eq_preprint} differs from a similar result presented in \cite[Lemma 2]{mokhtari2018iqn}, where they utilize \(\norm{\D_{i} ^ {t - 1} - \nabla ^ 2 f_{i}(\x ^ {\star})}\) instead of \(\norm{\D_{i} ^ {t - 1} - \nabla ^ 2 f_{i}(\z_{i} ^ {t - 1})}\). 
This modification helps connect the approximation error in the norm sense, i.e., \(\norm{\D_{i} ^ {t - 1} - \nabla ^ 2 f_{i}(\z_{i} ^ {t - 1})}\) with the error in the $\sigma$ sense, i.e., \(\sigma(\D_{i} ^ {t - 1}, \nabla ^ 2 f_{i}(\z_{i} ^ {t - 1}))\). This connection is crucial for quantifying the improvements achieved by the greedy updates.

\begin{lem}\label{lem:convergence_efficient_h_iqn}
	If Assumptions \ref{ass:smooth_strong} and \ref{ass:lip_hessian} hold, for any \(\rho\) such that \(0 < \rho < 1 - \frac{\mu}{dL}\), there exist positive constants \(\epsilon\) and \(\sigma_0\) such that if \(\norm{\x ^ {0} - \x^{\star}} \leq \epsilon\) and \(\sigma(\I_i ^ {0}, \nabla ^ 2 f_i(\x ^ {0})) \leq \sigma_0\) for all \(i \in [n]\), the sequence of iterates generated by Algorithm \ref{alg:eff_h_iqn} satisfy \begin{align}\label{contraction_verII}
		\norm{\x ^ {t} - \x^\star} \leq \rho ^ {\ceil{\frac{t}{n}}} \norm{\x ^ {0} - \x^\star}.
	\end{align}
	Further, it holds that\begin{align}\label{sigma_rel_verII}
		\sigma(\omega_t ^ {-1}\D_{i_t}^{t}, \nabla ^ 2 f_{i_t}(\z_{i_t}^t)) \leq \big(1 - \frac{\mu}{dL}\big) ^ {\ceil{\frac{t}{n}}} \delta,
	\end{align}
	where \(\delta :=  e ^ {\frac{4M\sqrt{L}\epsilon}{1 - \rho}} \big(\sigma_0 + \epsilon \frac{4Md\sqrt{L}}{1 - \frac{\rho}{1 - \frac{\mu}{dL}}}\big), M = \tilde{L} / {\mu ^ {\frac{3}{2}}}\), $\omega_t = (1 + \alpha_{\ceil{t / n}}) ^ 2$ if \(t\) is a multiple of \(n\) and \(1\) otherwise, and the sequence \(\{\alpha_k\}\) is defined as \({\alpha_k := M\sqrt{L}\epsilon\rho ^ k}, \forall k \ge 0\).
\end{lem}
The proof of Lemma \ref{lem:convergence_efficient_h_iqn} can be found in Appendix \ref{app:convergence_efficient_h_iqn}.
Under the hood, the proof uses induction to show that if the initialized \(\x ^ {0}\) is close to \(\x ^ {\star}\) and \(\D_i ^ {0}\) is close to \(\nabla ^ 2 f_{i}(\x ^ {0})\), then the iterate \(\x ^ {t}\) converges linearly to \(\x ^ {\star}\) and $\D_{i_t}^{t}$ converges linearly to $\nabla^2f_i(\z_{i_t} ^ t)$.
Our result is stronger than \cite[Lemma 3]{mokhtari2018iqn} as we establish the linear convergence of $\D_i^t$, whereas \cite{mokhtari2018iqn} were only able to establish that 
\(\norm{\D_{i} ^ {t} - \nabla ^ 2 f_{i}(\x ^ {\star})}\) does not grow with $t$. Moreover, \cite{mokhtari2018iqn} did not guarantee convergence of $\D_{i} ^ {t}$. 
Also, there is no equivalent convergence result presented in the analysis of IGS.
\begin{lem}\label{lem:mean_linear_verII}
	If Assumptions \ref{ass:smooth_strong}, \ref{ass:lip_hessian} hold, then there exist positive constants \(\epsilon\) and \(\sigma_0\) such that if \(\norm{\x ^ {0} - \x^{\star}} \leq \epsilon\) and \(\sigma(\I_i ^ {0}, \nabla ^ 2 f_i(\x ^ {0})) \leq \sigma_0\) for all \(i \in [n]\), the sequence of iterates generated by Algorithm \ref{alg:eff_h_iqn} satisfy
	\begin{align}\nn
		\norm{\x^{t} - \x^\star} \leq \big(1 - \frac{\mu}{dL}\big) ^ {\ceil{\frac{t}{n}}} \frac{1}{n}\sum_{i = 1} ^ n \norm{\x^{t - i} - \x^\star}. 
	\end{align}
\end{lem}
The proof of Lemma \ref{lem:mean_linear_verII} can be found in Appendix \ref{app:mean_linear_verII}. The main idea behind the proof is to substitute the linear convergence results, specifically \eqref{contraction_verII} and \eqref{sigma_rel_verII} from Lemma \ref{lem:convergence_efficient_h_iqn}, back into the result from Lemma \ref{lem:distance_change_single_iter_ver_II}.
By doing so, the first term on the right-hand side of \eqref{lem1eq_preprint} converges quadratically, while the second term converges superlinearly, which proves the result.

Our result is markedly different from \cite[Theorem 6]{mokhtari2018iqn} which proves asymptotic superlinear convergence of IQN. Their proof is based on a variant of the Dennis-Mor\'{e} theorem to show that superlinear convergence kicks asymptotically. 
Since their proof is existential, unlike Lemma \ref{lem:mean_linear_verII}, it cannot be used to derive an explicit rate of superlinear convergence.

\begin{thm}
	\label{lem:final_rate}
	If Assumptions \ref{ass:smooth_strong}, \ref{ass:lip_hessian} hold, then there exist positive constants \(\epsilon\) and \(\sigma_0\) such that if \(\norm{\x ^ {0} - \x^{\star}} \leq \epsilon\) and \(\sigma(\I_i ^ {0}, \nabla ^ 2 f_i(\x ^ {0})) \leq \sigma_0\) for all \(i \in [n]\), and for the sequence of iterates generated by Algorithm \ref{alg:eff_h_iqn}, there exists a sequence 
	\(\{\zeta ^ {k}\}\), \(k \ge 1\) 
	such that \(\norm{\x^t - \x^\star} \leq \zeta ^ {\floor{\frac{t - 1}{n}}}\) for all \(t \ge 1\) and \(\{\zeta ^ {k}\}\) satisfies, \begin{align}\label{rate_iqn}
		\zeta ^ {k} \leq \epsilon \big(1 - \frac{\mu}{dL}\big) ^ {\frac{(k + 2)(k + 1)}{2}}. 
	\end{align}
\end{thm}

The proof of Theorem \ref{lem:final_rate} involves the construction of a sequence that provides an upper bound on the residual \(\|{\x}^t - {\x}^\star\|\), and can be found in Appendix \ref{app:final_rate}. 
\begin{rem}
	It is instructive to compare our convergence rate with that of IGS \cite[Theorem 3]{gao2020incremental}.~According to their result, the rate is given as
	\(
	\|{\x}^t - {\x}^\star\| \leq \left(1 - \frac{\mu}{dL}\right)^{\frac{k(k+1)}{2}} r^{k_0} ||{\x}^0 - {\x}^\star||
	\), 
	\(\forall t \geq 1\), \(r \in (0, 1), k = \left\lfloor\frac{t-1}{n}\right\rfloor + 1 - k_0\) and \(k_0\) is a constant such that \(\left(1 - \frac{\mu}{dL}\right)^{k_0}D \leq 1\).~The parameter \(D\) depends on the underlying objective function.
	Observe that their superlinear rate only takes effect after \(\left\lfloor\frac{t-1}{n}\right\rfloor \geq k_0 - 1\), and \(k_0\) could potentially be large, though it is not possible to infer the bounds on \(k_0\) from \cite{gao2020incremental}. 
	In contrast, our convergence rate guarantees superlinear convergence right from the first iteration.
\end{rem}

%% file: content/numerical_experiments.tex
\subsection{Quadratic Function Minimization}
We begin with a comparative analysis of the empirical performance of SLIQN, IQN, Sharpened BFGS (SBFGS) \cite{jin2022sharpened}, and IGS on a synthetic quadratic minimization problem. 

\textbf{Problem Definition}:~We consider the function \(f(\x) = \frac{1}{n} \sum_{i = 1}^{n} \big(\frac{1}{2} \ip{\x}{\A_i \x} + \ip{\bm{b}_i}{\x}\big),\) where \(\A_i \succ \mathbf{0}\), and \(\bm{b}_i \in \Rn ^ d\), \( \forall i \in [n]\). The detailed generation scheme for \(\A_i, \b_{i}\) can be found in Appendix \ref{app:numerical_simulations}.

\textbf{Experiments and Inference}:
We study the performance of the algorithms on two extreme cases: $d \gg n$ (Fig. \ref{synthetic_N_20_d_500_xi_2.png}), and $n \gg d$ (Fig. \ref{synthetic_N_50000_d_10_xi_1.png}). In each case, we plot the normalized error \(\norm{\x ^ t - \x^\star} / \norm{\x^0 - \x^\star}\) against the number of effective passes or epochs. We see that in both these cases, SLIQN outperforms IGS, IQN, and SBFGS. 
We also observe that in the case where $n \gg d$, IGS outperforms IQN, whereas in the case where $d \gg n$, IQN surpasses IGS. This is because IGS devotes the initial \(\cO(d)\) iterations to constructing a precise Hessian approximation, after which its fast convergence phase kicks in. 
On the other hand, since IQN takes the descent step in the Newton direction, its Hessian approximation is never precise and therefore its normalized error decreases at more or less a ``consistent" rate. 
SLIQN combines the strengths of both IQN and IGS: during the initial iterations when its Hessian approximation is not accurate enough, the classical BFGS updates are responsible for the progress made. In the later iterations, when the Hessian has been sufficiently well approximated, its fast convergence phase kicks in.

\begin{figure}
	\centering
	\begin{subfigure}[b]{0.37\textwidth}
		\includegraphics[height = 0.18\textheight, width=\textwidth]{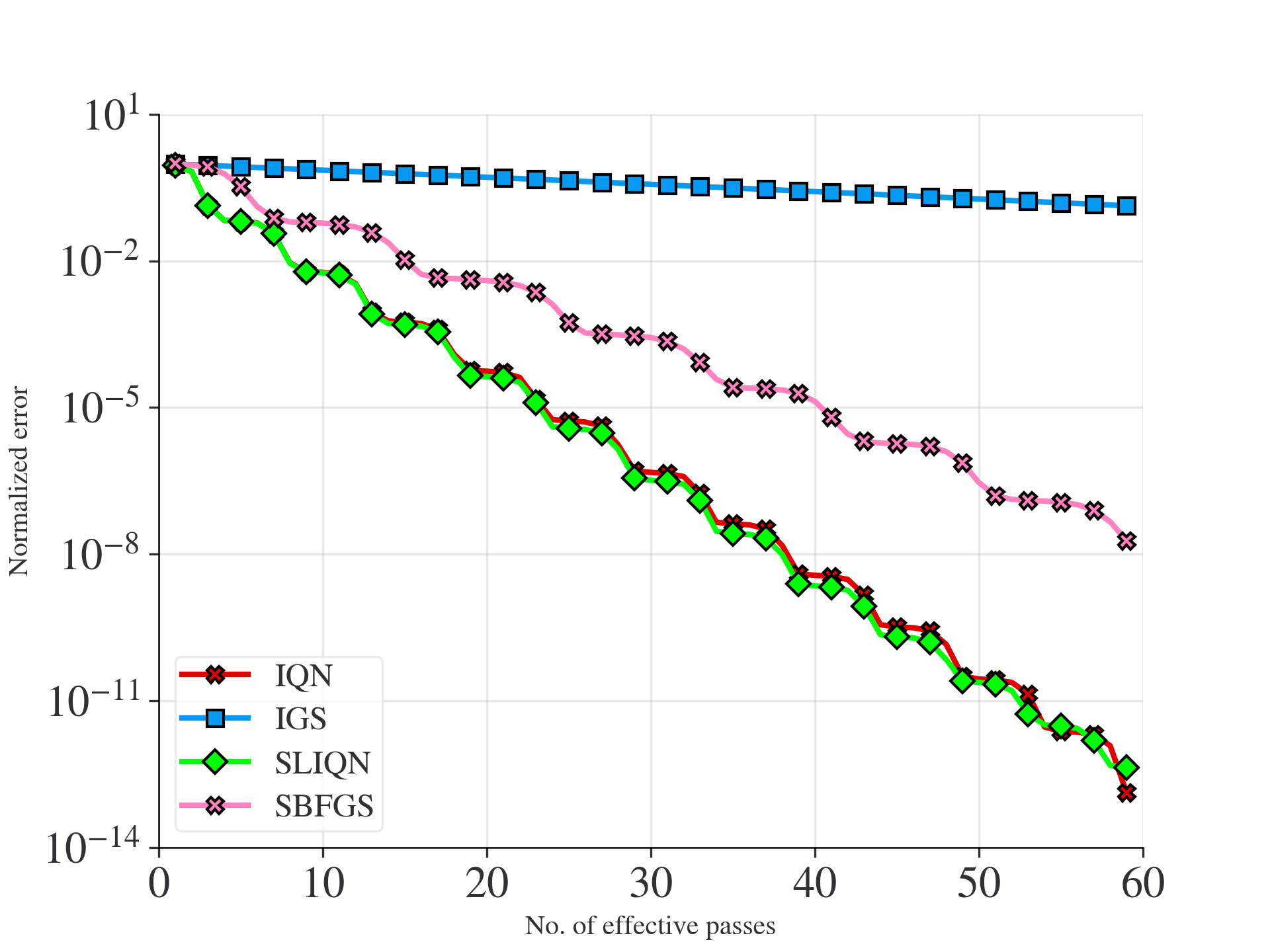}	
		\caption{\(n = 20, d = 500\)}
		\label{synthetic_N_20_d_500_xi_2.png}
	\end{subfigure}
	\begin{subfigure}[b]{0.37\textwidth}
		\includegraphics[height = 0.18\textheight, width=\textwidth]{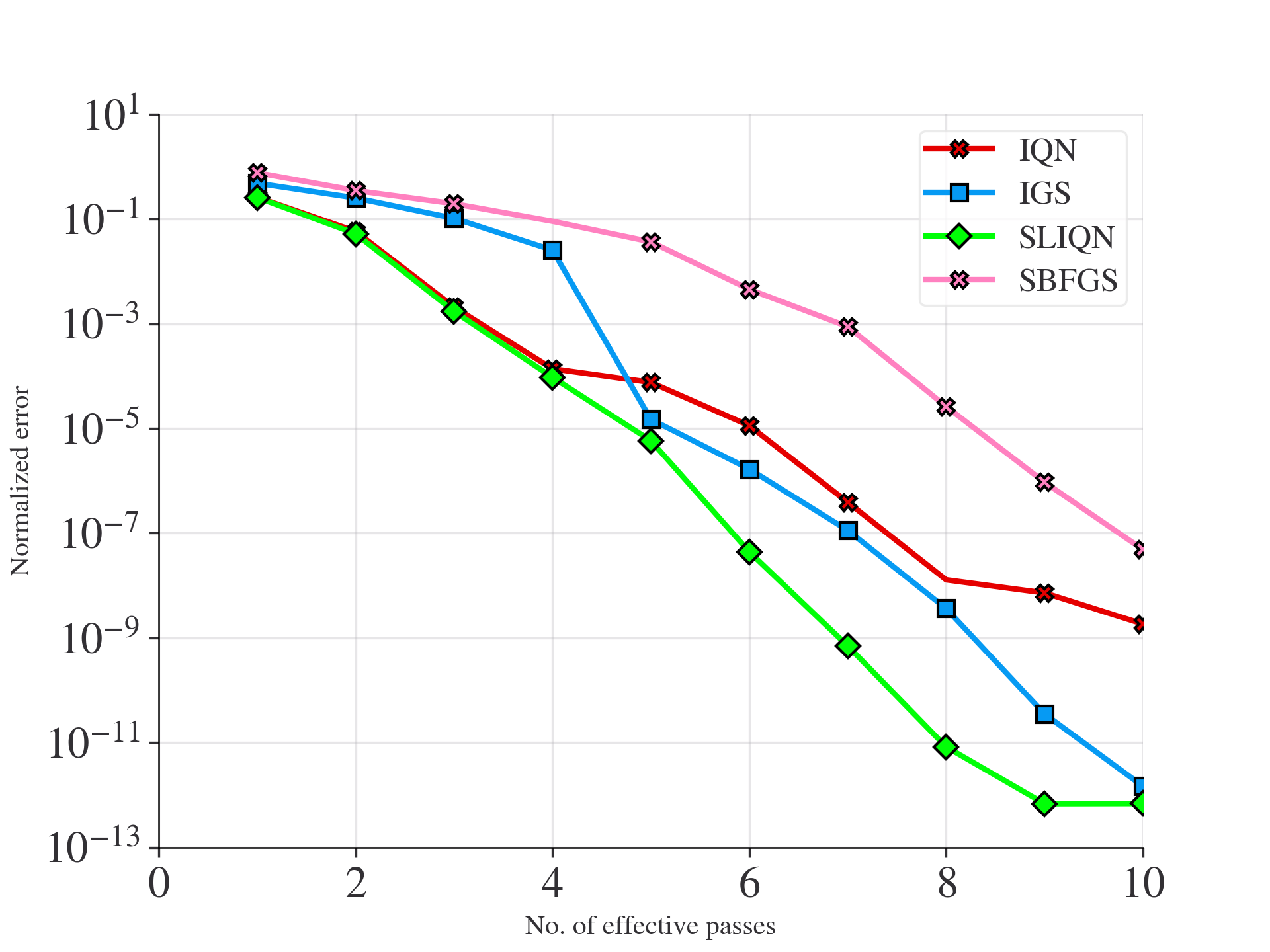}
		\caption{\(n = 50000, d = 10\)}	
		\label{synthetic_N_50000_d_10_xi_1.png}
	\end{subfigure}
	\caption{Normalized error vs. number of effective passes for the quadratic minimization problem}
\end{figure}

\begin{figure}[!htb]
\hspace{-2.25cm}
	\includegraphics[height=1.2\columnwidth, width=1.3\columnwidth]{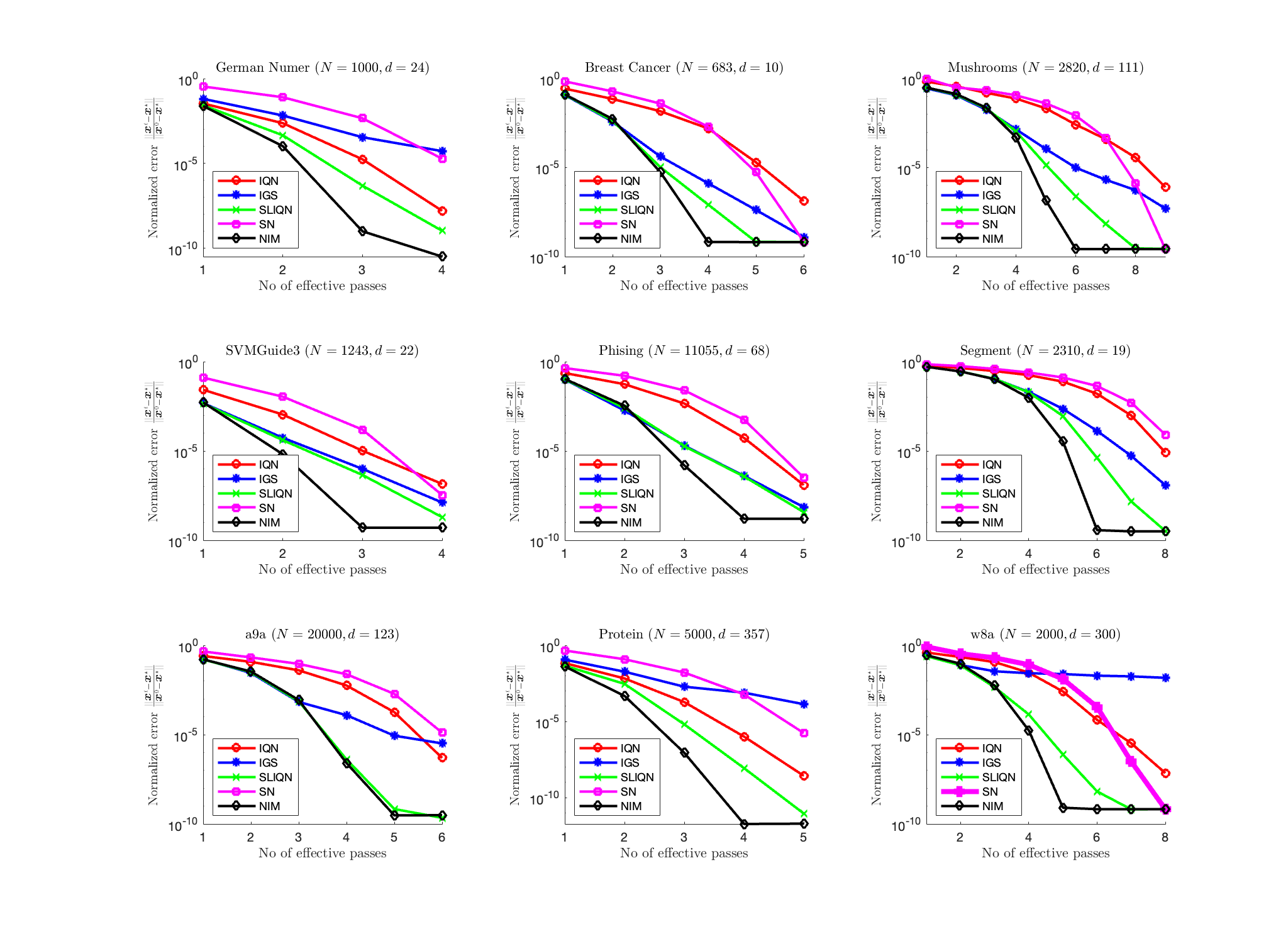}	
	\caption{Normalized error vs. number of effective passes for regularized logistic loss minimization}
	\label{all_combined_SN_edited_II.eps}
\end{figure}

\subsection{Regularized Logistic Regression}
We now compare the performance of SLIQN against IQN, IGS, NIM, and SN on the regularized logistic regression task given by  
\begin{align}
    \min_{\x}f(\x) &\coloneqq \frac{1}{N} \sum_{i = 1} ^ {N} \big(y_{i} \log(1 + e ^ {-\ip{\x}{\z_{i}}}) + (1 - y_{i}) \log(1 + e ^ {\ip{\x}{\z_{i}}}) \big) + \frac{\lambda}{2} \norm{\x} ^ {p},
\end{align}
where \(\{\z_{i}\}_{i = 1} ^ {N}\) are the training samples, 
\(\{y_{i}\}_{i = 1} ^ {N}\) are their corresponding binary labels
and $p$ is set at $2.1$. 
It is easy to observe that \(f(\x)\) is smooth, strongly convex and has a Lipschitz continuous Hessian, thereby satisfying Assumptions \ref{ass:smooth_strong}, \ref{ass:lip_hessian}. 
We compare the algorithms across \(9\) datasets with a large variation in the values of $n$ and $d$. 
Each algorithm is initialization with the same iterate and the regularization parameter is set as \(\lambda = {1}/{N}\). 
For SN, we set the mini-batch size \(\tau = N\) since that is the regime it works best in. Please refer to Appendix \ref{sec:numerical_simulations} for a complete experimental setup.
We observe in Figure \ref{all_combined_SN_edited_II.eps} that SLIQN outperforms IGS, IQN, and SN on each of the $9$ datasets from LIBSVM by \cite{libsvm}. This supports our claim that SLIQN offers the best of both, IQN and IGS.
Furthermore, we observe that while NIM outperforms SLIQN, their performance remains comparable. 
It is important to note that NIM utilizes the full Hessian information for the descent step and is an \(\mathcal{O}(d^3)\) algorithm, while SLIQN has a per-iteration complexity of \(\mathcal{O}(d^2)\).
Thus, these results underscore the superiority of SLIQN over other incremental QN style methods.

%% file: content/conclusion.tex
We introduced the SLIQN method for minimizing finite-sum problems. SLIQN enjoys the best known incremental rate of $\mathcal{O}((1-\tfrac{\mu}{dL})^{t^2/n^2})$,
has a $\cO(d ^ 2)$ per-iteration cost, an explicit superlinear convergence rate, and exhibits a superior empirical performance compared to several other incremental and stochastic QN methods. The key novelty in the proposed algorithm is the construction of modified update rules using a clever multiplicative factor and a lazy propagation strategy. We back up our empirical results with a comprehensive theory that explains the superior performance of SLIQN. The convergence rate of SLIQN is locally superlinear; analyzing the global convergence of the proposed algorithm remains as a future work.

%% file: appendices/restated_lemmas.tex
\begin{lemma}[Banach's Lemma]\label{lem:banach}
    		Let \(\A \in \mathbb{R}^{d \times d}\) be a matrix such that its norm     satisfies \(\norm{\A} < 1\). 
                Then the matrix \((\I + \A)\) is invertible and 
            \begin{align}\nn
				\frac{1}{1 + \norm{\A}} < \norm{(\I + \A) ^ {-1}} < \frac{1}{1 - \norm{\A}}.
		\end{align}
		\end{lemma}

		\begin{proposition}[Sherman-Morrison Formula]
			Let $\A \in \Rn^{d \times d}$ be an invertible matrix. Then, for all vectors $\u,\v \in \Rn^d$, we have
			\begin{align}
				(\A + \u\v^\T) ^ {-1} = \A ^ {-1} - \frac{\A ^ {-1}\u\v^\T\A^{-1}}{1 + \ip{\v}{\A^{-1}\u}}.\label{shermon}
			\end{align}
		\end{proposition}

		\begin{lemma}[Lemma 2.1, Lemma 2.2 \cite{rodomanov2021rates}]\label{boundedness}
			Consider positive definite matrices \(\A, \G \in \mathbb{R} ^ {d \times d}\) and suppose \(\G_{+} := \bfgs{\G, \A, \u}\), where \(\u \neq \zero \). Then, the following results hold:
			\begin{enumerate}
				\item For any constants \(\xi, \eta \geq 1\), we have \begin{align}\nn
					\frac{1}{\xi}\A \preceq \G \preceq \eta \A \implies \frac{1}{\xi} \preceq \G_{+} \preceq \eta \A.
				\end{align}
				\item If \(\A \preceq \G\), then we have \begin{align}\nn
					\sigma (\A, \G) \geq \sigma(\A, \G_{+}).
				\end{align}
			\end{enumerate}
		\end{lemma}

		\begin{lemma}[Lemma 4.2 \cite{rodomanov2021greedy}]\label{integral_lemma}
			Suppose an objective function \(f(\x)\) is strongly self-concordant with constant \(M > 0\). Consider \(\x, \y \in \mathbb{R}^{d}\), \(r := \norm{\y - \x}_{\x}\), and \(\K := \int_{0} ^ {1} \nabla ^ {2} f(\x + \tau(\y - \x)) d\tau\). Then, we have that \begin{align}\nn
				\frac{\nabla ^ {2} f(\x)}{1 + Mr} \preceq \nabla ^ {2} f(\y) \preceq (1 + Mr) \nabla ^ {2} f(\x), \\ \frac{\nabla ^ {2} f(\x)}{1 + \frac{Mr}{2}} \preceq \K \preceq (1 + \frac{Mr}{2}) \nabla ^ 2 f(\x),\nn \\
				\frac{\nabla ^ {2} f(\y)}{1 + \frac{Mr}{2}} \preceq \K \preceq (1 + \frac{Mr}{2}) \nabla ^ 2 f(\y)\nn.
			\end{align}
		\end{lemma}

		\begin{lemma}[Theorem 2.5 \cite{rodomanov2021greedy}]\label{one_greedy}
			Consider positive definite matrices \(\A, \G \in \mathbb{R} ^ {d \times d}\) such that \(\A \preceq \G\) and \(\mu \I \preceq \A \preceq L \I\) for constants \(\mu, L > 0\). Suppose \(\G_{+} := \bfgs{\G, \A, \bar{\u}(\G, \A)}\), where \(\bar{\u}(\G, \A)\) \eqref{greedy_vec} is the greedy vector of \(\G\) with respect to \(\A\). Then, the following holds:\begin{align}
				\sigma(\G_{+}, \A) \leq \big(1 - \frac{\mu}{d L}\big) \sigma(\G, \A).\nn
			\end{align}
		\end{lemma}

%% file: appendices/supporting_lemmas.tex
\begin{lemma}\label{lem:simple_lemma}
	For all positive definite matrices \(\A, \G \in \mathbb{R} ^ {d\times d}\),        if \(\A \preceq L\I\) and \(\A \preceq \G\), 
	then \begin{align}
		\norm{\G - \A} \leq L\sigma(\G, \A)\nn.
	\end{align}	
\end{lemma}
\begin{proof}
	For any positive definite matrix $\X \in \mathbb{R} ^ {d\times d}$, let $\lambda_{\max}(\X)$ denote its maximum eigenvalue and 
	let $\sum_{i = 1} ^ {d}\lambda_{i}(\X)$ denote the sum of all of its eigenvalues.
	We can bound $\tfrac{1}{L}\norm{\G - \A}$ as
	\begin{align} \nn
		\tfrac{1}{L}\norm{\G - \A} & \overset{\text{def}}{=} 
		\tfrac{1}{L}\lambda_{\max} (\G - \A)
		\le \nn
		\tfrac{1}{L} \sum_{i=1}^d \lambda_{i} (\G - \A)
		\overset{\text{def}}{=} \nn
		\tfrac{1}{L} \mathrm{Tr}(\G - \A).
	\end{align}
	Recall from the premise, we have \(\A \preceq L\I\), 
	which implies that \(\I \preceq L\A ^ {-1}\).
	Therefore,
	\begin{align*}
		\sigma(\G, \A)                    
		= \ip{\A ^ {-1}}{\G - \A}         
		\ge                               
		\tfrac{1}{L}                      
		\ip{\I}{\G - \A}                  
		=                                 
		\tfrac{1}{L} \mathrm{Tr}(\G - \A) 
		\ge                               
		\tfrac{1}{L}\norm{\G - \A}.       
	\end{align*}
	This completes the proof.
\end{proof}

\begin{lemma}\label{lem:unit_step_dec}
	Let \(f \colon \mathbb{R}^d \rightarrow \mathbb{R}\) be a real valued function
	that is \(\mu\)-strongly convex, \(L\)-smooth, and \(M\)-strongly self-concordant. 
	Let
	\(\x, \x_+ \in \mathbb{R}^d \backslash \{\mathbf{0}\}\) and 
	\(\B\) be a matrix such that \(\B \succeq \nabla ^ {2} f(\x)\).
	Define the constant \(r := \norm{\x_{+} - \x}_\x\) 
	and the matrix \({\bm{P}} := (1 + \tfrac{Mr}{2}) ^ 2\B\). 
	Consider the following BFGS updates: 
	\begin{align}\nn
		\Q     & := \bfgs{\P, (1 + \tfrac{Mr}{2})\K, \x_{+} - \x},                        \\
		\B_{+} & := \bfgs{\Q, \nabla ^ 2 f(\x_{+}), \bar{\u}(\Q, \nabla ^ {2} f(\x_+))}.\nn 
	\end{align} 
	Here, the  matrix 
	\(\K := \int_{0} ^ {1} \nabla ^ 2 f(\x + \tau (\x_{+} - \x)) d\tau\) and 
	the vector \(\bar{\u}(\Q, \nabla ^ 2 f(\x_+))\) is the greedy vector \ref{greedy_vec}. Then, \(\B_{+} \succeq \nabla ^ 2 f(\x_{+})\) and
	\begin{align}\nn
		\sigma(\B_+, \nabla ^ 2 f(\x_+)) \leq \big(1 - \tfrac{\mu}{dL}\big)\bigg(\big(1 + \tfrac{Mr}{2}\big) ^ 4 \sigma(\B, \nabla ^ 2 f(\x)) + d \big(1 + \tfrac{Mr}{2}\big) ^ 4 - d\bigg). 
	\end{align}
\end{lemma}
\begin{proof}
	We begin by analyzing the first BFGS update. Since \(\B \succeq \nabla ^ 2 f(\x)\), we have the following:
	\begin{align}\nn
		\B \succeq \nabla ^ 2 f(\x) \overset{\text{def}}{\implies} \P \succeq \big(1 + \frac{Mr}{2}\big) ^ 2 f(\x) \stackrel{\Lem. \ref{integral_lemma}}\succeq \big(1 + \frac{Mr}{2}\big)\K.
	\end{align}
	Since \(\P \succeq \big(1 + \frac{Mr}{2}\big)\K\), the metric \(\sigma\big(\P, (1 + \frac{Mr}{2})\K\big)\) is well defined. Applying Lemma \ref{boundedness}, we obtain \begin{align*}
		\Q = \bfgs{\P, \big(1 + \frac{Mr}{2}\big)\K, \x_{+} - \x}\succeq \big(1 + \frac{Mr}{2}\big)\K.
	\end{align*}
	Applying Lemma \ref{integral_lemma} to relate \(\K\) and \(\nabla ^ 2 f(\x_{+})\), we obtain \begin{align*}
		\Q \succeq \big(1 + \frac{Mr}{2}\big) \K \stackrel{\Lem. \ref{integral_lemma}}\succeq \nabla ^ 2 f(\x_{+}).
	\end{align*}
	We now begin analyzing the second BFGS update. Since \(\Q \succeq \nabla ^ 2 f(\x_{+})\), applying Lemma \ref{boundedness}, we obtain \begin{align}\nn
		\B_{+} = \bfgs{\Q, \nabla ^ 2 f(\x_{+}), \bar{\u}(\Q, \nabla ^ 2 f(\x_{+})} \succeq \nabla ^ 2 f(\x_{+}),
	\end{align}
	which completes the proof of the first part.
	Applying Lemma \ref{one_greedy}, we obtain
	\begin{align*}
		\sigma(\B_+, \nabla ^ 2 f(\x_+)) 
		\leq                             
		\left(                           
		1 - \tfrac{\mu}{dL}              
		\right)                          
		\sigma (\Q, \nabla ^ 2 f(\x_+)).  
	\end{align*}
	Define $c \coloneqq \left(
	1 - \tfrac{\mu}{dL}
	\right)$ for brevity.
	We are now ready to show the second result. Observe that
	\begin{align}
		\sigma(\B_+, \nabla ^ 2 f(\x_+)) 
		& \leq (1 - c)                                                                          
		\sigma (\Q, \nabla ^ 2 f(\x_+)), \nn 
		\\
		& = (1 - c) (\ip{\nabla ^ 2 f(\x_+) ^ {-1}}{\Q} - d) \nn,                               \\
		&\overset{(a)}{\leq} (1 - c) \big(\big(1 + \tfrac{Mr}{2}\big)\ip{\K ^ {-1}}{\Q} - d\big) \nn,         \\
		&\stackrel{(b)}{=} (1 - c) \big(\big(1 + \tfrac{Mr}{2}\big) ^ 2\ip{\tilde{\K} ^ {-1}}{\Q} - d\big), \label{continue_here}
	\end{align}
	where \((a)\) follows since
	$$
	\big(1 + \tfrac{Mr}{2}\big) \nabla ^ 2 f(\x_+)
	\stackrel{\Lem. \ref{integral_lemma}}\succeq \K
	\iff \nabla ^ 2 f(\x_{+}) ^ {-1} \preceq \big(1 + \frac{Mr}{2}\big) \K ^ {-1}.$$ In \((b)\), we have defined \(\tilde{\K} \coloneqq \big(1 + \frac{Mr}{2}\big)\K\). Recall that we have already established \(\P \succeq \tilde{\K}\). Applying Lemma \ref{boundedness}, we obtain \begin{align*}
		\sigma(\Q, \tilde{\K}) \leq \sigma(\P, \tilde{\K}) \iff \sigma(\Q, \K) \leq \sigma(\P, \K).
	\end{align*}
	% Recall that in the analysis of the first step, we had established that
	% \begin{align}
		% 	\sigma(\P, \tilde{\K}) \geq \sigma(\Q, \tilde{\K}) \implies \ip{\tilde{\K} ^ {-1}}{\P} \geq \ip{\tilde{\K} ^ {-1}}{\Q}. \nn 
		% \end{align}
	Continuing from \eqref{continue_here}, we obtain 
	\begin{align}
		\sigma(\B_+, \nabla ^ 2 f(\x_+)) & \leq (1 - c) \big(\big(1 + \tfrac{Mr}{2}\big) ^ 2\ip{\tilde{\K} ^ {-1}}{\P} - d\big)\nn,        \\
		& = (1 - c) \big(\big(1 + \tfrac{Mr}{2}\big) ^ 4\ip{\tilde{\K} ^ {-1}}{\B} - d\big)\nn,           \\
		& = (1 - c)\big(\big(1 + \tfrac{Mr}{2}\big) ^ 3\ip{{\K} ^ {-1}}{\B} - d\big)\nn,                  \\
		& \overset{(a)}{\leq} (1 - c)\big(\big(1 + \tfrac{Mr}{2}\big) ^ 4\ip{{\nabla ^ 2 f(\x)} ^ {-1}}{\B} - d\big), 
	\end{align}
	where \((a)\) follows from 
	$
	\K ^ {-1}
	\preceq \big(1 + \tfrac{Mr}{2}\big) \nabla ^ 2 f(\x) ^ {-1}
	$ (\(\because\) Lemma \ref{integral_lemma}). Finally, setting \(\ip{{\nabla ^ 2 f(\x)} ^ {-1}}{\B} = \sigma(\B, \nabla ^ 2 f(\x)) + d\) completes the proof.
\end{proof}

\begin{corollary}\label{cor:imp}
	Under the notation established in Lemma \ref{lem:unit_step_dec}, let
	\(\alpha \in \mathbb{R}_+\) be an upper bound on \(r\). Then, we have the following:\begin{align}\nn
		\sigma(\B_+, \nabla ^ 2 f(\x_+)) \leq (1 - \tfrac{\mu}{dL})e ^ {2M\alpha}\big( \sigma(\B, \nabla ^ 2 f(\x)) + 2Md \alpha\big). 
	\end{align}
\end{corollary}
\begin{proof}
	
	From Lemma \ref{lem:unit_step_dec}, we have \begin{align}
		\sigma(\B_+, \nabla ^ 2 f(\x_+)) &\leq \big(1 - \tfrac{\mu}{dL}\big)\big(1 + \tfrac{M\alpha}{2}\big) ^ 4\big( \sigma(\B, \nabla ^ 2 f(\x)) + d\big(1 - \tfrac{1}{\big(1 + \tfrac{M\alpha}{2}\big)^4})) \nn, \\
		&\stackrel{(a)}\leq \big(1 - \tfrac{\mu}{dL}\big)e ^ {2M\alpha}\big( \sigma(\B, \nabla ^ 2 f(\x)) + d(1 - e ^ {-2M\alpha})\big)\nn, \\
		&\stackrel{(b)}\leq \big(1 - \tfrac{\mu}{dL}\big)e ^ {2M\alpha}\big( \sigma(\B, \nabla ^ 2 f(\x)) + 2Md\alpha\big)\nn,
	\end{align}
	where \((a)\) follows by the inequality \(1 + x \leq e ^ {x}, \forall x \in \mathbb{R}\) and \((b)\) follows from the inequality \(1 - e ^ {-x} \leq x, \forall x > 0\). This completes the proof.
\end{proof}

\begin{lemma}\label{lem:sig}
	Let \(f \colon \mathbb{R}^d \rightarrow \mathbb{R}\) be a real valued function
	that is \(\mu\)-strongly convex, \(L\)-smooth, and \(M\)-strongly self-concordant.
	Let $\tilde{\x} \in \mathbb{R}^d$ be some fixed vector and $0 \le \gamma < 1$ be some fixed constant such that the sequence \(\{\x^{k}\}\), for all \(k \in [T]\) satisfies\begin{align}\nn
		\norm{\x^{k} - \tilde{\x}} \leq {\gamma} ^ {k} \norm{\x^{0} - \tilde{\x}}.
	\end{align}
	Define the constant \(r_{k} := \norm{\x^{k} -\x^{k - 1}}_{\x^{k - 1}}\) for every $k$. 
	Let \(\B^{0}\) be a matrix such that it satisfies \(\B^0 \succeq \nabla ^ 2 f(\x^0)\). Consider the following BFGS
	updates: \begin{align}\Q^k &:= \bfgs{\P^{k - 1}, (1 + \tfrac{Mr_k}{2})\K^k, \x^{k} - \x^{k - 1}},\nn \\
		\B^k &:= \bfgs{\Q ^ k, \nabla ^ 2 f(\x^{k}), \bar{\u}(\Q^k, \nabla ^ {2} f(\x^k))},\nn\end{align} where \(\P^{k - 1} := \big(1 + \tfrac{Mr_k}{2}\big) ^ 2 \B^{k - 1}\), \(\K^k := \int_{0} ^ {1} \nabla ^ 2 f(\x^{k - 1} + \tau (\x^{k} - \x^{k - 1})) d\tau\), and \(\bar{\u}(\Q ^ k, \nabla ^ 2 f(\x ^ k))\) is the greedy vector \ref{greedy_vec}. Then, the following holds for all \(k \in [T] \):\begin{align}
		\sigma(\B^{k}, \nabla ^ 2 f(\x ^ k)) \leq (1 - \tfrac{\mu}{dL}) ^ {k} e ^ {\tfrac{4M\sqrt{L}||{\x^0 - \tilde{\x}}||}{1 - \gamma}} \bigg(\sigma(\B^{0},\nabla ^ 2 f(\x^0)) +  \norm{\x^0 - \tilde{\x}} \tfrac{4Md\sqrt{L}}{1 - (1 - \tfrac{\mu}{dL})^{-1}\gamma}\bigg).
	\end{align}
\end{lemma}
\begin{proof}
	
	From Lemma \ref{lem:unit_step_dec}, it can be shown that \(\B^k \succeq \nabla ^ 2 f(\x^k)\) for \(k = 1, \dots, T\). Therefore, \(\sigma(\B^k, \nabla ^ 2 f(\x^k))\) is well defined. We introduce the notation \(\sigma_k := \sigma(\B^{k}, \nabla ^ 2 f(\x^k))\), \(d_{k} := \norm{\x^k - \tilde{\x}}\), and \(c := \tfrac{\mu}{dL}\) for simplicity. 
	
	To apply Corollary \ref{cor:imp}, we need an upper bound on \(r_{k}\). This is trivial via the Triangle inequality \begin{align}
		r_{k} = \norm{\x^{k} - \x^{k - 1}}_{\x^{k - 1}} &\stackrel{\Delta}\leq \norm{\x^{k} - \tilde{\x}}_{\x^{k - 1}} + \norm{\x^{k - 1} - \tilde{\x}}_{\x^{k - 1}} \nn \\
		&\stackrel{(a)}\leq \sqrt{L}\big(\norm{\x^{k} -\tilde{\x}} + \norm{\x^{k - 1} - \tilde{\x}}\big) \leq 2\sqrt{L}\gamma ^ {k - 1} \norm{\x^{0}- \tilde{\x}}, \nn
	\end{align}
	where \((a)\) follows since \(\nabla ^ 2 f(\x) \preceq LI\). Therefore \(\alpha_{k} = 2\sqrt{L} \gamma ^ {k - 1}\norm{\x^{0}- \tilde{\x}}\) is an upper bound on \(r_{k}\). Applying Corollary \ref{cor:imp}, we obtain \begin{align}
		\sigma_{k} &\leq (1 - c) e ^ {2M\alpha_k} (\sigma_{k - 1} + 2Md\alpha_{k}) \nn \\
		&\leq (1 - c) e ^ {4M\sqrt{L}d_0\gamma ^ {k - 1}} (\sigma_{k - 1} + 4Md\sqrt{L}d_0\gamma ^ {k - 1})\label{recursion}.
	\end{align}
	We solve the recursion \eqref{recursion} as follows:\begin{align}
		\sigma_{k} & \leq (1 - c) e ^ {4M\sqrt{L}d_0\gamma ^ {k - 1}} (\sigma_{k - 1} + 4Md\sqrt{L}d_0\gamma ^  { k -1}),\nn \\
		& \leq (1 - c) ^ 2 e ^ {4M\sqrt{L}d_0(\gamma ^ {k - 1} + \gamma ^ {k - 2})}\sigma_{k - 2} + 4Md\sqrt{L}d_0 \big(\gamma ^ {k - 2} (1 - c) ^ 2e ^ {4M\sqrt{L}d_0 (\gamma ^ {k - 1} + \gamma ^ {k - 2})} \nn \\
		&\hspace{80mm} + \gamma ^ {k - 1}(1 - c) e ^ {4M\sqrt{L}d_0\gamma ^ {k - 1}}\big),\nn\\ 
		& \leq (1 - c) ^ {k} e^ {4M\sqrt{L}d_0 \sum_{j = 0} ^ {k - 1} \gamma ^ j} \sigma_0 + 4Md\sqrt{L}d_0 \big(\sum_{j = 0} ^ {k - 1} \gamma ^ {j}(1 - c) ^ {k - j}e ^ {4M\sqrt{L}d_0 \sum_{i = 1} ^ {k  - j}\gamma ^ {k - i}}\big),\nn \\
		&\leq 	(1 - c) ^ {k} e^ {4M\sqrt{L}d_0 \sum_{j = 0} ^ {\infty} \gamma ^ j}\sigma_0 + 4Md\sqrt{L}d_0 e^{4M\sqrt{L}d_0 \sum_{i = 0} ^ {\infty} \gamma ^ {i}} \sum_{j = 0} ^ {k - 1} \gamma ^ {j} (1 - c) ^ {k - j}\nn, \\
		& \leq (1 - c) ^ {k} e ^ {\tfrac{4M\sqrt{L}d_0}{1 - \gamma}} \big(\sigma_0 + \tfrac{4Md\sqrt{L}d_0}{1 - \tfrac{\gamma}{1 - c}}\big)\nn.
	\end{align}
	This completes the proof.
\end{proof}

\begin{rem}\label{change_r}
	It can be concluded from the proof of Lemma \ref{lem:sig} that redefining \(r_k := 2\sqrt{L}\gamma ^ {k - 1}d_0\) (which is an upper bound on \(\norm{\x ^ k - \x ^ {k - 1}}_{\x ^ {k - 1}}\)), the results of Lemma \ref{lem:sig} remain unchanged.
\end{rem}

% \subsection{Proof of Lemma \ref{lem:banach}}\label{app:banach}
% \input{Proof_Banach_Lemma.tex}
% \subsection{Proof of Lemma \ref{lem:simple_lemma}}\label{app:simple_lemma}
% \input{Proof_simple_lemma.tex}
% \subsection{Proof of Lemma \ref{lem:unit_step_dec}}\label{app:unit_step_dec}
% \input{Proof_lemma_sigma_unit_dec.tex}
% \subsection{Proof of Corollary \ref{cor:imp}} \label{app:imp}
% \input{Proof_corollary_sigma_rel.tex}
% \subsection{Proof of Lemma \ref{lem:sig}} \label{app:sig}
% \input{Proof_lemma_sigma.tex}

%% file: appendices/efficient_implementation.tex
Recall that, that at time \(t\), IQN updates \begin{align}\nn
	\B_{i_t} ^ t = \bfgs{\B_{i_t} ^ {t - 1}, \K ^ {t}, \z_{i_t} ^ t - \z_{i_t} ^ {t - 1}}.
\end{align}

Define \(\bi{\phi} ^ {t} := \sum_{i = 1} ^ {n} \B_{i} ^ {t}\z_{i} ^ {t} - \sum_{i = 1} ^ {n} \nabla f_{i}(\z_{i} ^ {t})\). At time \(t\), since IQN only updates the tuple with index \(i_t\), we have \begin{align}\label{eff_update}
	\bi{\phi} ^ t = \bi{\phi} ^ {t - 1} + \big(\B_{i_t} ^ {t}\z_{i_t} ^ t - \B_{i_t} ^ {t - 1}\z_{i_t} ^ {t - 1}\big) - \big(\nabla f_{i_t}(\z_{i_t} ^ {t}) - \nabla f_{i_t}(\z_{i_t} ^ {t - 1})\big).
\end{align}
Therefore, given access to \(\bi{\phi} ^ {t - 1}\), we can compute \(\bi{\phi} ^ {t}\) in \(\cO(d ^ 2)\) time. This updating scheme can be implemented iteratively, where we only evaluate \(\bi{\phi} ^ 0\) explicitly and evaluate \(\bi{\phi} ^ t\), for all \(t \geq 1\) by \eqref{eff_update}. It only remains to compute \(\big(\bar{\B} ^ t\big) ^ {-1}\), where \(\bar{\B} ^ {t} := \big(\sum_{i = 1} ^ {n} \B_{i} ^ {t}\big)\). This can be done by applying the Sherman-Morrison formula \eqref{shermon} twice to the matrix on the right \eqref{B_update_IQN}. \begin{align}\label{B_update_IQN}
	\bar{\B} ^ t = \bar{\B} ^ {t - 1} + \B_{i_t} ^ t - \B_{i_t} ^ {t - 1} = \bar{\B} ^ {t - 1} + \frac{{\y_{i_t} ^ t}{\y_{i_t} ^ t} ^ \T}{\ip{{\y_{i_t} ^ t}}{{\s_{i_t} ^ t}}} - \frac{\B_{i_t} ^ {t - 1}{\s_{i_t} ^ t}{\s_{i_t} ^ t} ^ {\T}\B_{i_t} ^ {t - 1}}{\ip{{\s_{i_t} ^ t}}{\B_{i_t} ^ {t - 1}{\s_{i_t} ^ t}}},
\end{align}
where \(\y_{i_t} ^ t = \z_{i_t} ^ t - \z_{i_t} ^ {t - 1}, \s_{i_t} ^ t = \nabla f_{i_t}(\z_{i_t} ^ t) - \nabla f_{i_t}(\z_{i_t} ^ {t - 1})\).
Applying\eqref{shermon} twice, we get \eqref{first_step} and \eqref{second_step}
\begin{align}\label{first_step}
	(\bar{\B} ^ {t}) ^ {-1} = {\Z} ^ {t} + \frac{{\Z} ^ t (\B_{i_t} ^ {t - 1}{\s_{i_t} ^ t})(\B_{i_t} ^ {t - 1}{\s_{i_t} ^ t}) ^ {\T}{\Z} ^ t}{\ip{{\s_{i_t} ^ t}}{\B_{i_t} ^ {t - 1}{\s_{i_t} ^ t}} - \ip{(\B_{i_t} ^ {t - 1}{\s_{i_t} ^ t})}{\Z ^ t(\B_{i_t} ^ {t - 1}{\s_{i_t} ^ t})}},
\end{align}
where \(\Z^t\) is given by \begin{align}\label{second_step}
	\Z ^ t = (\bar{\B} ^ {t - 1}) ^ {-1} -\frac{(\bar{\B} ^ {t - 1}) ^ {-1}{\y_{i_t} ^ t}{\y_{i_t} ^ t} ^ \T(\bar{\B} ^ {t - 1}) ^ {-1}}{\ip{{\y_{i_t} ^ t}}{{\s_{i_t} ^ t}} + \ip{{\y_{i_t} ^ t}}{(\bar{\B} ^ {t - 1}) ^ {-1}{\y_{i_t} ^ t}}}. 
\end{align}

By iteratively implementing this scheme, we only need to compute \((\bar{\B} ^ {0}) ^ {-1}\) explicitly and \((\bar{\B} ^ {t}) ^ {-1}\), for all \(t \geq 1\) by the Sherman-Morrison formula. 

%% file: appendices/low_memory_implementation.tex
In this section, we illustrate how incremental methods can be effectively implemented for large-scale real-world scenarios which are characterized by substantial memory requirements of $\mathcal{O}(nd^2)$.
Our solution leverages the disk, which offers significantly larger storage capacity compared to main memory but comes with an increased load-store latency. To mitigate this latency issue, we employ a pipelining scheme. In this scheme, we partition the data into blocks and simultaneously run compute operations on one block while performing load-store operations on the blocks adjacent to it. This parallelization effectively extends the main memory capacity to the available the disk size, all the while avoiding its larger latency.

Formally, let the available main memory capacity be $g$ GB, the number of data samples be $n$, the dimensionality of the data be $d$, and the space requirement for each sample be $s$. We assume that the disk is sufficiently large to store the data for all samples, that it the size of the disk is greater than $ns$.
We divide the data into $m=\frac{2ns}{g}$ blocks, denoted as $b_i$ for $i\in[m]$. This choice of $m$ ensures that two blocks can be accommodated in memory simultaneously.
The processing proceeds as follows:

\begin{enumerate}[itemsep=2pt,parsep=3pt,topsep=3pt]
    \item We assume that the memory holds blocks $b_1$ and $b_2$, along with the global memoized quantities for SLIQN. All data blocks are also stored on the disk.
    \item At iteration $t=1$, we process the block $b_1$ by executing the corresponding algorithm updates on it.
    \item At any iteration $t > 1$, we execute the algorithm updates on $b_{i\%n+1}$ while concurrently storing the already processed block $b_i$ back into the disk and loading the block $b_{(i+1)\%n+1}$ into memory to be processed next.
\end{enumerate}

In practice, modern disks have access speeds of around 500 MBps, making processing the bottleneck in this parallel architecture, rather than disk access. For example, in our implementation of SLIQN with $g=1200$ MB, $n=20,000$, $d=123$, $s=0.1117$ MB, and $m=4$, we observed that processing one block took 7.8 seconds, while the load-store operation required only 3.8 seconds.

%% file: appendices/convergence_analysis_of_SIQN.tex
In this section, we provide the convergence analysis of SIQN, as the convergence analysis motivates the replacement of \(\beta_t\) with \(\alpha_{\ceil{\frac{t}{n} - 1}}\). We begin by showing that at each time \(t\), the matrix \(\Q ^ t\) obtained after the first BFGS update satisfies \(\Q ^ t \succeq  \nabla ^ 2 f_{i_t}(\z_{i_t} ^ t)\). Using this, we show that the updated Hessian approximation satisfies \(\B_{i_t} ^ t \succeq \nabla ^ 2 f_{i_t}(\z_{i_t} ^ t)\). These observations are essential in order ensure that \(\sigma(\Q ^ t, \nabla ^ 2 f_{i_t}(\z_{i_t} ^ t))\) and  \(\sigma(\B ^ t, \nabla ^ 2 f_{i_t}(\z_{i_t} ^ t))\) are well defined.
\begin{lemma}\label{lem:formalize_psdness}
	For all \(t \ge 1\) the following hold true: \begin{align}
		\label{Q_lem}\Q ^ t \succeq \nabla ^ 2 f_{i_t}(\z_{i_t} ^ t), \\
		\label{B_lem}\B_{i} ^ t \succeq \nabla ^ 2 f_{i}(\z_{i} ^ t),
	\end{align}
	for all \(i \in [n]\).
\end{lemma}
\begin{proof}
	The proof follows by induction on \(t\). 
	
	\textbf{{Base case:}} At \(t = 0\), \eqref{B_lem} holds due to the initialization.
	
	\textbf{{Induction Hypothesis (IH):}} Assume that \eqref{Q_lem} and \eqref{B_lem} hold for \(t = m - 1\), for some \(m \ge 1\). We prove that \eqref{Q_lem} and \eqref{B_lem} hold for \(t = m\) in the Induction step.

	\textbf{{Induction step:}}
	Since \(\B_{i_m} ^ {m - 1} \succeq \nabla ^ 2 f_{i_m}(\z_{i_m} ^ {m - 1})\), we have the following:
	\begin{align*}
		\big(1 + \beta_m\big) ^ 2 \B_{i_m} ^ {m - 1} \succeq \big(1 + \beta_m\big) ^ 2 \nabla ^ 2 f_{i_m}(\z_{i_m} ^ {m - 1}) \stackrel{\Lem. \ref{integral_lemma}}\succeq \big(1 + \beta_m\big)\K ^ m,
	\end{align*}
	where recall that \(\beta_{m} = \tfrac{M}{2}\norm{\z_{i_m} ^ m - \z_{i_m} ^ {m - 1}}_{\z_{i_m} ^ {m - 1}}\). 
	Applying Lemma \ref{boundedness}, we obtain \begin{align*}
		\Q ^ m = \bfgs{\big(1 + \beta_m\big) ^ 2 \B_{i_m} ^ {m - 1}, \big(1 + \beta_m\big)\K ^ m, \z_{i_m} ^ m - \z_{i_m} ^ {m - 1}}\stackrel{\Lem. \ref{boundedness}}\succeq \big(1 + \beta_m\big)\K ^ m.
	\end{align*}
	Applying Lemma \ref{integral_lemma} to relate \(\K ^ m\) and \(\nabla ^ 2 f_{i_m}(\z_{i_m} ^ {m})\), we obtain \begin{align*}
		\Q ^ m \succeq \big(1 + \beta_m\big) \K ^ m \stackrel{\Lem. \ref{integral_lemma}}\succeq \nabla ^ 2 f_{i_m}(\z_{i_m} ^ m).
	\end{align*}
	Therefore, \eqref{Q_lem} holds for \(t = m\). 
	Since \(\Q ^ m \succeq \nabla ^ 2 f_{i_m}(\z_{i_m} ^ m)\), applying Lemma \ref{boundedness}, we obtain \begin{align}\nn
		\B_{i_m} ^ m = \bfgs{\Q ^ m, \nabla ^ 2 f_{i_m}(\z_{i_m} ^ m), \bar{\u}(\Q ^ m, \nabla ^ 2 f_{i_m}(\z_{i_m} ^ m)} \stackrel{\Lem. \ref{boundedness}}\succeq \nabla ^ 2 f_{i_m}(\z_{i_m} ^ m).
	\end{align}
	Therefore, \eqref{B_lem} holds for \(t = m\). This completes the induction step. The proof is hence complete by induction.
\end{proof}
%		\begin{lemma}\label{lem:psdness}
	%			For all \(t \ge 0\) the following hold true: \begin{align}\label{sigma_well_defined}\B_{i} ^ t \succeq \nabla ^ 2 f_{i}(\z_{i} ^ t),\end{align} for all \(i \in [n]\). 
	%		\end{lemma}
%		\begin{proof}
	%			We prove \eqref{sigma_well_defined} by induction on \(t\) as follows: 
	%			
	%			\textbf{\textit{Base case:}} At \(t = 0\), \eqref{sigma_well_defined} clearly holds because of the initialization.
	%			\textbf{\textit{Induction Hypothesis:}} Assume that \eqref{sigma_well_defined} holds for some \(t \geq 1\). 	\textbf{\textit{Induction Step:}}  Applying Lemma \ref{lem:sig} with \(\x = \z_{i_{t + 1}} ^ {t}, \x_{+} = \z_{i_{t + 1}} ^ {t + 1}, \B = \B_{i_{t + 1}} ^ {t}\), we have \(\B_{i_{t + 1}} ^ {t + 1} \succeq \nabla ^ 2 f_{i_{t + 1}}(\z_{i_{t + 1}} ^ {t + 1})\). Therefore, \eqref{sigma_well_defined} holds for the index \(i = i_{t + 1}\). For any \(i \in [n]\) s.t. \(i \neq i_{t + 1}\), \(\z_{i} ^ {t + 1} = \z_{i} ^ {t}, \B_{i} ^ {t + 1} = \B_{i} ^ {t}\) and therefore \eqref{sigma_well_defined} holds by the induction hypothesis. Therefore, \eqref{sigma_well_defined} holds for \(t + 1\). This completes the proof.
	%		\end{proof}

\textbf{Key steps:} We establish the convergence guarantees in three steps: 
Lemma \ref{lem:distance_change_single_iter} establishes a one-step inequality that  bounds the residual \(\norm{\x^{t} - \x^{\star}}\) in terms of the previous residuals \(\norm{\z_{i} ^ {t - 1} - \x^\star}\) and the norm error in the Hessian approximation \(\norm{\B_{i} ^ {t - 1} - \nabla ^ 2 f_{i}(\z_{i} ^ {t - 1})}\), for all $i \in [n]$. 
% Note that the residuals \(\norm{\z_{i} ^ {t - 1} - \x^\star}\) for \(i \in [n]\) are related to the previous $n$ residuals \(\{\norm{\x ^ {t - j} - \x^\star}\}_{j = 1} ^ {n}\). 
Lemma \ref{lem:convergence_h_iqn} 
uses the result of Lemma \ref{lem:distance_change_single_iter} to inductively show that both the residual $\norm{\x^t-\x^\star}$ and the Hessian approximation error \(\sigma(\B_{i_t} ^ {t}, \nabla ^ 2 f_{i_t}(\z_{i_t} ^ t))\),  decrease linearly with $\ceil{t/n}$. 
Using the result of Lemma \ref{lem:distance_change_single_iter} and Lemma \ref{lem:convergence_h_iqn}, Lemma \ref{lem:mean_linear_ver_I} establishes a mean-superlinear convergence result. 
We finally show in Theorem \ref{lem:final_rate} that the residuals can be upper bounded by a superlinearly convergent sequence.

We now present our one-step inequality. A similar inequality with \(\B_{i} ^ {t - 1}\) replaced with \(\D_{i} ^ {t - 1}\) also appears in Lemma \ref{lem:distance_change_single_iter_ver_II} (for SLIQN). Since the proofs are identical, we refer Appendix \ref{app:distance_change_single_iter_ver_II} directly for the proof of Lemma \ref{lem:distance_change_single_iter}.
\begin{lemma}\label{lem:distance_change_single_iter}
	If Assumptions \ref{ass:smooth_strong} and \ref{ass:lip_hessian} hold, the sequence of iterates generated by SIQN satisfy 
	\begin{align}\label{lem1eq}
		\norm{\x^{t} - \x^\star} \leq {\tfrac{\tilde{L}\Gamma ^ {t - 1}}{2}} \sum_{i = 1} ^ {n} \norm{\z_{i} ^ {t - 1} - \x^\star} ^ 2 + {\Gamma ^ {t - 1}} \sum_{i = 1} ^ {n} \norm{\B_i ^ {t - 1} - \nabla ^ 2f_i(\z_{i} ^ {t - 1})}\norm{\z_{i}^{t - 1} - \x^\star}, 
	\end{align} 
	for all \(t \geq 1\), where \(\Gamma ^ {t} := \norm{\big(\sum_{i = 1} ^ {n} \B_{i} ^ {t}\big)^{-1}}\).
\end{lemma}
\begin{proof}
	Refer Appendix \ref{app:distance_change_single_iter_ver_II}.
\end{proof}
\begin{lemma}\label{lem:convergence_h_iqn}
	If Assumptions \ref{ass:smooth_strong} and \ref{ass:lip_hessian} hold, for any \(\rho\) such that \(0 < \rho < 1 - \tfrac{\mu}{dL}\), there exist positive constants \(\epsilon ^ {\iqn}\) and \(\sigma_0 ^ {\iqn}\) such that if \(\norm{\x ^ {0} - \x^{\star}} \leq \epsilon^{\iqn}\) and \(\sigma(\B_i ^ {0}, \nabla ^ 2 f_i(\x ^ {0})) \leq \sigma_0^{\iqn}\) for all \(i \in [n]\), the sequence of iterates generated by SIQN satisfy \begin{align}\label{contraction_siqn}
		\norm{\x ^ {t} - \x^\star} \leq \rho ^ {\ceil{\tfrac{t}{n}}} \norm{\x ^ {0} - \x^\star}.
	\end{align}
	Further, it holds that\begin{align}\label{sigma_rel_siqn}
		\sigma(\B_{i_t}^{t}, \nabla ^ 2 f_{i_t}(\z_{i_t}^t)) \leq \big(1 - \tfrac{\mu}{dL}\big) ^ {\ceil{\tfrac{t}{n}}} {\delta ^ {\iqn}},
	\end{align}
	where \({\delta ^ {\iqn}} :=  e ^ {\tfrac{4M\sqrt{L}\epsilon^{\iqn}}{1 - \rho}} \big(\sigma_0^{\iqn} + \epsilon^{\iqn} \tfrac{4Md\sqrt{L}}{1 - \tfrac{\rho}{1 - \tfrac{\mu}{dL}}}\big), M = \tilde{L} / {\mu ^ {\tfrac{3}{2}}}\).
\end{lemma}
\input{appendices/Proof_linear_hybrid_IQN.tex}
Next, we present our mean superlinear convergence result for the iterates of SIQN. The main idea behind the proof is to substitute the linear convergence results, specifically \eqref{contraction_siqn} and \eqref{sigma_rel_siqn} from Lemma \ref{lem:convergence_h_iqn}, back into the result from Lemma \ref{lem:distance_change_single_iter}.
By doing so, the first term on the right-hand side of \eqref{lem1eq} converges quadratically, while the second term converges superlinearly. This combination leads to the desired result.
\begin{lemma}\label{lem:mean_linear_ver_I}
	If Assumptions \ref{ass:smooth_strong} and \ref{ass:lip_hessian} hold, for any \(\rho\) such that \(0 < \rho < 1 - \tfrac{\mu}{dL}\), there exist positive constants \(\epsilon ^ {\iqn}\) and \(\sigma_0 ^ {\iqn}\) such that if \(\norm{\x ^ {0} - \x^{\star}} \leq \epsilon^{\iqn}\) and \(\sigma(\B_i ^ {0}, \nabla ^ 2 f_i(\x ^ {0})) \leq \sigma_0^{\iqn}\) for all \(i \in [n]\), the sequence of iterates produced by SIQN satisfy \begin{align}\nn
		\norm{\x^{t} - \x^\star} \leq \big(1 - \tfrac{\mu}{dL}\big) ^ {\ceil{\tfrac{t}{n}}} \tfrac{1}{n}\sum_{i = 1} ^ n \norm{\x^{t - i} - \x^\star}. 
	\end{align}
\end{lemma}
\input{appendices/Proof_mean_linear_hybrid_IQN.tex}

The mean superlinear convergence result of Lemma \ref{lem:mean_linear_ver_I} ultimately gives a superlinear rate for SIQN. Note that an identical result as Lemma \ref{lem:mean_linear_ver_I} is given by Theorem \ref{lem:final_rate}(SLIQN). Therefore, we directly provide the proof of Lemma \ref{lem:mean_linear_ver_I} in Appendix \ref{app:final_rate} while proving Theorem \ref{lem:final_rate} for SLIQN.
\begin{lemma}\label{lem:final_rate_siqn}
	If Assumptions \ref{ass:smooth_strong} and \ref{ass:lip_hessian} hold, for any \(\rho\) such that \(0 < \rho < 1 - \tfrac{\mu}{dL}\), there exist positive constants \(\epsilon ^ {\iqn}\) and \(\sigma_0 ^ {\iqn}\) such that if \(\norm{\x ^ {0} - \x^{\star}} \leq \epsilon^{\iqn}\) and \(\sigma(\B_i ^ {0}, \nabla ^ 2 f_i(\x ^ {0})) \leq \sigma_0^{\iqn}\) for all \(i \in [n]\), for the sequence of iterates \(\{\x^{t}\}\) generated by SIQN, there exists a sequence \(\{\zeta ^ {k}\}\) such that \(\norm{\x^t - \x^\star} \leq \zeta ^ {\floor{\tfrac{t - 1}{n}}}\) for all \(t \ge 1\) and the sequence \(\{\zeta ^ {k}\}\) satisfies \begin{align}
		\zeta ^ {k} \leq \epsilon \big(1 - \tfrac{\mu}{dL}\big) ^ {\tfrac{(k + 2)(k + 1)}{2}}, 
	\end{align}
	forall \(k \ge 0\).
\end{lemma}
\begin{proof}
	Refer Appendix \ref{app:final_rate}.
\end{proof}

%% file: appendices/Proof_linear_hybrid_IQN.tex
\begin{proof}For a given \(\rho\) that satisfies \(0 < \rho < 1 - \frac{\mu}{dL}\), let the variables \(\epsilon^{\iqn}, \sigma_0^{\iqn}\) be chosen to satisfy \begin{align}\label{choice_variables}
		\frac{\frac{\tilde{L}\epsilon^{\iqn}}{2} + L{\delta ^ {\iqn}}}{\mu} \leq \frac{\rho}{1 + \rho} < 1.
	\end{align}
	\begin{rem}\label{rem:existence}
			 Indeed there exists positive constants \(\epsilon ^ \iqn, \sigma_0 ^ \iqn\) that satisfy \eqref{choice_variables}. Recall from the premise that \(\delta\) is a function of \(\epsilon ^\iqn, \sigma_0 ^ \iqn\), and we can define the left-hand-side of \eqref{choice_variables} as a function \(g(\epsilon^{\iqn}, \sigma^{\iqn}_0) \) as \[
			g(\epsilon^{\iqn}, \sigma^{\iqn}_0) := \frac{\frac{\tilde{L}\epsilon^{\iqn}}{2} + L{\delta ^ {\iqn}}}{\mu} = \frac{\frac{\tilde{L}\epsilon^{\iqn}}{2} + Le ^ {\frac{4M\sqrt{L}\epsilon^{\iqn}}{1 - \rho}} \big(\sigma^{\iqn}_0 + \epsilon^{\iqn} \frac{4Md\sqrt{L}}{1 - \frac{\rho}{1 - \frac{\mu}{dL}}}\big)}{\mu}.\]
   
    Fix ${\sigma_0^\iqn = \tfrac{\mu}{2L} (\tfrac{\rho}{1+\rho}) > 0}$. It is easy to see that $g(\epsilon ^ \iqn,\sigma_0 ^ \iqn)$ is continuous and monotonically increasing in \(\epsilon^\iqn\).
 Also, note that $g(0, \tfrac{\mu}{2L} (\tfrac{\rho}{1+\rho})) = \tfrac{\rho}{2(1+\rho)}$ and 
 $\lim_{\epsilon \rightarrow \infty} g(\epsilon, \tfrac{\mu}{2L} (\tfrac{\rho}{1+\rho})) = \infty$.
 We can therefore apply the Intermediate Value Theorem (IVT) to guarantee that there exists $\epsilon > 0$ such that $g(\epsilon, \tfrac{\mu}{2L} (\tfrac{\rho}{1+\rho})) \le \tfrac{\rho}{1+\rho}$.
		\end{rem}

	\textbf{{Base case:}} At \(t = 1\), applying Lemma \ref{lem:distance_change_single_iter}, we have \begin{align}
		\norm{\x^1 - \x^\star} \leq \frac{{\tilde L}\Gamma ^ 0}{2} \sum_{i = 1} ^ {n} \norm{\z_i ^ 0 - \x^\star} ^ 2 + \Gamma^0 \sum_{i = 1} ^ {n} \norm{\B_i ^ 0 - \nabla ^ 2 f_i(\z_i ^ 0)} \norm{\z_i^0 - \x^\star}.\nn
	\end{align}
	Since \(\B_i ^ 0 \succeq \nabla ^ 2 f_i(\z_i ^ 0)\) and \(\nabla ^ {2}f_i(\z_i ^ 0) \preceq L\I\), applying Lemma \ref{lem:simple_lemma}, we have \(\sigma(\B_i ^ 0, \nabla ^ 2 f_i(\z_i^0)) \geq \frac{1}{L} \norm{\B_i ^ 0 - \nabla ^ 2 f_i(\z_i ^ 0)}\). This gives\begin{align}\nn
		\norm{\x^1 - \x^\star} \leq  n\Gamma ^ 0\big(\frac{\tilde{L}\epsilon^{\iqn}}{2} + L \sigma^{\iqn}_0\big) \norm{\x^0 - \x^\star},
	\end{align}
	where we have used \(\z_{i}^{0} = \x^0\) and \(\norm{\x^0 - \x^\star} \leq \epsilon^{\iqn}, \sigma(\B_{i}^0, \nabla ^ 2 f_{i}(\z_i^0)) \leq \sigma^{\iqn}_0\). 
 
 We now bound \(\Gamma ^ 0\).
	Define \(\X^0 := \frac{1}{n}\sum_{i = 1} ^ {n} \B_{i} ^ 0, \Y^{0} := \frac{1}{n}\sum_{i = 1} ^ {n} \nabla ^ 2 f_{i}(\z_i ^ 0)\). We have the following:\begin{align}\nn
		\frac{1}{n} \sum_{i = 1} ^ {n} \norm{\B_{i} ^ 0 - \nabla ^ 2 f_{i}(\z_i^0)} \stackrel{{\Delta}} \geq\norm{\X^0 - \Y^0} = \norm{(\Y^{0})((\Y^0)^{-1}\X^0 - \I)} \stackrel{(a)}\geq \mu\norm{(\Y^0)^{-1}\X^0 - \I},
	\end{align}
	where \((a)\) follows since \(\Y ^ 0 =  \frac{1}{n}\sum_{i = 1} ^ {n} \nabla ^ 2 f_{i}(\z_i ^ 0) \succeq \mu \I\)  (\(\because\) Assumption \ref{ass:smooth_strong}). This gives us \[\norm{(\Y^0)^{-1}\X^0 - \I}\leq \frac{L\sigma^{\iqn}_0}{\mu} \leq \frac{L{\delta ^ {\iqn}}}{\mu} \stackrel{\eqref{choice_variables}} < \frac{\rho}{1 + \rho}.\]

	 We can now upper bound $\Gamma^0$ using Banach's Lemma \ref{lem:banach}. Consider the matrix $(\Y^0)^{-1}\X^0 - \I$ and note that it satisfies the requirement of \ref{lem:banach} from the above result. Therefore, we have
\begin{align}
& \norm{(\X^{0})^{-1}\Y^0}
= \norm{\big(\I + ((\Y^0)^{-1}\X^0) - \I)\big)^{-1}} \stackrel{\Lem. \ref{lem:banach}}{\leq} \frac{1}{1 - \norm{ (\Y^0)^{-1}\X^0 - \I }} \leq 1 + \rho.\nn
\end{align}
Recall that \(\mu \I \preceq \Y^0\).
Using this and the previous result, we can upper bound $\norm{(\X^{0})^{-1}}$ as
$\mu \norm{(\X^{0})^{-1}} \leq \norm{(\X^{0})^{-1}\Y^0} \leq 1 + \rho.\nn$
This gives us, \(\Gamma^0 = \frac{1}{n}\norm{(\X^{0})^{-1}} \leq \frac{1 + \rho}{n\mu}\). Therefore, we have the following bound on \(\norm{\x^{1} - \x^\star}\):\begin{align}
		\norm{\x^1 - \x^\star} \leq \frac{\frac{\tilde{L}\epsilon^{\iqn}}{2} + L \sigma^{\iqn}_0}{\mu}(1 + \rho) \norm{\x^0 - \x^\star} \leq \frac{\frac{\tilde{L}\epsilon^{\iqn}}{2} + L{\delta ^ {\iqn}}}{\mu}(1 + \rho)\norm{\x^0 - \x^\star} \stackrel{\eqref{choice_variables}}\leq \rho \norm{\x^0 - \x^\star}.\nn
	\end{align}
	%Choose \(\rho < 1 - c\), where \(c := \frac{\mu}{dL}\). Choose \(\epsilon^{\iqn}, \sigma^{\iqn}\) in such a manner that \(\nn
	%	\frac{\tilde{L} \epsilon^{\iqn} + L {\delta ^ {\iqn}}}{\mu - L\sigma^{\iqn}_0} \leq \rho
	%\). 
	By the updates performed by Algorithm \ref{alg:h_iqn}, we have \(\z_{1} ^ 1 = \x^{1}\) and \(\z_{i} ^ 1 = \x^{0}\) for \(i \neq 1\). Applying Lemma \ref{lem:sig}, we obtain \begin{align}
		\sigma(\B^{1}_1, \nabla ^ 2 f_{1}(\z_{1} ^ {1})) \leq \big(1 - c\big)e ^ {\frac{4M\sqrt{L}\epsilon^{\iqn}}{1 - \rho}}\big(\underbrace{\sigma(\B_{1} ^ 0, \nabla ^ 2 f_{1}(\z_1 ^ 0))}_{\leq \sigma^{\iqn}_0} + \epsilon^{\iqn} \frac{4Md\sqrt{L}}{1 - \frac{\rho}{1 - c}}\big) \leq (1 - c){\delta ^ {\iqn}}.\nn
	\end{align} This completes the proof for \(t = 1\).

			\textbf{{Induction Hypothesis (IH):}} 	Let \eqref{contraction_siqn} and \eqref{sigma_rel_siqn} hold for \(t \in [jn + m]\), where \(j \geq 0, 0 \leq m < n\).

				\textbf{{Induction step:}} 
    We then prove that \eqref{contraction_siqn} and \eqref{sigma_rel_siqn} also hold for \(t = jn + m + 1\).
Recall that the tuples are updated in a deterministic cyclic order, and at the current time $t$, we are in the $j^{\text{th}}$ cycle and have updated the $m^{\text{th}}$ tuple. Therefore, it is easy to note that 
\(\z_{i}^{jn + m} = \x ^ {jn + i}\), for all \(i \in [m]\), which refer to the tuples updated in this cycle, and \(\z_{i} ^ {jn + m} = \x^{jn - n + i}\) for all $i \in [n]\backslash[m]$, which refer to the tuples updated in the previous cycle. 	From the induction hypothesis, we have
\begin{align}\label{induction_hyp_eff_1}
	\norm{\z_i ^ {jn + m} - \x ^ \star} \leq \begin{cases}
	\rho ^ {\ceil{\frac{jn + i}{n}}} \norm{\x^{0} - \x^{\star}}    & i \in [m], \\
	\rho ^ {\ceil{\frac{(j - 1)n + i)}{n}}} \norm{\x^0 - \x^\star} & i \in [n]\backslash[m].    
	\end{cases}
\end{align}

    % We prove that \eqref{contraction_siqn} and \eqref{sigma_rel_siqn} hold for \(t = jn + m + 1\). The current time \(t = jn + m + 1\) corresponds to the \(j\)-th cycle where indices \(i = 1, \dots, m\) have been updated and we are updating the \((m + 1)\)-th coordinate. Hence, \(\z_{i}^{jn + m} = \x ^ {jn + i}\), for \(i = 1, \dots, m\) and \(\z_{i} ^ {jn + m} = \x^{jn - n + i}\), for \(i = m + 1, \dots, n\). 
    {\textbf{Step 1}} 
    
    Applying Lemma \ref{lem:distance_change_single_iter} on updating \(\z_{m + 1}\), we have
			\begin{align}
				\norm{\z_{m + 1} ^ {jn + m + 1} - \x^\star} & \leq \frac{\tilde{L}\Gamma ^{jn + m}}{2} \sum_{i = 1} ^ {n} \norm{\z_{i} ^ {jn + m} - \x^\star} ^ 2 + \nn\\
				& \hspace{20mm} \Gamma ^ {jn + m} \sum_{i = 1} ^ {n} \norm{\B_i ^ {jn + m} - \nabla ^2f_{i}(\z_i ^ {jn + m})}\norm{\z_{i}^{jn + m} - \x^\star},\nn \\
				& = \frac{\tilde{L}\Gamma ^ {jn + m}}{2}\big(\sum_{i = 1} ^ {m} \norm{\z_{i} ^ {jn + m} - \x^\star} ^ 2  + \sum_{i = m + 1}^{n} \norm{\z_{i} ^ {jn + m} - \x^\star} ^ 2 \big) + \nn \\
				& \hspace{20mm} \Gamma^{jn + m}\big(\sum_{i = 1} ^ {m} \norm{\B_i ^ {jn + m} - \nabla ^2f_{i}(\z_i ^ {jn + m})}\norm{\z_{i}^{jn + m} - \x^\star}\big) + \nn \\
				& \hspace{20mm}\Gamma ^ {jn + m}\big(\sum_{i = m + 1}^ {n} \norm{\B_i ^ {jn + m} - \nabla ^2f_{i}(\z_i ^ {jn + m})}\norm{\z_{i}^{jn + m} - \x^\star}\big)\nn.
			\end{align}
			From the induction hypothesis, we have
			\begin{align}\label{induction_hyp_1}
				\norm{\z_i ^ {jn + m} - \x ^ \star} {\leq} \begin{cases}
					\rho ^ {\ceil{\frac{jn + i}{n}}} \norm{\x^{0} - \x^{\star}} & i \in [m], \\
					\rho ^ {\ceil{\frac{(j - 1)n + i)}{n}}} \norm{\x^0 - \x^\star} & i \in [n] \backslash [m],
				\end{cases}
			\end{align}
			Since \(\B_{i} ^ {jn + m} \succeq \nabla ^ 2 f_{i}(\z_{i} ^ {jn + m})\) (Lemma \ref{lem:formalize_psdness}), applying Lemma \ref{lem:simple_lemma} and the induction hypothesis for \(\sigma\), we have the following bound on \(\norm{\B_{i}^{jn + m} - \nabla ^ 2 f_i(\z_i^{jn + m})}\): \begin{align}\label{induction_hyp_2}
				\norm{\B_{i}^{jn + m} - \nabla ^ 2 f_i(\z_i^{jn + m})} \stackrel{\Lem. \ref{lem:simple_lemma}}\leq L \sigma^{\iqn}(\B_i ^ {jn + m}, \nabla ^ 2 f_i(\z_i ^ {jn + m})) \leq \begin{cases}
					L{\delta ^ {\iqn}} (1 - c) ^ {\ceil{\frac{
								jn + i}{n}}} & i \in [m], \\
					L {\delta ^ {\iqn}} (1 - c) ^ {\ceil{\frac{(j - 1)n + i}{n}}} & i \in [n]\backslash[m].
				\end{cases}
			\end{align}
			Therefore, the bound on \(\norm{\z_{m + 1} ^ {jn + m + 1} - \x^\star}\) simplifies to \begin{align}
				\norm{\z_{m + 1} ^ {jn + m + 1} - \x^\star} & \stackrel{\eqref{induction_hyp_1}
    \eqref{induction_hyp_2}}\leq \Gamma ^ {jn + m} \big(m\frac{\tilde{L}\epsilon^{\iqn}}{2} \rho ^ {2j + 2} + (n - m)\frac{\tilde{L}\epsilon^{\iqn}}{2}\rho ^ {2j} + mL {\delta ^ {\iqn}} (1 - c) ^ {j + 1} \rho ^ {j + 1} \\ 
				& \hspace{50mm} + (n - m)L{\delta ^ {\iqn}}(1 - c) ^ {j} \rho ^ j\big)\norm{\x^0 - \x^\star} \label{track_1}, \\
				&\hspace{-10mm}\leq \Gamma ^ {jn + m}\rho ^ j(n\frac{\tilde{L}\epsilon^{\iqn}}{2}\underbrace{\rho ^ {j}}_{\leq 1} + m L {\delta ^ {\iqn}} \underbrace{(1 - c) ^ {j + 1}\rho}_{\leq 1} + (n - m)L{\delta ^ {\iqn}} \underbrace{(1 - c) ^ {j})}_{\le 1}\norm{\x^{0} - \x^\star},\nn \\
				&\hspace{-10mm}\leq \Gamma ^ {jn + m} \rho ^ j  \big(n\frac{\tilde{L}\epsilon^{\iqn}}{2} + nL{\delta ^ {\iqn}} \big)\norm{\x^0 - \x^\star}.\label{track_2}
			\end{align}
			
			We now bound \(\Gamma ^ {jn + m}\).
			Define 
			\(\X^{jn + m} \coloneqq \tfrac{1}{n}\sum_{i = 1} ^ {n} \B_{i} ^ {jn + m}\) and 
			\(\Y^{jn + m} \coloneqq \tfrac{1}{n}\sum_{i = 1} ^ {n} \nabla ^ 2 f_{i}(\z_i ^ {jn + m})\).
			We follow the same recipe to bound \(\Gamma ^ 0\) in the base case. Observe that
			\begin{align*}
				\frac{1}{n} \sum_{i = 1} ^ {n} \norm{\B_{i} ^ {jn + m} - \nabla ^ 2 f_{i}(\z_i^{jn + m})} \stackrel{\Delta} \geq\norm{\X^{jn + m} - \Y^{jn + m}} 
				\overset{(a)}{\geq} \mu\norm{(\Y^{jn + m})^{-1}\X^{jn + m} - \I}.
			\end{align*}
			The inequality \((a)\) follows 
			from Assumption \ref{ass:smooth_strong} which implies that 
			\(
			\mu \I \preceq \frac{1}{n} \sum_{i = 1} ^ n \nabla ^ 2 f_{i}(\z_i ^ {jn + m})
			= \Y^{jn + m}\). By tracking steps \eqref{track_1}-\eqref{track_2}, we can establish that \begin{align*}
					\frac{1}{n} \sum_{i = 1} ^ {n} \norm{\B_{i} ^ {jn + m} - \nabla ^ 2 f_{i}(\z_i^{jn + m})} \leq L\delta^\iqn.
			\end{align*}
			From the above two chain of inequalities, we deduce
			\begin{align*}
				\norme{(\Y^{jn + m})^{-1}\X^{jn + m} - \I}
				\leq \tfrac{L\delta^\iqn}{\mu}
				\stackrel{\eqref{choice_variables}} < \frac{\rho}{1 + \rho}.
			\end{align*}
			We can now upper bound \(\Gamma ^ {jn + m}\) using Banach's Lemma by exactly following the procedure laid out in the base case. We get that 
\(\Gamma^{jn+m} = \norme{\sum_{i = 1} ^ {n} \B_{i} ^ {jn + m})^{-1}} \leq \frac{1 + \rho}{n\mu}\).  Therefore, we obtain \begin{align}\label{induction_step_1}
				\norm{\z_{m + 1} ^ {jn + m + 1} - \x^\star}\leq \frac{\frac{\tilde{L}\epsilon^{\iqn}}{2} + L{\delta ^ {\iqn}}}{\mu}(1 + \rho)\rho ^ {j} \norm{\x^0 - \x^\star}  \stackrel{\eqref{choice_variables}}\leq \rho ^ {j + 1} \norm{\x^0 - \x^\star}.
			\end{align}
			Since \(\z_{m + 1} ^ {jn + m + 1} = \x^{jn + m + 1}\), \eqref{contraction_siqn} holds for \(t = jn + m + 1\). 
   
   {\textbf{Step 2}}

   Next, we prove that \(\sigma(\B_{m + 1} ^ {jn + m + 1}, \nabla ^ 2 f_{m + 1}(\z_{m + 1} ^ {jn + m + 1})) \leq (1 - c) ^ {j + 1} {\delta ^ {\iqn}}\). We define the sequence \(\{\y_k\}\), for \(k = 0, \dots, j + 1\), such that 
    \(\{\y_k\} = \{\x ^ 0, \z_{m + 1} ^ {m + 1}, \dots, \z_{m + 1} ^ {jn + m + 1}\}\). 
 The sequence \(\{\y_k\}_{k = 1} ^ {j + 1}\) comprises of the updated value of \(\z_{m + 1}\) till the current cycle \(j\). Since \(\{\y_k\}\) comes about from the application of BFGS updates as described in the statement of Lemma \ref{lem:sig}, 
therefore $\{\y_k\}$ satisfies the conditions of Lemma \ref{lem:sig}. This implies that \begin{align}\nn
				\norm{\y_k - {\x^{\star}}} \leq \rho ^ {\ceil{\frac{(k - 1)n + m + 1}{n}}}\norm{\x^{0} - \x^\star} = \rho ^ {k} \norm{\x^0 - \x^\star},
			\end{align}
			for \(k \in [j + 1]\). 
		Since \(\y_{j + 1} = \z_{m + 1} ^ {jn + m + 1}\), we have proved \eqref{sigma_rel_siqn} for \(t = jn + m + 1\). The proof is hence complete via induction.
  	\end{proof}

		\begin{corollary}\label{cor:up_beta}
			If Assumptions \ref{ass:smooth_strong} and \ref{ass:lip_hessian} hold, the following holds true for all \(t \ge 1\): \begin{align}\label{bound_beta_t}
				\norm{\z_{i_t} ^ t - \z_{i_t} ^ {t - 1}}_{\z_{i_t} ^ {t - 1}} \leq 2U_t,
			\end{align}
			where \(U_t := \sqrt{L}\rho ^ {\ceil{\frac{t}{n} - 1}}\epsilon ^ {\iqn}\).
			%      Further, it holds that \begin{align}\label{correction}(1 + MU_{t})\B_{i_t} ^ t \succeq \nabla ^ 2 f_{i_t}(\z_{i_t} ^ t). \end{align}
		\end{corollary}
		
		\begin{proof}
			We can bound \(\norm{\z_{i_t} ^ t - \z_{i_t} ^ {t - 1}}_{\z_{i_t} ^ {t - 1}}\) in the following manner:
			\begin{align}\nn
				\norm{\z_{i_t} ^ t - \z_{i_t} ^ {t - 1}}_{\z_{i_t} ^ {t - 1}} \stackrel{(a)}\leq \sqrt{L} \norm{\z_{i_t} ^ t - \z_{i_t} ^ {t - 1}} &\stackrel{{\Delta}}\leq \sqrt{L}(\norm{\z_{i_t} ^ t - \x ^ \star} + \norm{\z_{i_t} ^ {t-1} - \x^\star}),\nn \\
				&\leq \sqrt{L} \big(\rho ^ {\ceil{\frac{t}{n}}} + \rho ^ {\ceil{\frac{t - n}{n}}}\big) \epsilon^{\iqn} \leq 2 \sqrt{L}\rho ^ {\ceil{\frac{t}{n} - 1}} \epsilon^{\iqn}\nn,
			\end{align} 
			where \((a)\) follows since \(\nabla ^ 2 f_{i_t}(\z_{i_t} ^ {t - 1}) \preceq L\I\) (\(\because\) Assumption \ref{ass:smooth_strong}). Therefore, the correction term \(\beta_t = \frac{M}{2}\norm{\z_{i_t} ^ t - \z_{i_t} ^ {t - 1}}_{\z_{i_t} ^ {t - 1}} \leq M\sqrt{L}\rho ^ {\ceil{\frac{t}{n} - 1}} \epsilon^{\iqn} = 2 U_{t}\), which establishes \eqref{bound_beta_t}. 
		\end{proof}
		\begin{rem}[\(\beta_t\) can be bounded by a quantity that remains constant in a cycle]
			Recall that the correction factor \(\beta_t = \frac{M}{2}\norm{\z_{i_t} ^ t - \z_{i_t} ^ {t - 1}}_{\z_{i_t} ^ {t - 1}}\) in SIQN was introduced to ensure that \((1 + \beta_t) ^ 2 {\B_{i_t}} ^ {t - 1} \succeq (1 + \beta_t)\K ^ t\) (we formalized this in Lemma \ref{lem:formalize_psdness}). Intuitively, executing SIQN with a higher correction factor \(\beta_{t} ^ \new = MU_{t} \geq \frac{M}{2}\norm{\z_{i_t} ^ t - \z_{i_t} ^ {t - 1}}_{\z_{i_t} ^ {t - 1}}\) (follows from Corollary \ref{cor:up_beta}) which remains constant in a cycle, it should continue to hold that \((1 + \beta_t ^ \new) ^ 2 \B_{i_t} ^ {t - 1} \succeq (1 + \beta_t ^ \new)\K ^ t\). We skip the proof for the sake of brevity as it is similar to the analysis SIQN.
		\end{rem}

%% file: appendices/Proof_mean_linear_hybrid_IQN.tex
\begin{proof}
    Let \(t = jn + m + 1\), where \(0 \leq j\) and \(0 \leq m < n\). From the proof of Lemma \ref{lem:convergence_h_iqn}, we have an uniform upper bound on \(\Gamma ^ {jn + m} = \norm{\big(\sum_{i = 1} ^ {n} \B_{i} ^ {jn + m}\big) ^ {-1}}\), given by \begin{align}\nn
   	\Gamma ^ {jn + m} \leq \frac{1 + \rho}{n\mu},
   \end{align}
   and the following upper bound on \(\norm{\B_{i}^{jn + m} - \nabla ^ 2 f_i(\z_i^{jn + m})}\):  \begin{align}\nn
   	\norm{\B_{i}^{jn + m} - \nabla ^ 2 f_i(\z_i^{jn + m})} \stackrel{\Lem. \ref{lem:simple_lemma}}\leq L \sigma(\B_i ^ {jn + m}, \nabla ^ 2 f_i(\z_i ^ {jn + m})) \leq \begin{cases}
   		L\delta^{\iqn} (1 - c) ^ {\ceil{\frac{
   					jn + i}{n}}} & i \in [m], \\
   		L \delta^{\iqn} (1 - c) ^ {\ceil{\frac{(j - 1)n + i}{n}}} & i \in [n]\backslash[m].
   	\end{cases}
   \end{align}
   This gives \(\norm{\B_i ^ {jn + m} - \nabla ^ 2 f_{i}(\z_i ^ {jn + m})} \leq L\delta^{\iqn}(1 - c)^j\), for all \(i \in [n]\). 
   
   Further, from Lemma\ref{lem:convergence_h_iqn}, we also have \(\norm{\z_i ^ {jn + m} - \x^\star} \leq \rho ^ {j + 1}\norm{\x^0 - \x^\star}\), for all \(i \in [n]\) and \(\norm{\z_{i} ^ {jn + m} - \x^\star} \leq \rho ^ {j} \norm{\x^0-\x^\star}\), for all \(i \in [n]\backslash[m]\). This clearly implies \(\norm{\z_{i} ^ {jn + m} - \x^\star} \leq \rho ^ {j} \norm{\x^0 - \x^\star}\), for \(i = 1, \dots, n\).

   Applying Lemma \ref{lem:distance_change_single_iter} at \(t = jn + m + 1\) and upper bounding \(\Gamma ^ {jn + m}\), we obtain \begin{align}
   	\norm{\x ^ {jn + m + 1} - \x^\star} &\leq \frac{1 +\rho}{n\mu}\sum_{i = 1} ^ {n} \big(\frac{\tilde{L}}{2}\norm{\z_{i} ^ {jn + m} - \x^\star} + \norm{\B_{i} ^ {jn + m} - \nabla ^ 2f_{i}(\z_i^{jn + m})}\big)\norm{\z_{i} ^ {jn + m} - \x^\star}\nn, \\
   	&\stackrel{(a)} \leq \frac{1 + \rho}{n\mu}\big(\frac{\tilde{L}\epsilon}{2} + L\delta^{\iqn}\big) (1 - c) ^ {j}\sum_{i = 1} ^ {n} \norm{\z_i ^ {jn + m} - \x^\star}\nn,\\
   	&\leq \rho (1 - c) ^ {j}\frac{1}{n} \sum_{i = 1} ^ {n} \norm{\z_i ^ {jn + m} - \x^\star} \stackrel{(b)}\leq (1 - c) ^ {\ceil{\frac{t}{n}}}\frac{1}{n}\sum_{i = 1} ^ {n}\norm{\x^{t - i} - \x^\star}\nn, 
   \end{align}
   where \((a)\) follows from the bounds  \(\norm{\z_{i} ^ {jn + m} - \x ^ \star} \leq \rho ^ {j}\norm{\x^0 - \x^\star}\) and \(\norm{\B_i ^ {jn + m} - \nabla ^ 2 f_{i}(\z_i ^ {jn + m})} \leq L\delta^{\iqn} (1 - c) ^ {j}\) discussed above and \(\rho < 1 - c\). Also \((b)\) follows since \(\z_{i} ^ {jn + m} = \x ^ {jn + i}\), for all \(i \in [m]\) and \(\z_{i} ^ {jn + m} = \x ^ {jn - n + i}\), for all \(i \in [n]\backslash[m]\), which implies \(\sum_{i = 1} ^ {n} \norm{\z_i ^  {jn + m} - \x^\star} = \sum_{i = 1} ^ {n} \norm{\x ^ {t - i} - \x^\star}\). This completes the proof.
\end{proof}

%% file: appendices/efficient_implementation_sliqn.tex
We begin by showing that the update of \(\x ^ t\) in the SLIQN algorithm (Algorithm \ref{alg:eff_h_iqn}) can be carried out in \(\cO(d ^ 2)\) cost.
\subsection{Carrying out \eqref{x_new_update} in Algorithm \ref{alg:eff_h_iqn} in \(\cO(d ^ 2)\) cost}\label{update_efficiently}
\input{appendices/Efficiently_inverting_B}

% \subsection{An efficient implementation of SLIQN with \(\cO(d ^ 2)\) per-iteration cost}
% \input{SLIQN.tex}

%% file: appendices/Efficiently_inverting_B.tex
We begin by defining the following variables that track the summands in \eqref{x_new_update} \begin{align}\label{def_D_ph_g}
	\bar{\D} ^ t := \sum_{i = 1} ^ {n} \D_i ^ {t}, \bi{\phi} ^ {t} := \sum_{i = 1} ^ {n} \D_{i} ^ {t}\z_{i} ^ {t}, \g ^ t := \sum_{i = 1} ^ n \nabla f_{i}(\z_i ^ t).
	\end{align} The update \eqref{x_new_update} (for time \(t + 1\)) can be expressed in terms of the defined variables as \begin{align}\label{new_x_def_vars}
	\x ^ {t + 1} = \big(\bar{\D} ^ t\big) ^ {-1}\big(\bi{\phi} ^ {t} - \g ^ t\big).
\end{align}From the updates performed by Algorithm \ref{alg:eff_h_iqn} at time \(t\), we have \begin{align}\label{update_g}
	 \g ^ t = \g ^ {t - 1} + \big(\nabla f_{i_t}(\z_{i_t} ^ {t}) - \nabla f_{i_t}(\z_{i_t} ^ {t - 1})\big).
\end{align}
%Therefore, given access to \(\bi{\phi} ^ {t - 1}\) and \(\g ^ {t - 1}\), we can compute \(\bi{\phi} ^ {t}\) and \(\g ^ t\) in \(\cO(d ^ 2)\) time. Further, when \(t\) is a multiple of \(n\), 

Further, we can express \(\bm{\phi} ^ t\) in terms of \(\bm{\phi} ^ {t - 1}\) as follows: \begin{align}\label{update_phi}
\bm{\phi} ^ t = \sum_{i = 1, i\neq i_t} ^ {n} \omega_t \D_i ^ {t - 1}\z_{i} ^ {t - 1} + \D_{i_t} ^ t \z_{i_t} ^ t  = \omega_t \bm{\phi} ^ {t - 1} - \omega_t \D_{i_t} ^ {t - 1}\z_{i_t} ^ {t - 1} + \D_{i_t} ^ t\z_{i_t} ^ t.
\end{align}
This updating scheme can be implemented iteratively, where we only evaluate \(\bi{\phi} ^ 0, \g ^ 0\) explicitly and evaluate \(\bi{\phi} ^ t\) in \(\cO(d ^ 2)\) cost, for all \(t \geq 1\) by \eqref{update_phi} and \(\g ^ t\)  in \(\cO(d)\) cost, for all \(t \ge 1\) by \eqref{update_g}.

	Next, we demonstrate the method for updating \((\bar{\D} ^ t) ^ {-1}\). We begin by expressing \(\big(\bar{\D} ^ t\big) ^ {-1}\) in terms of \(\big(\bar{\D} ^ {t - 1}\big) ^ {-1}\) in the following manner: \begin{align}\nn
		(\bar{\D} ^ t) ^ {-1} &= \bigg(\sum_{i = 1, i \neq i_t} ^ {n - 1}\omega_t\D_{i} ^ {t - 1} + \D_{i_t} ^ t\bigg) ^ {-1} = \big(\omega_t \big(\bar{\D} ^ {t - 1} - \D_{i_t} ^ {t - 1}\big) + \D_{i_t} ^ t\big) ^ {-1}\nn, \\
		&= \omega_t ^ {-1}\big(\bar{\D} ^ {t - 1} + \omega_t ^ {-1}\D_{i_t} ^ t - \D_{i_t} ^ {t - 1}\big) ^ {-1}\label{continue_from_here}.
	\end{align}
	
Expanding the BFGS update \eqref{eff_second_modified_update}, we can express \(\omega_t ^ {-1} \D_{i_t} ^ t\) as \begin{align} \label{omega_D_Q}
	\omega_t ^ {-1} \D_{i_t} ^ t  \stackrel{\eqref{eff_second_modified_update}}=  \Q^t - \frac{\Q^t{\bar{\u}} ^ t (\Q^t{\bar{\u}} ^ t)^\T}{\ip{{\bar{\u}} ^ t}{\Q^t{\bar{\u}} ^ t}} + \frac{\nabla ^ {2} f_{i_t}(\z_{i_t}^t) {\bar{\u}} ^ t (\nabla ^ {2} f_{i_t}(\z_{i_t}^t) {\bar{\u}} ^ t)^\T}{\ip{{\bar{\u}} ^ t}{\nabla ^ {2} f_{i_t}(\z_{i_t}^t) {\bar{\u}} ^ t}},
\end{align}

where  we have used the shorthand \({\bar{\u}} ^ {t}\) for \({{\bar{\u}}}(\Q_{t}, \nabla ^ {2}f_{i_t}(\z_{i_t}^t))\). Further, \(\Q ^ t\) can be expressed as \begin{align}\label{q_update}
	\Q^{t} \stackrel{\eqref{eff_first_modified_update_}}= \D_{i_t} ^ {t - 1} + \frac{{\y_{i_t} ^ t}{\y_{i_t} ^ t}^\T}{\ip{{\y_{i_t} ^ t}}{{\s_{i_t} ^ t}}} - \frac{\D_{i_t}^{t - 1}{\s_{i_t} ^ t} (\D_{i_t}^{t - 1}{\s_{i_t} ^ t})^\T}{\ip{{\s_{i_t} ^ t}}{\D_{i_t}^{t - 1}{\s_{i_t} ^ t}}},
\end{align}	

%	Next, we consider \(t\) which is not a multiple of \(n\) and demonstrate that given access to \((\bar{\D} ^ {t - 1}) ^ {-1}\), using the Sherman-Morrison formula we can invert the matrix \(\bar{\D} ^ t\) in \(\mathcal{O}(d ^ {2})\) time. 

where \({\s_{i_t} ^ t} = \z_{i_t} ^ {t} - \z_{i_t}^{t - 1}, {\y_{i_t} ^ t} = (1 + \alpha_{\ceil{\nicefrac{t}{n} - 1}}) \K ^ {t}{\s_{i_t} ^ t} =  (1 + \alpha_{\ceil{\nicefrac{t}{n} - 1}}) (\nabla f_{i_t}(\z_{i_t}^{t}) - \nabla f_{i_t}(\z_{i_t}^{t - 1}))\).

Adding \eqref{omega_D_Q} and \eqref{q_update}, we obtain \begin{align}\label{add}
	\omega_t ^ {-1} \D_{i_t} ^ t -  \D_{i_t} ^ {t - 1} &=  \frac{{\y_{i_t} ^ t}{\y_{i_t} ^ t}^\T}{\ip{{\y_{i_t} ^ t}}{{\s_{i_t} ^ t}}} - \frac{\D_{i_t}^{t - 1}{\s_{i_t} ^ t} (\D_{i_t}^{t - 1}{\s_{i_t} ^ t})^\T}{\ip{{\s_{i_t} ^ t}}{\D_{i_t}^{t - 1}{\s_{i_t} ^ t}}} - \frac{\Q^t{\bar{\u}} ^ t (\Q^t{\bar{\u}} ^ t)^\T}{\ip{{\bar{\u}} ^ t}{\Q^t{\bar{\u}} ^ t}} + \nn \\
	& \hspace{40mm}\frac{\nabla ^ {2} f_{i_t}(\z_{i_t}^t) {\bar{\u}} ^ t (\nabla ^ {2} f_{i_t}(\z_{i_t}^t) {\bar{\u}} ^ t)^\T}{\ip{{\bar{\u}} ^ t}{\nabla ^ {2} f_{i_t}(\z_{i_t}^t) {\bar{\u}} ^ t}}.
\end{align}
Continuing from \eqref{continue_from_here}, we obtain
\begin{align}
	\big(\bar{\D} ^ t\big) ^ {-1} = \omega_t ^ {-1}\bigg(\bar{\D} ^ {t - 1} + \frac{{\y_{i_t} ^ t}{\y_{i_t} ^ t}^\T}{\ip{{\y_{i_t} ^ t}}{{\s_{i_t} ^ t}}} &- \frac{\D_{i_t}^{t - 1}{\s_{i_t} ^ t} (\D_{i_t}^{t - 1}{\s_{i_t} ^ t})^\T}{\ip{{\s_{i_t} ^ t}}{\D_{i_t}^{t - 1}{\s_{i_t} ^ t}}} - \frac{\Q^t{\bar{\u}} ^ t (\Q^t{\bar{\u}} ^ t)^\T}{\ip{{\bar{\u}} ^ t}{\Q^t{\bar{\u}} ^ t}}\nn \\
	&+ \frac{\nabla ^ {2} f_{i_t}(\z_{i_t}^t) {\bar{\u}} ^ t (\nabla ^ {2} f_{i_t}(\z_{i_t}^t) {\bar{\u}} ^ t)^\T}{\ip{{\bar{\u}} ^ t}{\nabla ^ {2} f_{i_t}(\z_{i_t}^t) {\bar{\u}} ^ t}}\bigg) ^ {-1}.\nn
\end{align} 
Next, we define the following matrix: \begin{align}\nn
	\bp_{1} := \bar{\D} ^ {t - 1} + \frac{{\y_{i_t} ^ t}{{\y_{i_t} ^ t}}^\T}{\ip{{\y_{i_t} ^ t}}{{\s_{i_t} ^ t}}}  - \frac{\D_{i_t}^{t - 1}{\s_{i_t} ^ t} (\D_{i_t}^{t - 1}{\s_{i_t} ^ t})^\T}{\ip{{\s_{i_t} ^ t}}{\D_{i_t}^{t - 1}{\s_{i_t} ^ t}}} - \frac{\Q^t{\bar{\u}} ^ t (\Q^t{\bar{\u}} ^ t)^\T}{\ip{{\bar{\u}} ^ t}{\Q^t{\bar{\u}} ^ t}}.
\end{align}
Expressing \(\big(\bar{\D} ^ t\big)^{-1}\) in terms of \(\bp_{1} ^ {-1}\) we get the following:
\begin{align}
	\big(\bar{\D} ^ t \big) ^ {-1} &= \omega_t ^ {-1}\bigg(\bp_{1} + \frac{\nabla ^ {2} f_{i_t}(\z_{i_t}^t) {\bar{\u}} ^ t (\nabla ^ {2} f_{i_t}(\z_{i_t}^t) {\bar{\u}} ^ t)^\T}{\ip{{\bar{\u}} ^ t}{\nabla ^ {2} f_{i_t}(\z_{i_t}^t) {\bar{\u}} ^ t}}\bigg) ^ {-1}, \nn\\
	&\stackrel{\eqref{shermon}}= \omega_t ^ {-1}\bigg(\bp ^ {-1} - \frac{\bp_{1}^{-1}\nabla ^ {2} f_{i_t}(\z_{i_t}^t){\bar{\u}} ^ t(\nabla ^ {2} f_{i_t}(\z_{i_t}^t){\bar{\u}} ^ t)^\T \bp_{1}^{-1}}{\ip{{\bar{\u}} ^ t}{\nabla ^ {2} f_{i_t}(\z_{i_t}^t) {\bar{\u}} ^ t} + {\ip{\nabla ^ {2} f_{i_t}(\z_{i_t}^t){\bar{\u}} ^ t}{\bp_{1} ^ {-1}\nabla ^ {2} f_{i_t}(\z_{i_t}^t){\bar{\u}} ^ t}}}\bigg).\label{bar_D_inv}
\end{align}
Define the following matrix:
\begin{align}
	\bp_{2} :=\bar{\D} ^ {t - 1} + \frac{{\y_{i_t} ^ t}{\y_{i_t} ^ t}^\T}{\ip{{\y_{i_t} ^ t}}{{\s_{i_t} ^ t}}} - \frac{\D_{i_t}^{t - 1}{\s_{i_t} ^ t} (\D_{i_t}^{t - 1}{\s_{i_t} ^ t})^\T}{\ip{{\s_{i_t} ^ t}}{\D_{i_t}^{t - 1}{\s_{i_t} ^ t}}} \implies  \bp_{1} = \bp_{2} - \frac{\Q^t{\bar{\u}} ^ t (\Q^t{\bar{\u}} ^ t)^\T}{\ip{{\bar{\u}} ^ t}{\Q^t{\bar{\u}} ^ t}}. \nn
\end{align}
Expressing \(\bp_{1}^{-1}\) in terms of \(\bp_{2}^{-1}\) we get \begin{align}\label{psi_1_inv}
	\bp_{1} ^ {-1} \stackrel{\eqref{shermon}}= \bp_{2} ^ {-1} + \frac{\bp_{2} ^ {-1}\Q^t{\bar{\u}} ^ t (\Q^t{\bar{\u}} ^ t)^\T \bp_{2} ^ {-1}}{\ip{{\bar{\u}} ^ t}{\Q^t{\bar{\u}} ^ t} - \ip{\Q^t {\bar{\u}} ^ t}{\bp_{2}^{-1}\Q^t {\bar{\u}} ^ t}}.
\end{align}
Define the following matrix:
\begin{align}\nn
	\bp_{3} := \bar{\D} ^ {t - 1} + \frac{{\y_{i_t} ^ t}{\y_{i_t} ^ t}^\T}{\ip{{\y_{i_t} ^ t}}{{\s_{i_t} ^ t}}} \implies \bp_{2} = \bp_{3} - \frac{\D_{i_t}^{t - 1}{\s_{i_t} ^ t} (\D_{i_t}^{t - 1}{\s_{i_t} ^ t})^\T}{\ip{{\s_{i_t} ^ t}}{\D_{i_t}^{t - 1}{\s_{i_t} ^ t}}}.
\end{align}
Expressing \(\bp_{2}^{-1}\) in terms of \(\bp_{3} ^ {-1}\) we get the following:
\begin{align}\label{psi_2_inv}
	\bp_{2}^{-1} \stackrel{\eqref{shermon}}= \bp_{3} ^ {-1} + \frac{\bp_{3} ^ {-1}\D_{i_t}^{t - 1}{\s_{i_t} ^ t} (\D_{i_t}^{t - 1}{\s_{i_t} ^ t})^\T\bp_{3} ^ {-1}}{\ip{{\s_{i_t} ^ t}}{\D_{i_t}^{t - 1}{\s_{i_t} ^ t}} - \ip{\D_{i_t}^{t - 1}{\s_{i_t} ^ t}}{\bp_{3}^{-1}\D_{i_t}^{t - 1}{\s_{i_t} ^ t}}}.
\end{align}
Finally, using \eqref{shermon} to evaluate \(\bp_{3} ^ {-1}\), we obtain \begin{align}\label{psi_3_inv}
	\bp_{3} ^ {-1} = \big(\bar{\D} ^ {t - 1}\big) ^ {-1} - \frac{\big(\bar{\D} ^ {t - 1}\big) ^ {-1}{\y_{i_t} ^ t}{\y_{i_t} ^ t}^\T\big(\bar{\D} ^ {t - 1}\big) ^ {-1}}{\ip{{\y_{i_t} ^ t}}{{\s_{i_t} ^ t}} + \ip{{\y_{i_t} ^ t}}{\big(\bar{\D} ^ {t - 1}\big) ^ {-1}{\y_{i_t} ^ t}}}.
\end{align}
Therefore, given access to \(\big(\bar{\D} ^ {t - 1}\big) ^ {-1}\), we can evaluate \(\psi_{3} ^ {-1}\) in \(\mathcal{O}(d ^ {2})\) time. Similarly, given access to \(\bp_{3} ^ {-1}\), we can evaluate \(\bp_{2} ^ {-1}\) in \(\cO(d^2)\) time. Continuing similarly, we can evaluate \(\bp_{1} ^ {-1}\) and \(\big(\bar{\D} ^{t}\big) ^ {-1}\) in \(\cO(d^2)\) time. This scheme can be enumerated iteratively where we only compute \(\big(\bar{\D} ^ 0\big) ^ {-1}\) explicitly and evaluate \(\big(\bar{\D} ^ t\big) ^ {-1}, \forall t \ge 1\) by the steps  \eqref{psi_3_inv}, \eqref{psi_2_inv}, \eqref{psi_1_inv}, and \eqref{bar_D_inv}.   Therefore, the update \eqref{x_new_update} in Algorithm \ref{alg:eff_h_iqn} can be performed in \(\cO(d^2)\) time. 

\subsection{Efficient Implementation of Algorithm \ref{alg:eff_h_iqn} in \(\cO(d ^ 2)\) cost}\label{subsec:lazy_algorithm}
By lazily carrying out the scaling of the \(\D\)'s, i.e., only scaling when they are used, we can improve the per-iteration complexity of SLIQN to \(\cO(d ^ 2)\). The resulting algorithm is specified by the following pseudo code:
	\begin{algorithm}
	\caption{Sharpened Lazy Incremental Quasi-Newton (SLIQN)}
	\begin{algorithmic}[1]
 \STATE \textbf{Function} \{Sherman-Morrison\} \{$\A^{-1}, {{\u}}, \v$\}
     \STATE \hspace{3mm} \textbf{return} $\A^{-1} - \tfrac{\A^{-1}{{\u}}\v^\T\A^{-1}}{1 + \ip{\v}{\A^{-1}{{\u}}}}$
     
     \STATE \textbf{EndFunction}
     \vspace{3mm}
		% \STATE\textbf{Input:} Sequence \(\{\omega_k\}, k \in \mathbb{N
		% }\);
		\STATE\textbf{Initialize:} Initialize \(\{\z_{i}, \D_{i}\}_{i = 1} ^ {n}\) similar to Algorithm \ref{subsec:lazy_algorithm};
		\STATE Evaluate \(\bar{\D}  := \big(\sum_{i = 1} ^ {n}\D_i\big) ^ {-1}, \bm{\phi} :=  \sum_{i = 1} ^ {n} \D_i\z_i \), and \(\g := \sum_{i = 1} ^ {n} \nabla f_i (\z_i)\);
        \STATE Maintain running auxiliary variables \(\x, {\bar{\u}}, \y, \s, \Q, \D^\text{old}, \K\); // \textsc{\(\x\) keeps a track of \(\x ^ t\), \(\bar{\u}\) keeps a track of the greedy vector, \(\y\) keeps a track of \(\nabla f_{i_t}(\z_{i_t} ^t) - \nabla f_{i_t}(\z_{i_t} ^ {t - 1})\), \(\s\) keeps a track of \(\z_{i_t} ^ t - \z_{i_t} ^ {t - 1}\), whereas \(\Q, \D^\text{old}, \K\) keep track of the intermediate matrices}
		\STATE\textbf{while} \textit{not converged:}
		\STATE\hspace{3mm}Current index to be updated is \(i_t \leftarrow (t - 1) \mod n + 1;\)
		%			\STATE\hspace{3mm}\textbf{if} \(i_{t} = 1\)
		%			\STATE\hspace{3mm}\hspace{3mm}\textbf{for} \(j = 1, \dots, n\)
		%			\STATE\hspace{3mm}\hspace{3mm}\hspace{3mm} Reset \(\D_{j} ^ {t - 1} = (1 + M \sqrt{L}\rho ^ {\floor{\frac{t}{n}}}\epsilon) ^ {2}\B_{j} ^ {t - 1};\)
		%			\STATE\hspace{3mm}\hspace{3mm}\textbf{end for}
		%			\STATE\hspace{3mm}\textbf{end if}
		\STATE\hspace{3mm}Update \(\x\) as  \(\x \leftarrow \big(\bar{\D}\big)\big(\bm{\phi} - \g\big);\)
   \STATE\hspace{3mm}Update \(\s \leftarrow \x -\z_{i_t};\)
   \STATE\hspace{3mm}Update \(\y \leftarrow \nabla f_{i_t}(\x^t) -\nabla f_{i_t}(\z_{i_t});\)
   \STATE\hspace{3mm}Compute \(\omega_t\); // \textsc{\(\omega_t = (1 + \alpha_{\ceil{t / n}}) ^ 2\) if \(t\mod n = 0\) and \(1\) otherwise}
		% \STATE\hspace{3mm}Update \(\z_{i_t}\) as \(\z_{i_t} = \x ^ {t}\);
		%			\STATE\hspace{3mm}Set \(\Q^t = \bfgs{\D_{i_t}^{t - 1}, (1 + M\sqrt{L}\rho ^ {\floor{\frac{t}{n}}}\epsilon) ^ 2 \K ^ t, \z_{i_t}^t - \z_{i_{t}}^{t - 1}}\);
		\STATE\hspace{3mm}Update \(\Q\) as \(\Q \leftarrow \bfgs{(1 + \alpha_{\ceil{t / n - 1}}) ^ 2\D_{i_t}, (1 + \alpha_{\ceil{t / n - 1}})\K, \s}\), where  // \textsc{Lazy Step} \[\K \leftarrow \int_{0} ^ {1} \nabla ^ 2 f_{i_t}(\z_{i_t} + \tau (\x ^ t - \z_{i_t})) d\tau.\]

\STATE\hspace{3mm}Update \({\bar{\u}}\) as \({\bar{\u}} \leftarrow \argmax_{{{\u}} \in \{\e_{i}\}_{i = 1} ^ d} \frac{\ip{{{\u}}}{\Q {{\u}}}}{\ip{{{\u}}}{\nabla ^ 2 f_{i_t}(\x ^ t){{\u}}}}\);
  
  \STATE\hspace{3mm}Update \(\D^{\text{old}}\) as \(\D ^ \text{old} \leftarrow \D_{i_t}\);
  \STATE\hspace{3mm}Update \(\D_{i_t}\) as \(
  		    \D_{i_t} \leftarrow \bfgs{\Q, \nabla ^ 2 f_{i_t}(\x ^ t), {\bar{\u}}}
  		\);
		% \STATE\hspace{3mm}Update the tuples with index \(i \neq i_t\) as \(\z_i ^ t = \z_i ^ {t - 1}, \D_i ^ {t, s} = \D_i ^ {t - 1, s}\);
		\STATE\hspace{3mm}Update \(\bi{\phi}\) as \(\bi{\phi} \leftarrow \omega_t (\bm{\phi}  - \D^{\text{old}}\z_{i_t}) + \D_{i_t} \x ^ {t}\);

			\STATE\hspace{3mm}Update \(\g \) as  \(\g \leftarrow \g  + \big(\nabla f_{i_t}(\x ^ t) - \nabla f_{i_t}(\z_{i_t})\big)\);
    \STATE\hspace{3mm}Update \(\bar{\D}\) as 
                \(\bar{\D} \leftarrow \text{Sherman-Morrison}
                (
                    \bar{\D},
                    \y, \frac{1}{\ip{\y}{\s}}\y
                )
                \);
                
                 \STATE\hspace{3mm}Update \(\bar{\D}\) as 
                \(\bar{\D} \leftarrow \text{Sherman-Morrison}
                (
                    \bar{\D},
                    -\D^{\text{old}}\s, \frac{1}{\ip{\s}{\D^\text{old}\s}}\D^{\text{old}}\s
                )
                \);
                \STATE\hspace{3mm}Update \(\bar{\D}\) as 
                \(\bar{\D} \leftarrow \text{Sherman-Morrison}
                (
                    \bar{\D},
                    -\Q{\bar{\u}}, \frac{1}{\ip{{\bar{\u}}}{\Q{\bar{\u}}}}\Q{\bar{\u}}
                )
                \);
                \STATE\hspace{3mm}Update \(\bar{\D}\) as 
                \(\bar{\D} \leftarrow \omega_{t}^{-1}\text{Sherman-Morrison}
                (
                    \bar{\D},
                    \nabla ^ 2 f_{i_t}(\z_{i_t}){\bar{\u}},  \frac{1}{\ip{{\bar{\u}}}{ \nabla ^ 2 f_{i_t}(\z_{i_t}){\bar{\u}}}}\nabla ^ 2 f_{i_t}(\z_{i_t}){\bar{\u}}
                )
                \);
                \STATE\hspace{3mm}Update \(\z_{i_t}\) as \(\z_{i_t} \leftarrow \x\);
		\STATE Increment the iteration counter \(t\);
		\STATE\textbf{end while}

	\end{algorithmic}
	\label{alg:sl_iqn}
\end{algorithm}

%% file: appendices/convergence_analysis_sliqn.tex
\subsection{Proof of Lemma \ref{lem:distance_change_single_iter_ver_II}} \label{app:distance_change_single_iter_ver_II}
\input{appendices/Proof_distance_change_single_iter.tex}
\subsection{Proof of Lemma \ref{lem:convergence_efficient_h_iqn}}\label{app:convergence_efficient_h_iqn}
\input{appendices/Proof_linear_eff_hybrid_IQN.tex}

\subsection{Proof of Lemma \ref{lem:mean_linear_verII}}\label{app:mean_linear_verII}
\input{appendices/Proof_mean_linear_eff_hybrid_IQN.tex}

\subsection{Proof of Theorem \ref{lem:final_rate}}\label{app:final_rate}
\input{appendices/Proof_superlinear_hybrid_IQN.tex}

%% file: appendices/Proof_distance_change_single_iter.tex
For all \(t \geq 0\), we define \(\H ^ {t} := \big(\sum_{i = 1} ^ {n} \D_{i} ^ t \big) ^ {-1}\). 
From the update for $\x^t$ \eqref{x_update}, we have \begin{align}\nn
  \x ^ t - \x ^ {*} & = \H ^ {t - 1}\bigg(\sum_{i = 1} ^ {n} \D_{i} ^ {t - 1}\z_{i} ^ {t - 1} - \sum_{i = 1} ^ {n} \nabla f_{i}(\z_{i} ^ {t - 1})\bigg) - \x ^ {*},                                                                                                   \\
                    & \stackrel{(a)}= \H ^ {t - 1} \bigg( \sum_{i = 1} ^ {n}\D_{i} ^ {t - 1}(\z_{i} ^ {t - 1} - \x ^ \star) - \sum_{i = 1} ^ {n} \nabla f_{i}(\z_{i} ^ {t - 1})\bigg),\nn                                                \\
                     & \stackrel{(b)}= \H ^ {t - 1} \bigg( \sum_{i = 1} ^ {n}\D_{i} ^ {t - 1}(\z_{i} ^ {t - 1} - \x ^ \star) - \sum_{i = 1} ^ {n} (\nabla f_{i}(\z_{i} ^ {t - 1})-  \nabla f_{i}(\x ^ \star))\bigg),\nn                                                \\
                    & \stackrel{(c)}= \H ^ {t - 1} \bigg( \sum_{i = 1} ^ {n}\D_{i} ^ {t - 1}(\z_{i} ^ {t - 1} - \x ^ \star) - \sum_{i = 1} ^ {n} \int_{0} ^ {1} \nabla ^ 2 f(\x ^ {*} + (\z_{i} ^ {t - 1} - \x ^ \star)v)(\z_{i} ^ {t - 1} - \x ^ \star) dv\bigg),\nn \\
                    & \stackrel{(d)}= \H ^ {t - 1} \bigg(\sum_{i = 1} ^ {n}\big(\D_{i} ^ {t - 1} - \nabla ^ 2 f_{i}(\z_i ^ {t - 1})\big)(\z_{i} ^ {t - 1} - \x ^ \star) + \nn                                                                                         \\
                    & \hspace{30mm}\sum_{i = 1} ^ {n} \int_{0} ^ {1} (\nabla ^ 2f_{i}(\z_{i} ^ {t - 1}) -  \nabla ^ 2f(\x ^ {*} + (\z_{i} ^ {t - 1} - \x ^ \star)v))(\z_{i} ^ {t-1} - \x^\star) dv\bigg)\nn.
\end{align}
The equality \((a)\) follows from the definition \(\H ^ {t - 1} = \big(\sum_{i = 1} ^ {n}\D_i ^ {t - 1}\big) ^ {-1}\). 
The equality \((b)\) uses the fact that \(\nabla f(\x ^ {*}) = \frac{1}{n}\sum_{i = 1} ^ {n} \nabla f_{i}(\x ^ \star) = \mathbf{0}\). 
The equality (c)
follows from the Fundamental Theorem of Calculus, 
and the equality \((d)\) follows by adding and subtracting \(\sum_{i = 1} ^ {n} \nabla ^ 2 f_{i}(\z_i ^ {t - 1})(\z_{i} ^ {t - 1} - \x ^ \star)\). Taking norm on both sides and applying the Triangle inequality, we obtain\begin{align}
  \norm{\x ^ t - \x ^ {*}} & \stackrel{\Delta}\leq \norm{\H ^ {t - 1}} \bigg(\sum_{i = 1} ^ {n}\norm{\big(\D_{i} ^ {t - 1} - \nabla ^ 2 f_{i}(\z_i ^ {t - 1})\big)}\norm{\z_{i} ^ {t - 1} - \x ^ \star} +  \nn                                                                                                              \\
                           & \hspace{20mm}\sum_{i = 1} ^ {n} \norm{\int_{0} ^ {1} (\nabla ^ 2f_{i}(\z_{i} ^ {t - 1}) -  \nabla ^ 2f(\x ^ {*} + (\z_{i} ^ {t - 1} - \x ^ \star)v))(\z_{i} ^ {t-1} - \x^\star) dv}\nn\bigg),                                                                                                  \\
                           & \hspace{-17.5mm}\stackrel{(a)}\leq \norm{\H ^ {t - 1}} \bigg(\sum_{i = 1} ^ {n}\norm{\big(\D_{i} ^ {t - 1} - \nabla ^ 2 f_{i}(\z_i ^ {t - 1})\big)}\norm{\z_{i} ^ {t - 1} - \x ^ \star} +  \nn                                                                                                 \\
                           & \sum_{i = 1} ^ {n} \int_{0} ^ {1}\norm{ (\nabla ^ 2f_{i}(\z_{i} ^ {t - 1}) -  \nabla ^ 2f(\x ^ {*} + (\z_{i} ^ {t - 1} - \x ^ \star)v))(\z_{i} ^ {t-1} - \x^\star) }dv\nn\bigg),                                                                                                               \\
                           & \hspace{-17.5mm}\leq \norm{\H ^ {t - 1}} \bigg(\sum_{i = 1} ^ {n}\norm{\big(\D_{i} ^ {t - 1} - \nabla ^ 2 f_{i}(\z_i ^ {t - 1})\big)}\norm{\z_{i} ^ {t - 1} - \x ^ \star} +  \nn                                                                                                               \\
                           & \hspace{-2.5mm}\sum_{i = 1} ^ {n} \int_{0} ^ {1}\norm{ (\nabla ^ 2f_{i}(\z_{i} ^ {t - 1}) -  \nabla ^ 2f(\x ^ {*} + (\z_{i} ^ {t - 1} - \x ^ \star)v))}\norm{(\z_{i} ^ {t-1} - \x^\star)}dv\nn\bigg),                                                                                          \\
                           & \hspace{-17.5mm}\stackrel{(b)} \leq  \norm{\H ^ {t - 1}}\bigg( \sum_{i = 1} ^ {n}\norm{\big(\D_{i} ^ {t - 1} - \nabla ^ 2 f_{i}(\z_i ^ {t - 1})\big)}\norm{\z_{i} ^ {t - 1} - \x ^ \star} + \tilde{L}\int_{0} ^ {1}(1 - v)dv\sum_{i = 1} ^ {n}\norm{\z_{i} ^ {t - 1} - \x ^ {*}} ^ 2 \nn\bigg), \\
                           & \hspace{-17.5mm}\leq   \norm{\H ^ {t - 1}}\bigg( \sum_{i = 1} ^ {n}\norm{\big(\D_{i} ^ {t - 1} - \nabla ^ 2 f_{i}(\z_i ^ {t - 1})\big)}\norm{\z_{i} ^ {t - 1} - \x ^ \star} + \frac{\tilde{L}}{2}\sum_{i = 1} ^ {n}\norm{\z_{i} ^ {t - 1} - \x ^ {*}} ^ 2 \nn\bigg).
\end{align}
The inequality \((a)\) follows from the known result that if 
$g \colon \mathbb{R} \rightarrow \mathbb{R}^d$ is a continuous function, then
\(\norm{\int_0^1 g(v) \,dv} \leq \int_0^1 \norm{g(v)} \,dv\), 
and the inequality \((b)\) follows from the assumption that the Hessian of $f_i$ is $\tilde{L}-$ Lipschitz (\ref{ass:lip_hessian}), i.e. 
\( \norm{\nabla ^ 2f_{i}(\z_{i} ^ {t - 1}) -  \nabla ^ 2f(\x ^ {*} + (\z_{i} ^ {t - 1} - \x ^ \star)v)} \leq \tilde{L}(1 - v)\norm{\z_{i} ^ {t - 1} - \x ^ \star}\). This completes the proof.

%% file: appendices/Proof_linear_eff_hybrid_IQN.tex
For a given \(\rho\) that satisfies \(0 < \rho < 1 - \frac{\mu}{dL}\), we choose 
\(\epsilon, \delta > 0\) such that they satisfy
\begin{align}\label{choice_variables_II}
\tfrac{1}{\mu}
\left(
    \tfrac{\tilde{L}\epsilon}{2} + L\delta (1 + M\sqrt{L}\epsilon) ^ {2} + L^{\tfrac{3}{2}} M\epsilon (2 + M\sqrt{L}\epsilon)
\right)
\leq 
\frac{\rho}{1 + \rho} < 1.
\end{align}

\begin{rem}
	   Indeed, there exists positive constants \(\epsilon, \delta > 0\) that satisfy \ref{choice_variables_II}.
 Recall from the premise that $\delta$ is a function of $\epsilon, \sigma_0$, and we can define the left-hand-side of \ref{choice_variables_II} as a function $h(\epsilon, \sigma_0)$ as
 \begin{align*}
     h(\epsilon, \sigma_0) 
     \coloneqq
     \tfrac{1}{\mu}
     \left(
         \tfrac{\tilde{L}\epsilon}{2} 
         + 
         Le ^ {\frac{4M\sqrt{L}\epsilon}{1 - \rho}} 
         \big(
            \sigma_0 
            + 
            \epsilon \tfrac{4Md\sqrt{L}}{1 - \rho (1 - \tfrac{\mu}{dL})^{-1}}
        \big)
        (1 + M\sqrt{L}\epsilon) ^ {2} 
        + 
        L^{\tfrac{3}{2}} M\epsilon (2 + M\sqrt{L}\epsilon)
     \right).
 \end{align*}
 Fix ${\sigma_0 = \tfrac{\mu}{2L} (\tfrac{\rho}{1+\rho}) > 0}$. It is easy to see that $h(\epsilon,\sigma_0)$ is continuous and monotonically increasing in \(\epsilon\).
 Also, note that $h(0, \tfrac{\mu}{2L} (\tfrac{\rho}{1+\rho})) = \tfrac{\rho}{2(1+\rho)}$ and 
 $\lim_{\epsilon \rightarrow \infty} h(\epsilon, \tfrac{\mu}{2L}(\tfrac{\rho}{1+\rho})) = \infty$.
 We can therefore apply the Intermediate Value Theorem (IVT) to guarantee that there exists $\epsilon > 0$ such that $h(\epsilon, \tfrac{\mu}{2L} (\tfrac{\rho}{1+\rho})) \le \tfrac{\rho}{1+\rho}$.
\end{rem}
Similar to Lemma \ref{lem:convergence_h_iqn}, we use Induction on $t$ to prove the result.

\textbf{{Base case:}} 
At \(t = 1\), from Lemma \ref{lem:distance_change_single_iter_ver_II}, we have \begin{align}
\norm{\x^1 - \x^\star} \leq \frac{\tilde L\Gamma ^ 0}{2} \sum_{i = 1} ^ {n} \norm{\z_i ^ 0 - \x^\star} ^ 2 + \Gamma^0 \sum_{i = 1} ^ {n} \norm{\D_i ^ 0 - \nabla ^ 2 f_i(\z_i ^ 0)} \norm{\z_i^0 - \x^\star}.\nn
\end{align}
From the initialization, we have that 
\(\D_{i} ^ {0} = (1 + \alpha_0) ^ 2 \I_{i} ^ 0\) and
\(\z_{i}^{0} = \x^0\), for all \(i \in [n]\),
and
\(\norm{\x^0 - \x^\star} \leq \epsilon\).
Substituting these in the above expression, we obtain
\begin{align}
\norm{\x^1 - \x^\star} 
&\leq 
\Gamma ^ 0 \bigg(
    n\tfrac{\tilde L\epsilon}{2}  + \sum_{i = 1} ^ {n}\norm{(1 + \alpha_0) ^ 2\I_i ^ 0 - \nabla ^ 2 f_i(\z_i ^ 0)}
\bigg) \norm{\x^0 - \x^\star}
, \label{first_step_1}\\ 
& \overset{(a)}{\le}  
\Gamma ^ 0
\bigg(
    n\tfrac{\tilde L\epsilon}{2}  + (1 + \alpha_0) ^ 2
\sum_{i = 1} ^ {n}\norm{\I_i ^ 0 - \nabla ^ 2 f_i(\z_i ^ 0)} + \alpha_0 (\alpha_0 + 2)\sum_{i = 1}^ {n}\underbrace{\norm{\nabla ^ 2f_i(\z_i ^ 0)}}_{\leq L}\bigg)\norm{\x^0 - \x^\star}\nn, \\
&\overset{(b)}{\le} \Gamma ^ 0\
\bigg(
    n\tfrac{\tilde L\epsilon}{2}  + nL(1 + \alpha_0) ^ 2\sigma_0 + nL \alpha_0 (\alpha_0 + 2)
\bigg)\norm{\x^0 - \x^\star}\nn, \\
&\le \Gamma ^ 0
\bigg(
    n\tfrac{\tilde L\epsilon}{2}  + nL(1 + M\sqrt{L}\epsilon) ^ 2\sigma_0 + nL^{\frac{3}{2}} M\epsilon (M\sqrt{L}\epsilon + 2)\bigg)
    \norm{\x^0 - \x^\star} \nn, \\
&\le \Gamma ^ 0
\bigg(
    n\tfrac{\tilde L\epsilon}{2}  + nL(1 + M\sqrt{L}\epsilon) ^ 2\delta + nL^{\frac{3}{2}} M\epsilon (M\sqrt{L}\epsilon + 2)\bigg)
    \norm{\x^0 - \x^\star},\label{end_step_1}
\end{align}
where \((a)\) follows by adding and subtracting \((1 + \alpha_0)^ 2 \nabla^2f_{i}(\z_i ^ 0)\) to \((1 + \alpha_0) ^ 2\I_i ^ 0 - \nabla ^ 2 f_i(\z_i ^ 0)\) and 
applying the Triangle inequality.
To see why inequality \((b)\) is true, first recall from the initialization that
\(\sigma(\I_i ^ {0}, \nabla ^ 2 f_i(\z_i ^  0)) \leq \sigma_0\) and 
\(\I_{i} ^ 0 \succeq \nabla ^ 2 f_i(\z_i ^ 0)\). 
Applying Lemma \ref{lem:simple_lemma}, we have  \(\norm{\I_i ^ {0} - \nabla ^ 2 f_i(\z_i ^  0)} \le L\sigma(\I_i ^ {0}, \nabla ^ 2 f_i(\z_i ^  0)) \leq L\sigma_0\).
	
We now upper bound \(\Gamma ^ 0\).
Define 
\(
    \X^0 \coloneqq \frac{1}{n}\sum_{i = 1} ^ {n} \D_{i} ^ 0
\)
and 
\(\Y^{0} \coloneqq \frac{1}{n}\sum_{i = 1} ^ {n} \nabla ^ 2 f_{i}(\z_i ^ 0)\). 
Then, we have
\begin{align}\nn
\frac{1}{n} \sum_{i = 1} ^ {n} \norm{\D_{i} ^ 0 - \nabla ^ 2 f_{i}(\z_i^0)} \stackrel{\Delta} \geq\norm{\X^0 - \Y^0} = \norm{(\Y^{0})((\Y^0)^{-1}\X^0 - \I)} \stackrel{(a)}\geq \mu\norm{(\Y^0)^{-1}\X^0 - \I},
\end{align}
where the inequality \((a)\) follows 
from Assumption \ref{ass:smooth_strong} which implies \(\mu \I \preceq \frac{1}{n} \sum_{i = 1} ^ n \nabla ^ 2 f_{i}(\z_i ^ 0) = \Y^0\). 
By tracking the parts of steps \eqref{first_step_1}-\eqref{end_step_1} which  
bound $\norm{\D_i ^ 0 - \nabla ^ 2 f_i(\z_i ^ 0)}$, we get
\begin{align}\nn
\sum_{i = 1} ^ {n} \norm{\D_{i} ^ 0 - \nabla ^ 2 f_{i}(\z_i^0)} \leq nL(1 + M\sqrt{L}\epsilon) ^ 2\delta + nL ^ {\frac{3}{2}}M\epsilon(M\sqrt{L}\epsilon + 2).
\end{align}
From the above two chain of inequalities, we obtain
\begin{align*}
    \norm{(\Y^0)^{-1}\X^0 - \I}
    \leq \tfrac{1}{\mu}\bigg(
        L(1 + M\sqrt{L}\epsilon) ^ 2 \delta + L ^ {\frac{3}{2}}M\epsilon(M\sqrt{L}\epsilon + 2) \bigg) \stackrel{\eqref{choice_variables_II}} < \frac{\rho}{1 + \rho}.
\end{align*}
We can now upper bound $\Gamma^0$ using Banach's Lemma \ref{lem:banach}. Consider the matrix $(\Y^0)^{-1}\X^0 - \I$ and note that it satisfies the requirement of \ref{lem:banach} from the above result. Therefore, we have
\begin{align}
& \norm{(\X^{0})^{-1}\Y^0}
= \norm{\big(\I + ((\Y^0)^{-1}\X^0) - \I)\big)^{-1}} \stackrel{\Lem. \ref{lem:banach}}{\leq} \frac{1}{1 - \norm{ (\Y^0)^{-1}\X^0 - \I }} \leq 1 + \rho.\nn
\end{align}
Recall that \(\mu \I \preceq \Y^0\).
Using this and the previous result, we can upper bound $\norm{(\X^{0})^{-1}}$ as
$\mu \norm{(\X^{0})^{-1}} \leq \norm{(\X^{0})^{-1}\Y^0} \leq 1 + \rho.\nn$
This gives us, \(\Gamma^0 = \frac{1}{n}\norm{(\X^{0})^{-1}} \leq \frac{1 + \rho}{n\mu}\).

Substituting this bound on $\Gamma^0$ in \ref{end_step_1}, we obtain
\begin{align}
\norm{\x^{1} - \x^\star}\leq \tfrac{1 + \rho}{\mu}\big(\tfrac{\tilde L\epsilon}{2}  + L(1 + M\sqrt{L}\epsilon) ^ 2\delta + L^{\frac{3}{2}} M\epsilon (M\sqrt{L}\epsilon + 2)\big)\norm{\x^0 - \x^\star} \stackrel{\eqref{choice_variables_II}}\leq \rho \norm{\x^0- \x^\star}.\nn
\end{align}
To complete the base step, we now upper bound 
$\sigma(\omega_1^{-1}\D^{1}_1, \nabla ^ 2 f_{1}(\z_{1} ^ {1}))$, where $\omega_1 = 1$ .
Applying Lemma \ref{lem:sig} with parameters as
\(
    T =  1, \tilde{\x}= \x^\star, \P^{0} = (1 + \alpha_0) ^ 2 \I_1 ^ 0 = \D_1 ^ 0
\) 
(refer to Remark \ref{change_r} we made for Lemma \ref{lem:sig},
where we stated that the results of Lemma \ref{lem:sig} remain unchanged on redefining \(r_k := 2\sqrt{L}\rho ^ {k - 1}\epsilon \ge 2\sqrt{L}\rho ^ {k - 1}\norm{\x^0 - \x^\star}\)), we get

\begin{align}
\sigma(\D^{1}_1, \nabla ^ 2 f_{1}(\z_{1} ^ {1})) \leq \big(1 - c\big)e ^ {\frac{4M\sqrt{L}\epsilon}{1 - \rho}}\big(\underbrace{\sigma(\I_{1} ^ 0, \nabla ^ 2 f_{1}(\z_1 ^ 0))}_{\leq \sigma_0} + \epsilon \frac{4Md\sqrt{L}}{1 - \frac{\rho}{1 - c}}\big) \leq (1 - c)\delta.\nn
\end{align}
Finally, as a technical remark, the proof of Lemma \ref{lem:sig} already shows that 
$\D^{1}_1 \succeq \nabla ^ 2 f_{1}(\z_{1} ^ {1})$, and therefore 
$\sigma(\D^{1}_1, \nabla ^ 2 f_{1}(\z_{1} ^ {1}))$ is well defined. 
This completes the proof for for the base case.  

We now prove that \eqref{contraction_verII} and \eqref{sigma_rel_verII} hold for any \(t > 1\) by induction. 

% \begin{rem}\label{rem:sigma_defined}
% 	Note that for the \(\sigma\) metric in \eqref{sigma_rel_verII} to be well defined, we need \(\omega_{t} ^ {-1}\D_{i} ^ t \succeq \nabla ^ 2 f_{i_t}(\z_{i_t} ^ t), \forall t \ge 1\). Therefore, we also establish the claim that \(\omega_{t} ^ {-1}\D_{i} ^ t \succeq \nabla ^ 2 f_{i_t}(\z_{i_t} ^ t)\) holds for all \(t \ge 1\) as a part of our induction based proof.
% \end{rem}
\textbf{Induction hypothesis (IH):} Let \eqref{contraction_verII} and \eqref{sigma_rel_verII} hold for all $t \in [jn+m]$  for some \(j \geq 0\) and \( 0 \leq m < n\).
	 
\textbf{Induction step:} 
We then prove that \eqref{contraction_verII} and \eqref{sigma_rel_verII} also hold for \(t = jn + m + 1\).
Recall that the tuples are updated in a deterministic cyclic order, and at the current time $t$, we are in the $j^{\text{th}}$ cycle and have updated the $m^{\text{th}}$ tuple. Therefore, it is easy to note that 
\(\z_{i}^{jn + m} = \x ^ {jn + i}\), for all \(i \in [m]\), which refer to the tuples updated in this cycle, and \(\z_{i} ^ {jn + m} = \x^{jn - n + i}\) for all $i \in [n]\backslash[m]$, which refer to the tuples updated in the previous cycle. 	From the induction hypothesis, we have
\begin{align}\label{induction_hyp_eff_1}
	\norm{\z_i ^ {jn + m} - \x ^ \star} \leq \begin{cases}
	\rho ^ {\ceil{\frac{jn + i}{n}}} \norm{\x^{0} - \x^{\star}}    & i \in [m], \\
	\rho ^ {\ceil{\frac{(j - 1)n + i)}{n}}} \norm{\x^0 - \x^\star} & i \in [n]\backslash[m].    
	\end{cases}
\end{align}
We will execute the induction step in three distinct stages. 
In the first stage we will establish an upper bound on 
\(\sum_{i = 1} ^ {n} \| \D_i ^ {jn + m} - \nabla ^ 2 f_i(\z_i ^ {jn + m})\|\).
In the second stage, we will use the previous result and Lemma 
\ref{lem:distance_change_single_iter_ver_II} to bound 
\(\|\z_{m + 1} ^ {jn + m + 1} - \x^\star\|\).
In the final stage, we will 
prove the linear convergence of the updated Hessian approximation, i.e., 
\(\sigma(\omega_{jn + m + 1} ^ {-1}\D_{m + 1} ^ {jn + m + 1}, \nabla ^ 2 f_{m + 1}(\z_{m + 1} ^ {jn + m + 1})) \leq (1 - c) ^ {j + 1} \delta\).

Since \(\D_{i_t} ^ t\) is updated in a different manner for \(t \bmod n \neq 0\) and \(t \bmod n = 0\), we split the first stage into two cases corresponding to 
\(t \bmod n \neq 0\) and \(t \bmod n = 0\).

{\textbf{Stage 1, Case 1:}} ${t \bmod n \neq 0}$

Since \(t = jn + m + 1\), this case is equivalent to considering \(0 \leq m < n - 1\).
From the structure of the cyclic updates and the pre-multiplication of the scaling factor, it is easy to note that 
\(\D_{i} ^ {jn + m} = \D_i ^ {jn + i}\), for all \(i \in [m]\), 
\(\D_{i} ^ {jn + m} = (1 + M\sqrt{L}\epsilon \rho ^ j) ^ 2\D_i ^ {jn - n + i}\), for all $i \in [n-1]\backslash[m]$, 
and \(\D_{i} ^ {jn + m} = \D_i ^ {jn}\), for \(i = n\). 

We want to bound 
\(\sum_{i = 1} ^ {n} \| \D_i ^ {jn + m} - \nabla ^ 2 f_i(\z_i ^ {jn + m})\|\).
For all $i \in [m]$, from the induction hypothesis, we have
\begin{align}
    &\sigma(
    \D_i ^ {jn + m}, \nabla ^ 2 f_i(\z_i ^ {jn + m}))
    \leq (1-c) ^ {\ceil{\frac{jn + i}{n}}} \delta, \nn
    \\
    \overset{\Lem. {\ref{lem:simple_lemma}}}{\implies}
    &
    \norm{\D_{i}^{jn + m} - \nabla ^ 2 f_i(\z_i^{jn + m})} 
    = 
    \norm{\D_i ^ {jn + i} - \nabla ^ 2f_{i}(\x ^ {jn + i})}  
    \leq 	L(1-c) ^ {\ceil{\frac{jn + i}{n}}} \delta.
    \label{i=ind_hyp}
\end{align}

For all \(i \in [n-1] \backslash [m]\), we follow in the footsteps of \ref{first_step_1}-\ref{end_step_1} from the base case to get
\begin{align}\nn
\norm{\D_i ^ {jn + m} - \nabla^2 f_i(\z_i ^ {jn + m})} 
&= \norm{(1 + M\sqrt{L}\epsilon \rho ^ {j}) ^ 2 \D_i ^ {jn - n + i} - \nabla ^ 2 f_i(\x ^ {jn - n + i})}\nn,
\end{align}
\begin{multline*}
    \hspace{8mm}\leq (1 + M \sqrt{L}\epsilon \rho ^ j) ^ {2}
    \norm{\D_i ^ {jn - n + i} - \nabla ^ 2 f_i(\x ^ {jn - n + i})} +\\
    M\sqrt{L}\epsilon \rho ^ {j}(2 + M\sqrt{L}\epsilon \rho ^ {j})\underbrace{\norm{\nabla ^ 2 f_i(\x ^ {jn - n + i})}}_{\le L},\nn
\end{multline*}
\vspace{-3.5mm}
\begin{align}
    &\leq (1 + M \sqrt{L}\epsilon \rho ^ j) ^ {2}\norm{\D_i ^ {jn - n + i} - \nabla ^ 2 f_i(\x ^ {jn - n + i})} + ML ^ {\frac{3}{2}}\epsilon \rho ^ {j}(2 + M\sqrt{L}\epsilon \rho ^ j),\nn
    \\
    &\stackrel{(a)}\leq (1 + M \sqrt{L}\epsilon \rho ^ j) ^ {2}L\delta(1 - c)^ {j} + ML ^ {\frac{3}{2}}\epsilon \rho ^ {j}(2 + M\sqrt{L}\epsilon \rho ^ j) 
\label{i=m+1_n-1}.
\end{align}
The inequality \((a)\) follows from 
\(\norme{\D_i ^ {jn - n + i} - \nabla ^ 2 f_i(\x ^ {jn - n + i})} \leq (1 - c) ^ j\delta\), 
which can be established from the induction hypothesis in a similar way as we did for the case with $i \in [m]$.
Next, for $i=n$, we have
\begin{align}
\norm{\D_i ^ {jn + m} - \nabla^2 f_i(\z_i ^ {jn + m})} & =  \norm{\D_i ^ {jn} - \nabla^2 f_i(\z_i ^ {jn})},\nn                                                      \\
& = \norm{\omega_{jn} (\omega_{jn} ^ {-1}\D_i ^ {jn} - \nabla^2 f_i(\z_i ^ {jn})) + (\omega_{jn} - 1)\nabla^2 f_i(\z_i ^ {jn + m})}, \nn 
\\
& \stackrel{\Delta}\leq \omega_{jn}\norme{\omega_{jn} ^ {-1}\D_i ^ {jn} - \nabla^2 f_i(\z_i ^ {jn})} + \abs{\omega_{jn} - 1} \norme{\nabla^2 f_i(\z_i ^ {jn})}, \nn \\
& \stackrel{(a)}\leq (1 + M\sqrt{L}\epsilon \rho ^ j)L\delta (1 - c) ^ j +  
\abs{\omega_{jn} - 1} \norme{\nabla^2 f_i(\z_i ^ {jn})}\nn,  
\\
&\stackrel{(b)}\leq (1 + M\sqrt{L}\epsilon \rho ^ j)L\delta (1 - c) ^ j +  ML ^ {\frac{3}{2}}\epsilon \rho ^ {j}(2 + M\sqrt{L}\epsilon \rho ^ j)\label{i=n}.
\end{align}
To see the deduction \((a)\), 
we follow in the footsteps of the case with $i \in [m]$. 
Concretely,
from induction and Lemma \ref{lem:simple_lemma}, 
\(\norme{
    \omega_{jn} ^ {-1}\D_i ^ {jn} - \nabla^2 f_i(\z_i ^ {jn})
    } \leq 
    L\sigma (\omega_{jn} ^ {-1}\D_i ^ {jn}, \nabla^2 f_i(\z_i ^ {jn})) 
    \leq L\delta(1 - c) ^ j
\).
Inequality \((b)\) uses the fact that
\(\omega_{jn} = 1 + M\sqrt{L}\epsilon \rho ^ j\) and
\(\|\nabla ^ 2 f_i(\z_i ^ {jn})\| \leq L\) (\(\because\) Assumption \ref{ass:smooth_strong}).

We can now bound the quantity \(\sum_{i = 1} ^ {n} \| \D_i ^ {jn + m} - \nabla ^ 2 f_i(\z_i ^ {jn + m})\|\) using\eqref{i=ind_hyp}, \eqref{i=m+1_n-1} and \eqref{i=n}, as follows:
\begin{multline*}
\hspace{-0.5cm}
\sum_{i = 1} ^ {n} \norm{\D_i ^ {jn + m} - \nabla ^ 2 f_i(\z_i ^ {jn + m})} 
\leq 
mL\delta (1 - c) ^ {j + 1} + 
(n - m)\big((1 + M\sqrt{L} \epsilon \rho ^ j) ^ {2}L\delta(1 - c)^ {j} \nn 
\\
+ ML ^ {\frac{3}{2}}\epsilon \rho ^ {j}(2 + M\sqrt{L}\epsilon \rho ^ j)\big)\nn, 
\end{multline*}
\begin{align}
& \hspace{4.0cm}\leq mL\delta + (n - m) \big((1 + M\sqrt{L}\epsilon) ^ 2 L\delta + ML ^ {\frac{3}{2}}\epsilon(2 + M\sqrt{L}\epsilon)\big)\nn,
\\
& \hspace{4.0cm}
\stackrel{(a)}\leq nL\delta (1 + M\sqrt{L}\epsilon) ^ 2 + nML^{\frac{3}{2}}\epsilon (2 + M\sqrt{L}\epsilon)
\label{cumulative_bound},
\end{align}
where \((a)\) follows since \(mL \delta + (n - m) (1 + M\sqrt{L}\epsilon) ^ 2 L\delta < mL(1 + M\sqrt{L}\epsilon) ^ 2 L\delta + (n - m) (1 + M\sqrt{L}\epsilon) ^ 2 L\delta = nL(1 + M\sqrt{L}\epsilon) ^ 2 L\delta\).
	
{\textbf{Stage 1, Case 2}: $t \bmod n = 0$}

Since \(t = jn + m + 1\), this case is equivalent to considering \(m = n - 1\).
This is a simpler case, as compared to the previous case. Here, 
we have \(\D_i ^ {jn} = (1 + M\sqrt{L}\epsilon \rho ^ j) ^ 2\D_i ^ {jn - n + i}\), for all \(i \in [n-1]\),
and $\D_n^{jn}$ would be used as it is.

We follow exactly the steps leading up to \eqref{i=m+1_n-1} and \eqref{i=n}. For, $i \in [n-1]$, using the reasoning in the derivation of \eqref{i=m+1_n-1}, we get
\begin{align}\nn
\norm{\D_i ^ {jn + m} - \nabla^2 f_i(\z_i ^ {jn + m})} 
\leq (1 + M \sqrt{L}\epsilon \rho ^ j) ^ {2}L\delta(1 - c)^ {j} + ML ^ {\frac{3}{2}}\epsilon \rho ^ {j}(2 + M\sqrt{L}\epsilon \rho ^ j) 
,\nn
\end{align}
For $i=n$, from equation \eqref{i=n}, we get
\begin{align*}
    \norm{\D_i ^ {jn + m} - \nabla^2 f_i(\z_i ^ {jn + m})}
    \leq 
    (1 + M\sqrt{L}\epsilon \rho ^ j)L\delta (1 - c) ^ j +  ML ^ {\frac{3}{2}}\epsilon \rho ^ {j}(2 + M\sqrt{L}\epsilon \rho ^ j).
\end{align*}
We can now bound the quantity \(\sum_{i = 1} ^ {n} \norm{\D_i ^ {jn + m} - \nabla ^ 2 f_i(\z_i ^ {jn + m})}\) in the following manner:
{
\begin{align}\nn
\sum_{i = 1} ^ {n} \norm{\D_i ^ {jn + m} - \nabla ^ 2 f_i(\z_i ^ {jn + m})} 
& \leq  n(1 + M\sqrt{L}\epsilon \rho ^ j)L\delta (1 - c) ^ j +  nML ^ {\frac{3}{2}}\epsilon \rho ^ {j}(2 + M\sqrt{L}\epsilon \rho ^ j), \nn \\
& \stackrel{(a)}\leq nL\delta (1 + M\sqrt{L}\epsilon) ^ 2 + nML^{\frac{3}{2}}\epsilon (2 + M\sqrt{L}\epsilon)\label{cumulative_bound_II},
\end{align}
}
where \((a)\) follows by bounding the terms \(< 1\).

\textbf{Stage 2}

We now use the result from Stage $1$ to bound 
\(\|\z_{m + 1} ^ {jn + m + 1} - \x^\star\|\),
\begin{multline*}
\norme{\z_{m + 1} ^ {jn + m + 1} - \x^\star} 
\overset{\Lem. \ref{lem:distance_change_single_iter_ver_II}}{\leq} 
\frac{
    \tilde{L}\Gamma ^{jn + m}
    }{2} 
\sum_{i = 1} ^ {n} \norme{\z_{i} ^ {jn + m} - \x^\star} ^ 2 + \nn                  \\
\Gamma ^ {jn + m} \sum_{i = 1} ^ {n} \norm{\D_i ^ {jn + m} - \nabla ^2f_{i}(\z_i ^ {jn + m})}\norme{\z_{i}^{jn + m} - \x^\star},\nn  
\end{multline*}
\vspace{-3mm}
\begin{align}
&\hspace{1.5cm} \stackrel{(a)}
\leq 
\Gamma ^ {jn + m} \big(
    n\frac{\tilde{L}\epsilon}{2}\rho ^ {j} + \sum_{i = 1} ^ {n} \norm{\D_i ^ {jn + m} - \nabla ^2f_{i}(\z_i ^ {jn + m})} \rho ^ {j}
\big)\norm{\x^0 - \x^\star}\nn,
\\
&\hspace{1.1cm} \overset{\eqref{cumulative_bound}, \eqref{cumulative_bound_II}}{\leq}
\Gamma ^ {jn + m}\big(
    n\frac{\tilde{L}\epsilon}{2} + nL\delta (1 + M\sqrt{L}\epsilon) ^ 2 + nML^{\frac{3}{2}}\epsilon (2 + M\sqrt{L}\epsilon)
\big)\rho ^ {j}\norm{\x^0 - \x^\star}.\nn 
\end{align}
The inequality \((a)\) follows from the induction hypothesis that 
\(\norme{\z_{i} ^ {jn + m} - \x^\star} \leq \rho ^ {j} \norm{\x^0 - \x^\star}\), for all \(i \in [n]\backslash[m]\)
and \(\norme{\z_i ^ {jn + m} - \x^\star} \le \rho ^ {j + 1}\norme{\x^0 - \x^\star}\), for all \(i \in [m]\). 
Since $\rho < 1$, we can have a common upper bound, 
\(\norme{\z_{i} ^ {jn + m} - \x^\star} \leq \rho ^ j \norme{\x^0 - \x^\star}\), for all \(i \in [n]\). 

We now bound \(\Gamma ^ {jn + m}\).
Define 
\(\X^{jn + m} \coloneqq \tfrac{1}{n}\sum_{i = 1} ^ {n} \D_{i} ^ {jn + m}\) and
\(\Y^{jn + m} \coloneqq \tfrac{1}{n}\sum_{i = 1} ^ {n} \nabla ^ 2 f_{i}(\z_i ^ {jn + m})\).
We follow the same recipe when we bound \(\Gamma ^ 0\) in the base case. Observe that
\begin{align*}
    \frac{1}{n} \sum_{i = 1} ^ {n} \norm{\D_{i} ^ {jn + m} - \nabla ^ 2 f_{i}(\z_i^{jn + m})} \stackrel{\Delta} \geq\norm{\X^{jn + m} - \Y^{jn + m}} 
    \overset{(a)}{\geq} \mu\norm{(\Y^{jn + m})^{-1}\X^{jn + m} - \I}.
\end{align*}
The inequality \((a)\) follows 
from Assumption \ref{ass:smooth_strong} that 
\(
\mu \I \preceq \frac{1}{n} \sum_{i = 1} ^ n \nabla ^ 2 f_{i}(\z_i ^ {jn + m})
= \Y^{jn + m}\). 
Also, restating the result from Stage $1$, we have
\begin{align*}
    \frac{1}{n} \sum_{i = 1} ^ {n} \norm{\D_{i} ^ {jn + m} - \nabla ^ 2 f_{i}(\z_i^{jn + m})}
    \leq 
    L\delta (1 + M\sqrt{L}\epsilon) ^ 2 + ML^{\frac{3}{2}}\epsilon (2 + M\sqrt{L}\epsilon)
\end{align*}
From the above two chain of inequalities, we deduce
\begin{align*}
    \norme{(\Y^{jn + m})^{-1}\X^{jn + m} - \I}
    \leq \tfrac{1}{\mu}\bigg(
        L\delta (1 + M\sqrt{L}\epsilon) ^ 2 + ML^{\frac{3}{2}}\epsilon (2 + M\sqrt{L}\epsilon)\bigg)
        \stackrel{\eqref{choice_variables_II}} < \frac{\rho}{1 + \rho}.
\end{align*}
We can now upper bound \(\Gamma ^ {jn + m}\) using Banach's Lemma by exactly following the procedure laid out in the base case. We get that 
\(\Gamma^{jn+m} = \norme{(\sum_{i = 1} ^ {n} \D_{i} ^ {jn + m})^{-1}} \leq \frac{1 + \rho}{n\mu}\). Substituting this above, we get
\begin{align}\nn
\norme{\z_{m + 1} ^ {jn + m + 1} - \x^\star}& \leq\frac{1 + \rho}{n\mu}\big(n\frac{\tilde{L}\epsilon}{2} + nL\delta (1 + M\sqrt{L}\epsilon) ^ 2 + nML^{\frac{3}{2}}\epsilon (2 + M\sqrt{L}\epsilon)\big)\rho ^ {j}\norm{\x^0 - \x^\star}, \\
&\stackrel{\eqref{choice_variables_II}}\leq \rho ^ {j + 1} \norm{\x^0 - \x^\star}\label{induction_step_eff_1}.
\end{align}
This completes the induction step proof for \eqref{contraction_verII} at \(t = jn + m + 1\).
	
\textbf{Stage 3}

In this stage, we
prove the linear convergence of the updated Hessian approximation, that is, we show that
\(\sigma(\omega_{jn + m + 1} ^ {-1}\D_{m + 1} ^ {jn + m + 1}, \nabla ^ 2 f_{m + 1}(\z_{m + 1} ^ {jn + m + 1})) \leq (1 - c) ^ {j + 1} \delta\). But first, we establish that the $\sigma(\cdot)$ metric is well defined, by showing that
\begin{align}\label{prove_this}\omega_{jn + m + 1} ^ {-1}\D_{m + 1} ^ {jn + m + 1} \succeq \nabla ^ 2 f_{m + 1}(\z_{m + 1} ^ {jn + m + 1}).
\end{align}
We make two observations to establish \ref{prove_this}. The first observation is that
\begin{align*}
    \omega_{jn + m} ^ {-1} \D_{m + 1} ^ {jn + m} 
    \overset{(a)}{\succeq} 
    \nabla ^ 2 f_{m + 1}(\z_{m + 1} ^ {jn + m})
    \overset{(b)}{\succeq} 
    (1 + \tfrac{1}{2}Mr_{jn + m + 1})^{-1} \K ^ {jn + m},
\end{align*}
where $(a)$ follows from the
induction hypothesis
and 
$(b)$ follows from Lemma \ref{integral_lemma}.
For convenience, we restate that 
\(r_{t} \coloneqq \norme{\z_{i_t} ^ t - \z_{i_t} ^ {t - 1}}_{\z_{i_t} ^ {t - 1}}.\) 

For the next observation, we first bound \(r_{jn + m + 1}\) as we did in Corollary \ref{cor:up_beta}, in the following manner:
\begin{align}
r_{jn + m + 1} = 
\norme{\z_{m + 1} ^ {jn + m + 1} - \z_{m + 1} ^ {jn + m}}_{\z_{m + 1} ^ {jn + m}} &\stackrel{(a)}\leq \sqrt{L} \norme{\z_{m + 1} ^ {jn + m + 1} - \z_{m + 1} ^ {jn + m}}, \nn \\
&= \sqrt{L} \norme{\z_{m + 1} ^ {jn + m + 1} - \x ^ \star - \z_{m + 1} ^ {jn + m} + \x ^ \star}, \nn \\
&\stackrel{{\Delta}}\leq \sqrt{L}\left(\norme{\z_{m + 1} ^ {jn + m + 1} - \x ^ \star} + \norme{\z_{m + 1} ^ {jn + m} - \x^\star}\right),\nn \\
&\stackrel{(b)}\leq \sqrt{L} \big(\rho ^ {j + 1} + \rho ^ {j}\big) \epsilon \leq 2 \sqrt{L}\rho ^ {j} \epsilon = \frac{2\alpha_j}{M} \nn,  
\end{align}
where \((a)\) follows from Assumption \ref{ass:smooth_strong} which implies \(\nabla ^ 2 f_{m + 1}(\z_{m + 1} ^ {jn + m}) \preceq L\I\), $(b)$ follows from the induction hypothesis and \eqref{induction_step_eff_1}. 
Therefore, our second observation is that 
$
\frac{Mr_{jn + m + 1}}{2} \le \alpha_j.
$

  We consider two cases depending on \(m\) (or equivalently \(t\)), similar to Stage 1.
  
  \textbf{{Case 1: \(0 \leq m < n\)}}

  Since all the \(\D_{i}\)'s were scaled by a factor \((1 + \alpha_j) ^ 2\) at the end of the cycle \(j - 1\), i.e., at \(t = j n\), we have \(\D_{m + 1} ^ {jn + m} = (1 + \alpha_{j}) ^ 2 \D_{m + 1} ^ {(j - 1)n + m + 1}\). Also, from the induction hypothesis, we have\begin{align*}
        \omega_{(j - 1)n + m + 1} ^ {-1} \D_{m + 1} ^ {(j - 1)n + m + 1} \succeq \nabla ^ 2 f_{m + 1}(\z_{m + 1} ^ {(j - 1)n + m + 1}).
  \end{align*}
  Note that if \((j - 1)n + m + 1 < 0\), we can simply assume those quantities to be super scripted/sub scripted (as appropriate) with \(0\). For example, \(\D_{m + 1} ^ {-n + m + 1} = \D_{m + 1} ^ {0}, \omega_{-n + m + 1} = \omega_0\), etc.

  Since \(0 \le m < n\), we have \(\omega_{(j - 1)n + m + 1} = 1\). Further, \(\z_{m + 1} ^ {jn + m} = \z_{m + 1} ^ {(j - 1)n + m + 1}\). Therefore, \begin{align*}
      \D_{m + 1} ^ {jn + m} = (1 + \alpha_{j}) ^ 2 \D_{m + 1} ^ {(j - 1)n + m + 1} &\succeq (1 + \alpha_j) ^ 2 \nabla ^ 2 f_{m + 1}(\z_{m + 1} ^ {jn + m}), \nn \\
      &\hspace{-3mm}\stackrel{\Lem. \ref{boundedness}}\succeq \big(1 + \alpha_j \big) ^ {2}\big(1 + \frac{Mr_{jn + m + 1}}{2}\big) ^ {-1}\K ^ {jn + m + 1}\nn, \\
      &\succeq \big(1 + \alpha_j \big)\K ^ {jn + m + 1}\nn.
  \end{align*}

\textbf{{Case 2: \(m = n - 1\)}}

Since the current index \(m + 1\) was last updated at time \(t = jn\), we have \(\D_{m + 1} ^ {jn + m} = \D_{m + 1} ^ {jn}\) and \(\z_{m + 1} ^ {jn + m} = \z_{m + 1} ^ {jn}\). Further, the induction hypothesis yields \(\omega_{jn} ^ {-1}\D_{m + 1} ^ {jn} \succeq \nabla ^ 2f_{m + 1}(\z^{jn}_ {m + 1})\). Also, \(\omega_{jn} = (1 + \alpha_{j}) ^ 2\) by definition. Therefore, \begin{align}\nn
    \D_{m + 1} ^ {jn + m} = \omega_{jn} (\omega_{jn} ^ {-1}\D_{m + 1} ^ {jn}) &\succeq  (1 + \alpha_{j}) ^ 2 \nabla ^ 2 f_{m + 1}(\z_{m + 1} ^ {jn + m}). \\
    &\hspace{-3mm}\stackrel{\Lem. \ref{integral_lemma}}\succeq\big(1 + \alpha_j \big) ^ {2}\big(1 + \frac{Mr_{jn + m + 1}}{2}\big) ^ {-1}\K ^ {jn + m + 1}\nn, \\
    &\succeq \big(1 + \alpha_j \big)\K ^ {jn + m + 1}\nn.
\end{align}

Summarizing, for both cases \(0 \leq m < n\) and \(m = n - 1\), we have established that \( \D_{m + 1} ^ {jn + m}  \succeq (1 + \alpha_{j}) \K ^ {jn + m + 1}\). The next steps are common for both the cases.

 Since \(\D_{m + 1} ^ {jn + m} \succeq \big(1 + \alpha_j \big)\K ^ {jn + m + 1},\) applying Lemma \ref{boundedness}, we obtain \begin{align*}
       \Q ^ {jn + m + 1} &= \bfgs{\D ^ {jn + m}, \big(1 + \alpha_j \big)\K ^ {jn + m + 1}, \z_{m + 1} ^ {jn + m + 1} - \z_{m + 1} ^ {jn + m + 1}}, \nn\\
       &\stackrel{\Lem. \ref{boundedness}}\succeq \big(1 + \alpha_j \big)\K ^ {jn + m + 1}.
  \end{align*}
  Applying Lemma \ref{integral_lemma} to relate \(\K ^ {jn + m + 1}\) and \(\nabla ^ 2 f_{m + 1}(\z_{m + 1} ^ {jn + m})\), we obtain \begin{align*}
      \Q ^ {jn + m + 1} \succeq \big(1 + {\alpha_j}\big) \K ^ {jn + m + 1} \succeq \big(1 + \frac{Mr_{jn + m + 1}}{2}\big) \K ^ {jn + m + 1} \stackrel{\Lem. \ref{boundedness}}\succeq \nabla ^ 2 f_{i_m}(\z_{i_m} ^ m).
  \end{align*}
  Since \(\Q ^ {jn + m + 1} \succeq \nabla ^ 2 f_{m + 1}(\z_{m + 1} ^ {jn + m + 1})\), applying Lemma \ref{boundedness}, we obtain \begin{align}\nn
      \omega_{jn + m + 1} ^ {-1}\D_{m + 1} ^ {jn + m + 1} &= \bfgs{\Q ^ {jn + m + 1}, \nabla ^ 2 f_{m + 1}(\z_{m + 1} ^ {jn + m + 1}), \bar{\u}(\Q ^ {jn + m + 1}, \nabla ^ 2 f_{m + 1}(\z_{m + 1} ^ {jn + m + 1})},\nn \\
      &\stackrel{\Lem. \ref{boundedness}}\succeq \nabla ^ 2 f_{m + 1}(\z_{m + 1} ^ {jn + m + 1})\nn.
  \end{align}
  
   We can now prove the linear convergence of the Hessian approximation.
% As the final observation, 
% recall from Section \ref{section_siqn} in the main paper, 
% that we can rewrite $\Q^ {jn + m + 1}$, \vspace{2mm}
% \begin{align*}
%     \Q ^ {jn + m + 1} 
%     &= \bfgs{(1 + \alpha_j) ^ 2\D_{m + 1} ^ {jn + m}, (1 + \alpha_j) ^ 2\K ^ {jn + m + 1}, \z_{m + 1} ^ {jn + m + 1} - \z_{m + 1} ^ {jn + m}},
%     \\[3pt]
%     &=\big(1 + \alpha_j) ^ 2 \bfgs{\D_{m + 1} ^ {jn + m}, \K ^ {jn + m + 1}, \z_{m + 1} ^ {jn + m + 1} - \z_{m + 1} ^ {jn + m}}.
% \end{align*}

% \textcolor{blue}
% {
% Combining these three observations,
% \begin{align}
% 	  & \frac{1}{\big(1 + \alpha_{j}\big) ^ 2} \nabla ^ 2 f_{m + 1}(\z_{m + 1} ^ {jn + m + 1}) \preceq \frac{1}{1 + \alpha_j}\K ^ {jn + m + 1} \preceq \frac{1}{(1 + \alpha_j) ^ 2}\Q ^ {jn + m + 1}, \nn \\
% 	  & \hspace{20mm}\implies \nabla ^ 2 f_{m + 1}(\z_{m + 1} ^ {jn + m + 1}) \preceq \Q ^ {jn + m + 1},\nn                                                                                               \\
% 	  & \hspace{20mm}\implies \nabla ^ 2 f_{m + 1}(\z_{m + 1} ^ {jn + m + 1}) \stackrel{\Lem. \ref{boundedness}}\preceq\omega_{jn + m} ^ {-1}\D_{m + 1} ^ {jn + m + 1},\nn                                
% \end{align}
% we have the result \ref{prove_this}.
% }
We define the sequence \(\{\y_k\}\), for \(k = 0, \dots, j + 1\), such that 
\(\{\y_k\} = \{\x ^ 0, \z_{m + 1} ^ {m + 1}, \dots, \z_{m + 1} ^ {jn + m + 1}\}\). 
From the induction hypothesis and Stage $2$, 
we have that $\norme{\y_k - \x^\star} \le \rho^{k} \norme{\x^0 - \x^\star}$, for all $k$.
Since \(\{\y_k\}\) comes about from the application of BFGS updates with,
$r_{k} := 2\sqrt{L}\rho ^ {k - 1}\epsilon$, 
as described in the statement of Lemma \ref{lem:sig}, 
therefore $\{\y_k\}$ satisfies the conditions of Lemma \ref{lem:sig}. This implies that,
\begin{align}
\sigma(\omega_{jn + m + 1} ^ {-1} \D_{m + 1} ^ {jn + m + 1}, \nabla ^ {2}f_{m + 1}(\y_{j + 1})) \leq (1 - c) ^ {j + 1} \delta.\nn
\end{align}
Since \(\y_{j + 1} = \z_{m + 1} ^ {jn + m + 1}\), the proof is complete via induction.

%% file: appendices/Proof_mean_linear_eff_hybrid_IQN.tex
 We prove the Lemma for a generic iteration \(t = jn + m + 1\), for some \(j \geq 0\) and \( 0 \leq m < n\). 
 We restate a few observations derived in the proof of Lemma \ref{lem:convergence_efficient_h_iqn}.
First, we proved an upper bound on \(\Gamma ^ {jn + m} = \norme{\big(\sum_{i = 1} ^ {n} \D_{i} ^ {jn + m}\big) ^ {-1}}\), given by \begin{align}
		\Gamma ^ {jn + m} \leq \frac{1 + \rho}{n\mu}.
  \label{proof_l3_c3}
	\end{align}
	We also derived upper bounds \eqref{i=m+1_n-1}, \eqref{i=n} on \(\norme{\D_{i}^{jn + m} - \nabla ^ 2 f_i(\z_i^{jn + m})}\):  \begin{align}\nn
		\norm{\D_{i}^{jn + m} - \nabla ^ 2 f_i(\z_i^{jn + m})} 
        \leq \begin{cases}
			(1 + M \sqrt{L}\epsilon \rho ^ j) ^ {2}L\delta(1 - c)^ {j} + ML ^ {\frac{3}{2}}\epsilon \rho ^ {j}(2 + M\sqrt{L}\epsilon \rho ^ j) & i \in [n] \backslash [m],\\
			L\delta (1 - c) ^ {j + 1} & i \in [m].
		\end{cases}
	\end{align}
	A common larger upper bound for both the cases is given by
 \begin{align}
     \norm{\D_{i}^{jn + m} - \nabla ^ 2 f_i(\z_i^{jn + m})} \leq (1 + M \sqrt{L}\epsilon \rho ^ j) ^ {2}L\delta(1 - c)^ {j} + ML ^ {\frac{3}{2}}\epsilon \rho ^ {j}(2 + M\sqrt{L}\epsilon \rho ^ j),
     \:\: i \in [n].
     \label{proof_l3_c1}
 \end{align}
 Finally, we also established that,
 \begin{align}\label{proof_l3_c2}
     \norme{\z_i ^ {jn + m} - \x^\star} 
     &\leq \rho ^ {j + 1}\norme{\x^0 - \x^\star} 
     \leq \rho ^ {j}\norme{\x^0 - \x^\star},
     \:\: i \in [m]\nn \\
     \norme{\z_i ^ {jn + m} - \x^\star} 
     &\leq \rho ^ {j}\norme{\x^0 - \x^\star}, 
     \:\: i \in [n]\backslash[m].
 \end{align}
We are now ready to prove the mean-superlinear convergence.
From Lemma \ref{lem:distance_change_single_iter_ver_II}, we have 
\begin{multline*}
		\norme{\z_{m + 1} ^ {jn + m + 1} - \x^\star} 
        \leq 
        \Gamma ^{jn + m}\tfrac{\tilde{L}}{2} \sum_{i = 1} ^ {n} \norme{\z_{i} ^ {jn + m} - \x^\star} ^ 2 
        + \Gamma ^ {jn + m} \sum_{i = 1} ^ {n} \norm{\D_i ^ {jn + m} - \nabla ^2f_{i}(\z_i ^ {jn + m})}\norme{\z_{i}^{jn + m} - \x^\star}
        \nn,
        \\
       \hspace{2.55cm}
       \leq 
        \Gamma ^{jn + m}
        \sum_{i = 1} ^ {n} 
        \Big(
        \tfrac{\tilde{L}}{2} \norme{\z_{i} ^ {jn + m} - \x^\star} + \norm{\D_i ^ {jn + m} - \nabla ^2f_{i}(\z_i ^ {jn + m})}
        \Big)
        \norme{\z_{i}^{jn + m} - \x^\star}
        \nn,
\end{multline*}
We now substitute the upper bounds from \ref{proof_l3_c3}, \ref{proof_l3_c1}, and \ref{proof_l3_c2} to get
\begin{align}
		&\leq \tfrac{1 + \rho}{n\mu} \big(\tfrac{\tilde{L}}{2}\rho ^ {j} \epsilon + (1 + M \sqrt{L}\epsilon \rho ^ j) ^ {2}L\delta(1 - c)^ {j} + ML ^ {\frac{3}{2}}\epsilon \rho ^ {j}(2 + M\sqrt{L}\epsilon \rho ^ j)\big)\sum_{i = 1} ^ {n}\norme{\z_i ^ {jn + m} - \x^\star}\nn, \\
		&\leq \tfrac{1 + \rho}{n\mu}\big(\tfrac{\tilde{L}}{2}\rho ^ {j} \epsilon + (1 + M \sqrt{L}\epsilon) ^ {2}L\delta(1 - c)^ {j} + ML ^ {\frac{3}{2}}\epsilon \rho ^ {j}(2 + M\sqrt{L}\epsilon)\big)\sum_{i = 1} ^ {n}\norme{\z_i ^ {jn + m} - \x^\star}\nn,\\
		&\stackrel{(a)}\leq (1 - c) ^ j\tfrac{1 + \rho}{\mu}\big(\tfrac{\tilde{L}}{2} \epsilon + (1 + M \sqrt{L}\epsilon) ^ {2}L\delta + ML ^ {\frac{3}{2}}\epsilon(2 + M\sqrt{L}\epsilon)\big)\big(\tfrac{1}{n}\sum_{i = 1} ^ {n}\norme{\z_i ^ {jn + m} - \x^\star}\big)\nn, \\
		&\stackrel{\eqref{choice_variables_II}}\leq \rho (1 - c) ^ {j} \big(\tfrac{1}{n}\sum_{i = 1} ^ {n}\norme{\z_i ^ {jn + m} - \x^\star}\big) \stackrel{(b)}\leq (1 - c) ^ {j + 1} \big(\tfrac{1}{n}\sum_{i = 1} ^ {n}\norm{\x ^ {t - i} - \x^\star}\big)\nn,
	\end{align}
	where \((a)\) and \((b)\) follow since \(\rho < 1 - c\). This completes the proof.

%% file: appendices/Proof_superlinear_hybrid_IQN.tex
Define the sequence \(\{\zeta ^ t\}\), for all \(t \geq 0\), as 
\(
	\zeta ^ {t} \coloneqq \max_{j \in [n]} \norme{\x ^ {nt + j} - \x^\star}.
\)
	Let the constant $c$ be defined as \(c := \frac{\mu}{dL}\).  
 Then, from Lemma \ref{lem:mean_linear_verII}, we have \begin{align}
		\norm{\x ^ {nt + i} - \x^\star} 
        &\leq (1 - c) ^ {t + 1}\tfrac{1}{n} \sum_{j = 1} ^ {n} \norme{\x^{nt + i - j} - \x^\star} 
        \\
        &\leq (1 - c) ^ {t + 1} \max_{j \in [n]} \norme{\x^{nt + i - j} - \x^\star}.
        \label{seq_eqn}
	\end{align}
	By induction on \(i\), we prove the following first:\begin{align}\label{induction_hyp}
		\norm{\x^{nt + i} - \x^\star} \leq (1 - c) ^ {t + 1} \max_{j = 1, \dots, n} \norm{\x^{nt + 1 - j} - \x^\star} = (1 - c) ^ {t + 1} \zeta ^ {t - 1},
	\end{align}
	for \(i = 1, \dots, n\). Substituting \(i = 1\) in \eqref{seq_eqn} we get \begin{align}\nn
		\norm{\x^{nt + 1} - \x^\star} \leq (1 - c) ^ {t + 1} \max_{j = 1, \dots, n} \norm{\x^{nt + 1 - j} - \x^\star} = (1 - c) ^ {t + 1} \zeta ^ {t - 1}.
	\end{align}
	This proves the base step (\(i = 1\)). Assume, \eqref{induction_hyp} holds for \(i = k\). We prove that \eqref{induction_hyp} holds for \(i = k + 1\). Substituting \(i = k + 1\) in \eqref{seq_eqn} we get \begin{align}\nn
		\norm{\x^{nt + k + 1} - \x^\star}&\leq (1 - c) ^ {t + 1} \max_{j = 1, \dots, n} \norm{\x^{nt + k + 1 - j} - \x^\star},\\
		&\leq (1 - c) ^ {t + 1} \max_{j = 1, \dots, n + k} \norm{\x^{nt + k + 1 - j} - \x^\star},\nn\\
		&= (1 - c) ^ {t + 1} \max{\big(\norm{\x^{nt + k} - \x^\star}, \dots, \norm{\x^{nt + 1} - \x^\star},  \max_{j = 1, \dots, n}\norm{\x^{nt + 1 - j} - \x^\star}\big)}\nn, \\
		&\stackrel{(a)}\leq (1 - c) ^ {t + 1} \max_{j = 1, \dots, n} \norm{\x^{nt + 1 - j} - \x^\star}\nn,
	\end{align}
	where \((a)\) follows since \(\norm{\x ^ {nt + 1} - \x^\star} \leq (1 - c) ^ {t + 1}\max_{j = 1, \dots, n}\norm{\x^{nt + 1 - j} - \x^\star}, \dots, \norm{\x ^ {nt + k} - \x^\star} \leq (1 - c) ^ {t + 1}\max_{j = 1, \dots, n}\norm{\x^{nt + 1 - j} - \x^\star}\) by the induction hypothesis. Therefore \eqref{induction_hyp} holds for \(i = k + 1\). This proves \eqref{induction_hyp} for \(i = 1, \dots, n\).
	
	Since \eqref{seq_eqn} holds for \(i = 1, \dots, n\), we have \begin{align}\label{zeta_eq}
		\zeta ^ {t} \leq (1 - c) ^ {t + 1} \zeta ^ {t - 1},
	\end{align}
for all \(t \geq 1\).
%	Since \(\zeta_{nt + 1} = \dots, \zeta_{n(t + 1)}\), we can equivalently express \eqref{zeta_eq} as \(\zeta ^ {n(t + 1)} \leq (1 - c)^{t + 1} \zeta ^ {nt}\). 
	Unrolling this recursion, we get \begin{align}\nn
		\zeta ^ {t} \leq (1 - c) ^ {\sum_{j = 2} ^ {t + 1} j} \zeta ^ {0} \stackrel{(a)}\leq (1 - c) ^ {\frac{t(t + 3)}{2}}\rho \norm{\x^0 - \x^\star} \stackrel{(b)}\leq \epsilon(1 - c) ^ {\frac{(t + 1)(t + 2)}{2}},
	\end{align}
	where \((a)\) follows from Lemma \ref{lem:convergence_efficient_h_iqn} and \((b)\) follows since \(\rho < 1-  c\). This completes the proof.
	%Again, substituting \(i = 2\) in \eqref{seq_eqn} we get \begin{align}\nn
		%	\norm{\x ^ {nt + 2} - \x^\star} & \leq (1 - c) ^ {t + 1} \max_{j = 1, \dots, n} \norm{\x^{nt + 2 - j} - \x^\star}, \\
		%	&\leq (1 - c) ^ {t + 1} \max_{j = 1, \dots, n + 1} \norm{\x^{nt + 2 - j} - \x^\star}, \nn\\
		%	&= (1 - c) ^  {t + 1} \max{\big(\norm{\x^{nt + 1} - \x^\star}, \max_{j = 1, \dots, n} \norm{\x^{nt + 1 - j} - \x^\star}\big)}\nn, \\
		%	&\leq (1 - c) ^ {t + 1} \max{\big((1 - c) ^ {t + 1}\max_{j = 1, \dots, n} \norm{\x^{nt + 1 - j} - \x^\star}, \max_{j = 1, \dots, n} \norm{\x^{nt + 1 - j} - \x^\star}\big)}\nn, \\
		%	&= (1 - c) ^ {t + 1} \max_{j = 1, \dots, n} \norm{\x^{nt + 1 - j} - \x^\star}.\nn
		%\end{align}

%% file: appendices/broyden_class_extension.tex
In this section, we extend the SLIQN algorithm, whose Hessian approximation updates \ref{eff_first_modified_update_}, \ref{eff_second_modified_update}
are built on the $\text{BFGS}$ operator, to the class of  \textit{restricted Broyden} operators. 
We refer to this class of algorithms as 
G-SLIQN.
First, we define the DFP operator
\begin{align*}
	\dfp{\B, \K, \z}
	\coloneqq
	\B -
	\frac{\K\z{\z}^\T\B + \B\z{\z}^\T\K} {\ip{\z}{\K\z}} + 
	\left( 1 + \frac{\ip{\z}{\B\z}}{\ip{\z}{\K\z}} \right)       \frac{\K\z{\z}^\T\K}{\ip{\z}{\K\z}},
\end{align*}
for \(\B, \K \succ \mathbf{0}\) and \(\z \in \Rn ^ d\backslash \{\mathbf{0}\}\). The {restricted Broyden} operator with parameter $0 \le \tau \le 1$ is defined as a convex combination of DFP and BFGS as follows: %\textcolor{blue}{please propagate notational changes such that the gradient and variable variation in \eqref{eq:var_grad_variation} matches the below updates.}
\begin{align*}
	\Br{\tau}{ \B,\K,\z}{\textit{res}} &\coloneqq 
	\tau\text{DFP}\left( \B,\K,\z \right) + 
	(1-\tau)\text{BFGS}\left( \B,\K,\z \right),
\end{align*}
where \(\text{BFGS}(\B, \K, \z)\) is given by \eqref{bfgs}.

\subsection{The generalized algorithm G-SLIQN}
\label{g-sliqn_alg}
As with SLIQN, we first present G-SLIQN with a maximum per-iteration cost of $\cO(nd^2)$ and an average \(\cO(d ^ 2)\) cost per epoch. 
Similar to SLIQN, G-SLIQN can be implemented with a maximum per-iteration cost of $\cO(d^2)$ as per Appendix \ref{update_efficiently}.

\textbf{Hyperparameters:} Choose $\tau_1,\tau_2 \in [0,1]$ as the restricted Broyd operator paramater for the classic and the greedy updates respectively. Note that setting $\tau_1=\tau_2=0$ reduces G-SLIQN to SLIQN.

We denote the Hessian approximation matrices at time \(t\) as \(\{\G_{i} ^ t\}_{i = 1} ^ n\).

\textbf{Initialize:} At \(t = 0\), we initialize \(\{\z_i ^ {0}\}_{i = 1} ^ {n}\) as \(\z_i ^ {0} = \x ^ 0\), \(\forall i \in [n]\), for a suitably chosen \(\x ^ 0\). We initialize \(\{\G_{i} ^ 0\}_{i = 1} ^ n\) as \(\G_{i} ^ 0 =  (1 + \alpha_0) ^ 2\I_i ^ 0\), where \(\{\I_i ^ 0\}_{i = 1} ^ n\) are chosen such that \(\I_{i} ^ 0 \succeq \nabla ^ 2 f_i (\z_i ^ 0)\), \(\forall i \in [n]\). Here \(\{\alpha_k\}, k \in \mathbb{N}\) is as defined in Lemma \ref{lem:convergence_efficient_h_iqn}.

\textbf{Algorithm:} For any iteration \(t \geq 1\), just like the update in SLIQN \ref{x_new_update}, we update $\x^t$ as
\begin{align}\label{g_x_new_update}
	\x ^ t = \bigg(\sum_{i = 1} ^ {n} \G_{i} ^ {t - 1}\bigg) ^ {-1}\bigg(\sum_{i = 1} ^ {n} \G_{i} ^ {t - 1}\z_{i} ^ {t - 1} - \sum_{i = 1} ^ {n} \nabla f_{i}(\z_{i} ^ {t - 1})\bigg).
\end{align} 
Next, we update \(\z_{i_t} ^ t\) as \(\z_{i_t} ^ t = \x^t\).
To update \(\Q ^ {t}\) and \(\G_{i_t} ^ t\), we use the chosen restricted Broyd operators in place of the BFGS operators \ref{eff_first_modified_update_}, 
\ref{eff_second_modified_update}: \begin{align}\label{g_eff_first_modified_update_}
	\Q ^ {t} &= \Br{\tau_1}{\G_{i_t} ^ {t - 1}, (1 + \alpha_{\ceil{t / n - 1}}) \K ^ t, \z_{i_t} ^ t - \z_{i_t} ^ {t - 1}}{\textit{res}}, \\
	\G_{i_t} ^ t &= \omega_t\Br{\tau_2}{\Q ^ {t}, \nabla ^ 2 f_{i_t}(\z_{i_t} ^ t), \bar{\u} ^ t (\Q^t, \nabla ^ 2 f_{i_t}(\z_{i_t} ^ t))}{\textit{res}} \label{g_eff_second_modified_update},
\end{align}
where \(\omega_t := (1 + \alpha_{\ceil{t / n}}) ^ 2\) if \(t\) is a multiple of \(n\) and \(1\) otherwise. For the indices \(i \neq i_t\), we update \(\z_{i} ^ {t}\) and \(\G_i ^ t\) in the following manner:
\begin{align}\label{g_gamma_i_neq_i_t}
	\z_i ^ t = \z_i ^ {t - 1}, \G_i ^ t = \omega_t \G_i ^ {t - 1}, \forall i \in [n]; i \neq i_t.
\end{align}

Finally, we update 
$
\left(\sum_{i = 1} ^ {n} \G_{i} ^ {t}\right) ^ {-1}
$
, $
\sum_{i = 1} ^ {n} \G_{i} ^ {t}\z_{i} ^ {t}
$ 
and
$\sum_{i = 1} ^ {n} \nabla f_{i}(\z_{i} ^ {t})
$.
Observe that the restricted Broyd operator induces a correction of at-most $5$ rank-$1$ matrices. 
We can therefore, carry out this update in $\cO(d^2)$ cost
by repeatedly applying Sherman-Morrison formula. This is similar to what was done for SLIQN in Appendix \ref{update_efficiently}. All other updates are a constant number of matrix-vector multiplications which can be done in \(\cO(d ^ 2)\) cost.

\begin{algorithm}
	\caption{Generalized Sharpened Lazy Incremental Quasi-Newton (G-SLIQN)}
	\begin{algorithmic}[1]
		% \STATE\textbf{Input:} Sequence \(\{\omega_k\}, k \in \mathbb{N
			% }\);
		\STATE \textbf{Function} \{Sherman-Morrison\} \{$\A^{-1}, \u, \v$\}
		\STATE \hspace{3mm} \textbf{return} $\A^{-1} - \tfrac{\A^{-1}\u\v^T\A^{-1}}{1 + \v^T\A^{-1}\u}$
		\STATE \textbf{EndFunction}
		\vspace{3mm}
		\STATE\textbf{Initialize:} Initialize \(\{\z_{i}, \G_{i}\}_{i = 1} ^ {n}\) as described in Section \ref{g-sliqn_alg};
		\STATE Evaluate \(\bar{\G} \coloneqq \big(\sum_{i}\G_i\big) ^ {-1}, \bm{\phi} \coloneqq  \sum_{i} \G_i \z_i\), and \(\g  \coloneqq \sum_i \nabla f_i (\z_i)\);
		\STATE\textbf{while} \textit{not converged:}
		\STATE\hspace{3mm}Current index to be updated is \(i_{t} \leftarrow (t - 1) \mod n + 1;\)
		\STATE\hspace{3mm}Update $\omega_t$ as $\omega_t \leftarrow 1 + \alpha_{\ceil{t/n - 1}}$;
		\IF {$i_t = 1$} 
		\STATE Update $\bar{\G}$ as $\bar{\G} \leftarrow \bar{\G}/\omega_t^2$;
		\STATE Update $\bm{\phi}$ as $\bm{\phi} \leftarrow \omega_t^2\bm{\phi}$;
		\ENDIF 
		%			\STATE\hspace{3mm}\textbf{if} \(i_{t} = 1\)
		%			\STATE\hspace{3mm}\hspace{3mm}\textbf{for} \(j = 1, \dots, n\)
		%			\STATE\hspace{3mm}\hspace{3mm}\hspace{3mm} Reset \(\D_{j} ^ {t - 1} = (1 + M \sqrt{L}\rho ^ {\floor{\frac{t}{n}}}\epsilon) ^ {2}\B_{j} ^ {t - 1};\)
		%			\STATE\hspace{3mm}\hspace{3mm}\textbf{end for}
		%			\STATE\hspace{3mm}\textbf{end if}
		\STATE\hspace{3mm}Update \(\x^t\) as  \(\x^t \leftarrow \big(\bar{\G}\big)\big(\bm{\phi}  - \g\big)\) as per (\ref{g_x_new_update});
		
		%			\STATE\hspace{3mm}Set \(\Q^t = \bfgs{\D_{i_t}^{t - 1}, (1 + M\sqrt{L}\rho ^ {\floor{\frac{t}{n}}}\epsilon) ^ 2 \K ^ t, \z_{i_t}^t - \z_{i_{t}}^{t - 1}}\);
		\STATE\hspace{3mm}Update $\G_{i_t}$ as  $\G_{i_t} \leftarrow \omega_t^2 \G_{i_t}$;
		\STATE\hspace{3mm}Update $\v_1$ as 
		$\v_1 \leftarrow \x ^ t - \z_{i_t}$;
		\STATE\hspace{3mm}Update $\Q_{i_t}$ as \(\Q_{i_t} \leftarrow \Br{\tau_1}{\G_{i_t},
			\omega_t\K ^ t, \v_1}{\textit{res}}\) as per \eqref{g_eff_first_modified_update_}; 
		\STATE\hspace{3mm}Update $\v_2$ as 
		$\v_2 \leftarrow \bar{\u} ^ t (\Q_{i_t}, \nabla ^ 2 f_{i_t}(\x^t))$;
		\STATE \hspace{3mm}Update $\tilde{\G}_{i_t}$ as $\tilde{\G}_{i_t} \leftarrow \Br{\tau_2}{\Q_{i_t}, \nabla ^ 2 f_{i_t}(\x^t), \v_2}{\textit{res}}$  as per \eqref{g_eff_second_modified_update};
		\STATE\hspace{3mm}Update \(\bi{\phi}\) as \(\bi{\phi} \leftarrow \bm{\phi}  - \G_{i_t}\z_{i_t} + \tilde{\G}_{i_t} \x^t\);
		
		\STATE\hspace{3mm}Update \(\g\) as  \(\g \leftarrow \g - \nabla f_{i_t}(\z_{i_t}) + \nabla f_{i_t}(\x^t)\);
		\STATE\hspace{3mm}Update \(\bar{\G}\) as 
		\(\bar{\G} \leftarrow \text{Sherman-Morrison}
		(
		\bar{\G},
		-\tfrac{\K\v_1}{\v_1^T\K\v_1},
		\G_{i_t}\v_1
		)
		\);
		\STATE\hspace{3mm}Update \(\bar{\G}\) as 
		\(\bar{\G} \leftarrow \text{Sherman-Morrison}
		(
		\bar{\G},
		-\G_{i_t}\v_1,
		\tfrac{\K\v_1}{\v_1^T\K\v_1}
		)
		\);
		\STATE\hspace{3mm}Update \(\bar{\G}\) as 
		\(\bar{\G} \leftarrow \text{Sherman-Morrison}
		(
		\bar{\G},
		\omega_t\left(1 + \tfrac{\v_1^T\G_{i_t}\v_1}{\omega_t^2 \v_1^T\K\v_1}\right)
		\tfrac{\K\v_1}{\v_1^T\K\v_1},
		\tfrac{\K\v_1}{\v_1^T\K\v_1}
		)
		\);
		\STATE\hspace{3mm}Update \(\bar{\G}\) as 
		\(\bar{\G} \leftarrow \text{Sherman-Morrison}
		(
		\bar{\G},
		-\tfrac{\Q_{i_t}\v_2}{\v_2^T\Q_{i_t}\v_2},
		\Q_{i_t}\v_2
		)
		\);
		\STATE\hspace{3mm}Update \(\bar{\G}\) as 
		\(\bar{\G} \leftarrow \text{Sherman-Morrison}
		(
		\bar{\G},
		\tfrac{\nabla^2f_{i_t}(\x^t)\v_2}{\v_2^T\nabla^2f_{i_t}(\x^t)\v_2},
		\nabla^2f_{i_t}(\x^t)\v_2
		)
		\);
		\STATE\hspace{3mm}Update \(\z_{i_t}\) as \(\z_{i_t} \leftarrow \x^t\);
		\STATE Increment the iteration counter \(t\);
		\STATE\textbf{end while}
		
	\end{algorithmic}
	\label{alg:sl_iqn}
\end{algorithm}

\subsection{Overview of the Convergence Analysis of G-SLIQN}

The analysis of SLIQN can be readily extended to G-SLIQN. Since most of the results established for SLIQN would continue to hold for G-SLIQN, we do not explictly state them here for the sake of brevity. However, we discuss the mappings that allow us to conclude that similar results hold for SLIQN in this section. 

{Firstly, the result in Lemma \ref{boundedness} holds even for the restricted Broyden operator as per \cite[Lemma 2.1]{rodomanov2021rates} and \cite[Lemma 2.2]{rodomanov2021rates}, which are restated here for completeness. Note that the results in \cite{rodomanov2021rates} are for the Broyden operator, but as per \cite{nocedal} the restricted Broyden operator is a subset of the Broyden operator, hence the results in \cite{rodomanov2021rates} are applicable for the restricted Broyden operator as well.}

\begin{lemma}\cite[Lemma 2.1]{rodomanov2021rates}\label{boundedness_broyd}
	Let, \(\G, \A\) be positive definite matrices such that \(\frac{1}{\xi} \A \preceq \G \preceq \A\), where \(\xi, \eta \geq 1\). Then, for any \(\u \neq 0\), and any \(\tau \in [0, 1]\), we have \begin{align}\nn
		\frac{1}{\xi}\A \preceq \G_{+} := \Br{\tau}{\G, \A, \u}{\textit{res}} \preceq \eta \A.
	\end{align}
\end{lemma} 

\begin{lemma}\cite[Lemma 2.2]{rodomanov2021rates}\label{sigma_dec_broyd}
	Let \(\G, \A\) be positive definite matrices such that \(\A \preceq \G \preceq \eta \A\), for some \(\eta \ge 1\). Then, for any \(\tau \in [0, 1]\) and any \(\u \neq \mathbf{0}\), we have \begin{align}\nn
		\sigma(\G, \A) - \sigma(\Br{\tau}{\G, \A, \u}{\textit{res}}, \A) \geq \bigg(\frac{\tau}{\eta} + 1 - \tau\bigg) \theta ^ 2 (\G, \A, \u),
	\end{align}
	where \begin{align}\theta(\G, \A, \u) := \bigg(\frac{\ip{(\G - \A)\u}{\A ^ {-1}(\G - \A)\u}}{\ip{\G\u}{\A^{-1}\G\u}}\bigg) ^ {\frac{1}{2}}\nn.\end{align}
\end{lemma}

Since \(\tau \in [0, 1]\) and \(\eta \geq 1\), we have \(\frac{\tau}{\eta} + 1 - \tau \geq 0 \iff \tau \leq \frac{\eta}{\eta - 1}\) holds vacuously. Therefore, on taking a restricted Broyden update, we get \(\sigma(\G, \A) \geq \sigma(\Br{\tau}{\G, \A, \u}{\textit{res}}, \A)\) under the assumptions of Lemma \ref{sigma_dec_broyd}.

Futher, the result in Lemma \ref{one_greedy} holds for the restricted Broyd operator as per \cite[Theorem 2.5]{rodomanov2021greedy} which is restated here for completeness.
\begin{lemma}\cite[Theorem 2.5]{rodomanov2021greedy}\label{sigma_contraction_broyd}
	Let \(\G, \A\) be positive definite matrices such that \(\A \preceq \G\). Further, let \(\mu, L > 0\) be such that \(\mu \I \preceq \A \preceq L \I\). Then, for any \(\tau \in [0, 1]\), we have \begin{align}\nn
		\sigma(\Br{\tau}{\G, \A, \bar{\u}(\G, \A)}{\textit{res}}, \A) \leq \bigg(1 - \frac{\mu}{dL}\bigg) \sigma(\G, \A),
	\end{align}
	where \(\bar{\u}(\G, \A)\) the greedy vector \eqref{greedy_vec}.
\end{lemma}

Using Lemma \ref{boundedness_broyd}, \ref{sigma_dec_broyd}, \ref{sigma_contraction_broyd}, we can establish that the Lemma \ref{lem:unit_step_dec} (with BFGS replaced by restricted Broyd), Corollary \ref{cor:imp}, and Lemma \ref{lem:sig} (with BFGS replaced by restricted Broyd) hold. Therefore, the supporting lemmas in Appendix \ref{app:supporting_lemma} hold even for the Broyden update. 

Next, we discuss about the main results in Section \ref{convergence}. Firstly, observe that Lemma \ref{lem:distance_change_single_iter_ver_II} remains the same for G-SLIQN. This is because the proof of Lemma \ref{lem:distance_change_single_iter_ver_II}
hinges on the structure of update \eqref{x_new_update} and abstracts out the specific updates made to $\D_{i_t}^t$. Since the update \eqref{g_x_new_update} for G-SLIQN is the same as \eqref{x_new_update}, the guarantees of the Lemma and its proof carries through for G-SLIQN. Further, since the supporting lemmas in Appendix \ref{app:supporting_lemma} hold for G-SLIQN, using Lemma \ref{lem:distance_change_single_iter_ver_II} and the supporting lemmas, we can establish that Lemma \ref{lem:convergence_efficient_h_iqn} holds even for G-SLIQN. Since the proof of Lemma \ref{lem:mean_linear_verII} leverages the result of Lemma \ref{lem:convergence_efficient_h_iqn}, we can establish that the mean superlinear convergence result given by Lemma \ref{lem:mean_linear_verII} holds for G-SLIQN as well. Finally, using Lemma \ref{lem:mean_linear_verII} we can show that Theorem \ref{lem:final_rate} holds for G-SLIQN.

%In particular, the guarantees established in Lemma \ref{lem:distance_change_single_iter_ver_II}, \ref{lem:convergence_efficient_h_iqn}, 
%\ref{lem:mean_linear_verII}
%and \ref{lem:final_rate}
%hold exactly for G-SLIQN.
%Moreover, only the proof of Lemma \ref{lem:convergence_efficient_h_iqn} requires some modification; the proof of other results do not need to be altered.
%
%
%To see this, observe that the proof of Lemma \ref{lem:distance_change_single_iter_ver_II}
%hinges on the structure of update \ref{x_new_update} and and abstracts out the specific updates made to $\D_{i_t}^t$. Since the update \ref{g_x_new_update} for G-SLIQN is the same as \ref{x_new_update}, the guarantees of the lemma and its proof carries through for G-SLIQN.
%
%We now show that the claims in Lemma 
%\ref{lem:convergence_efficient_h_iqn} also hold for G-SLIQN.
%
%
%Similar to \ref{lem:mean_linear_verII} and Lemma \ref{lem:final_rate}

%% file: appendices/additional_numerical_simulations.tex
As can be clearly observed, the proposed algorithm SLIQN requires the knowledge about \(\epsilon, \sigma_0\) in order to tune the correction factor \(\alpha_t\). However, we observed that, empirically SLIQN outperforms a number of incremental and stochastic QN methods without the correction factor, i.e., \(\alpha_t = 0\). For IGS however, the performance is quite sensitive to the correction factor, $\beta_t$, and \(\beta_t = 0\) was not the best performing correction factor for all the simulations. Therefore, SLIQN does not require hyper-parameter tuning, unlike IGS.

\textbf{Initialization:} For all our simulations, all algorithms start at the same initial \(\x_{0} = \alpha \v\), where \(\v \in \Rn ^ d\) is such that each coordinate \(v_{i, j} \sim \text{Unif}[0, 1]\). Since, all the algorithms considered for performance comparison are only locally convergent, the parameter \(\alpha\) affects the convergence of the algorithms.

\textbf{Stopping Criterion:} We stop the execution of each algorithm when the gradient norm of \(f\) is less than a threshold. Formally, letting the threshold be \textit{gstop}, the stopping condition can be expressed as \begin{align*}
	\frac{1}{N}\norm{\sum_{i = 1} ^ {N} \nabla f_{i}(\x ^ t)} < \textit{gstop}.
\end{align*}
Typical values of \textit{gstop} used in our simulations range from \(10 ^ {-7}\) to \(10 ^ {-8}\).
\subsection{Generating Scheme for Quadratic minimization}\label{app:numerical_simulations}
We follow the scheme proposed in \cite{mokhtari2018iqn} to generate \(\{\A_i, \bm{b}_i\}_{i = 1} ^ n\).
We set
each matrix
\(\A_i \coloneqq \mathrm{diag}(\{\a_{i}\}_{i = 1} ^ d) \), 
% as \(\A_i := \mathrm{diag}(\{\a_{i}\}_{i = 1} ^ d)\), 
by sampling the diagonal elements as  
\(
\{\a_{i}\}_{i = 1} ^ {d/2} 
\overset{\text{i.i.d.}}{\sim} 
\text{Unif}[1, 10 ^ {\frac{\xi}{2}}]
\) 
and 
\(
\{\a_{i}\}_{i = d/2+1} ^ {d} 
\overset{\text{i.i.d.}}{\sim} 
\text{Unif}[10 ^ {-\frac{\xi}{2}}, 1]
\).  
The parameter \(\xi\) controls the condition number of the matrix \(\A_i\).  
Under the limit $d \rightarrow \infty$, the condition number of \(\A_i\) is given by $10^{\xi}$. 
Each coordinate $b_{i, j}$ of the vector \(\b_i\) is sampled as 
\(b_{i, j} \sim \text{Unif}[0, 1000]\).

%% file: sliqn_arxiv.bbl
\begin{thebibliography}{}

\bibitem[Byrd et~al., 2016]{byrd2016stochastic}
Byrd, R.~H., Hansen, S.~L., Nocedal, J., and Singer, Y. (2016).
\newblock A stochastic quasi-newton method for large-scale optimization.
\newblock {\em SIAM Journal on Optimization}, 26(2):1008--1031.

\bibitem[Chang and Lin, 2011]{libsvm}
Chang, C.-C. and Lin, C.-J. (2011).
\newblock Libsvm: A library for support vector machines.
\newblock {\em ACM Trans. Intell. Syst. Technol.}, 2(3).

\bibitem[Chang et~al., 2019]{chang2019accelerated}
Chang, D., Sun, S., and Zhang, C. (2019).
\newblock An accelerated linearly convergent stochastic l-bfgs algorithm.
\newblock {\em IEEE transactions on neural networks and learning systems}, 30(11):3338--3346.

\bibitem[Chen et~al., 2022]{san}
Chen, J., Yuan, R., Garrigos, G., and Gower, R.~M. (2022).
\newblock San: Stochastic average newton algorithm for minimizing finite sums.
\newblock In Camps-Valls, G., Ruiz, F. J.~R., and Valera, I., editors, {\em Proceedings of The 25th International Conference on Artificial Intelligence and Statistics}, volume 151 of {\em Proceedings of Machine Learning Research}, pages 279--318. PMLR.

\bibitem[Defazio et~al., 2014]{saga}
Defazio, A., Bach, F.~R., and Lacoste-Julien, S. (2014).
\newblock Saga: A fast incremental gradient method with support for non-strongly convex composite objectives.
\newblock In {\em Neural Information Processing Systems}.

\bibitem[Derezinski, 2023]{derezinski2022stochastic}
Derezinski, M. (2023).
\newblock Stochastic variance-reduced newton: Accelerating finite-sum minimization with large batches.
\newblock In {\em OPT 2023: Optimization for Machine Learning}.

\bibitem[Derezinski et~al., 2021]{derezinski2021newton}
Derezinski, M., Lacotte, J., Pilanci, M., and Mahoney, M.~W. (2021).
\newblock Newton-less: Sparsification without trade-offs for the sketched newton update.
\newblock {\em Advances in Neural Information Processing Systems}, 34:2835--2847.

\bibitem[Gao et~al., 2020]{gao2020incremental}
Gao, Z., Koppel, A., and Ribeiro, A. (2020).
\newblock Incremental greedy bfgs: An incremental quasi-newton method with explicit superlinear rate.
\newblock In {\em Adv. Neural Inf. Process. Syst. 12th OPT Workshop Optim. Mach. Learn.}

\bibitem[Gonen et~al., 2016]{gonen2016solving}
Gonen, A., Orabona, F., and Shalev-Shwartz, S. (2016).
\newblock Solving ridge regression using sketched preconditioned svrg.
\newblock In {\em International conference on machine learning}, pages 1397--1405. PMLR.

\bibitem[Jin et~al., 2022]{jin2022sharpened}
Jin, Q., Koppel, A., Rajawat, K., and Mokhtari, A. (2022).
\newblock Sharpened quasi-newton methods: Faster superlinear rate and larger local convergence neighborhood.
\newblock In {\em International Conference on Machine Learning}, pages 10228--10250. PMLR.

\bibitem[Jin and Mokhtari, 2022]{jin2022non}
Jin, Q. and Mokhtari, A. (2022).
\newblock Non-asymptotic superlinear convergence of standard quasi-newton methods.
\newblock {\em Mathematical Programming}, pages 1--49.

\bibitem[Johnson and Zhang, 2013]{johnson2013accelerating}
Johnson, R. and Zhang, T. (2013).
\newblock Accelerating stochastic gradient descent using predictive variance reduction.
\newblock {\em Advances in neural information processing systems}, 26.

\bibitem[Kovalev et~al., 2019]{kovalev2019stochastic}
Kovalev, D., Mishchenko, K., and Richt{\'a}rik, P. (2019).
\newblock Stochastic newton and cubic newton methods with simple local linear-quadratic rates.
\newblock {\em arXiv preprint arXiv:1912.01597}.

\bibitem[Li et~al., 2021]{mle_1}
Li, J., Bian, S., Zeng, A., Wang, C., Pang, B., Liu, W., and Lu, C. (2021).
\newblock Human pose regression with residual log-likelihood estimation.
\newblock In {\em Proceedings of the IEEE/CVF International Conference on Computer Vision (ICCV)}, pages 11025--11034.

\bibitem[Li et~al., 2020]{mle_2}
Li, J., Zhou, P., Xiong, C., Socher, R., and Hoi, S. C.~H. (2020).
\newblock Prototypical contrastive learning of unsupervised representations.
\newblock {\em CoRR}, abs/2005.04966.

\bibitem[Liu et~al., 2019]{liu2019acceleration}
Liu, Y., Feng, F., and Yin, W. (2019).
\newblock Acceleration of svrg and katyusha x by inexact preconditioning.
\newblock In {\em International Conference on Machine Learning}, pages 4003--4012. PMLR.

\bibitem[Mokhtari et~al., 2018]{mokhtari2018iqn}
Mokhtari, A., Eisen, M., and Ribeiro, A. (2018).
\newblock Iqn: An incremental quasi-newton method with local superlinear convergence rate.
\newblock {\em SIAM Journal on Optimization}, 28(2):1670--1698.

\bibitem[Mokhtari and Ribeiro, 2014]{mokhtari2014res}
Mokhtari, A. and Ribeiro, A. (2014).
\newblock Res: Regularized stochastic bfgs algorithm.
\newblock {\em IEEE Transactions on Signal Processing}, 62(23):6089--6104.

\bibitem[Mokhtari and Ribeiro, 2015]{mokhtari2015global}
Mokhtari, A. and Ribeiro, A. (2015).
\newblock Global convergence of online limited memory bfgs.
\newblock {\em The Journal of Machine Learning Research}, 16(1):3151--3181.

\bibitem[Moritz et~al., 2016]{moritz2016linearly}
Moritz, P., Nishihara, R., and Jordan, M. (2016).
\newblock A linearly-convergent stochastic l-bfgs algorithm.
\newblock In {\em Artificial Intelligence and Statistics}, pages 249--258. PMLR.

\bibitem[Nocedal and Wright, 1999]{nocedal}
Nocedal, J. and Wright, S.~J. (1999).
\newblock {\em Numerical optimization}.
\newblock Springer-Verlag, New York, NY, 2 edition.

\bibitem[Pilanci and Wainwright, 2016]{pilanci2016iterative}
Pilanci, M. and Wainwright, M.~J. (2016).
\newblock Iterative hessian sketch: Fast and accurate solution approximation for constrained least-squares.
\newblock {\em The Journal of Machine Learning Research}, 17(1):1842--1879.

\bibitem[Pilanci and Wainwright, 2017]{pilanci2017newton}
Pilanci, M. and Wainwright, M.~J. (2017).
\newblock Newton sketch: A near linear-time optimization algorithm with linear-quadratic convergence.
\newblock {\em SIAM Journal on Optimization}, 27(1):205--245.

\bibitem[Rodomanov and Kropotov, 2016]{rodomanov2016superlinearly}
Rodomanov, A. and Kropotov, D. (2016).
\newblock A superlinearly-convergent proximal newton-type method for the optimization of finite sums.
\newblock In {\em International Conference on Machine Learning}, pages 2597--2605. PMLR.

\bibitem[Rodomanov and Nesterov, 2021a]{rodomanov2021greedy}
Rodomanov, A. and Nesterov, Y. (2021a).
\newblock Greedy quasi-newton methods with explicit superlinear convergence.
\newblock {\em SIAM Journal on Optimization}, 31(1):785--811.

\bibitem[Rodomanov and Nesterov, 2021b]{rodomanov2020new}
Rodomanov, A. and Nesterov, Y. (2021b).
\newblock New results on superlinear convergence of classical quasi-newton methods.
\newblock {\em Journal of optimization theory and applications}, 188:744--769.

\bibitem[Rodomanov and Nesterov, 2021c]{rodomanov2021rates}
Rodomanov, A. and Nesterov, Y. (2021c).
\newblock Rates of superlinear convergence for classical quasi-newton methods.
\newblock {\em Mathematical Programming}, pages 1--32.

\bibitem[Song and Ermon, 2020]{song}
Song, Y. and Ermon, S. (2020).
\newblock Improved techniques for training score-based generative models.
\newblock {\em ArXiv}, abs/2006.09011.

\bibitem[Wu et~al., 2018]{wu2018error}
Wu, J., Huang, W., Huang, J., and Zhang, T. (2018).
\newblock Error compensated quantized sgd and its applications to large-scale distributed optimization.
\newblock In {\em International Conference on Machine Learning}, pages 5325--5333. PMLR.

\end{thebibliography}
